\documentclass[a4paper,oneside,notitlepage,10pt]{article}%
\usepackage{amsmath}
\usepackage{amssymb}
\usepackage{amscd}
\usepackage{amsfonts}
\usepackage{amstext}
\usepackage{graphicx}

\providecommand{\U}[1]{\protect\rule{.1in}{.1in}}
\newtheorem{theorem}{Theorem}

\newtheorem{remark}[theorem]{Remark}

\newtheorem{lemma}[theorem]{Lemma}

\newtheorem{algorithm}[theorem]{Algorithm}

\newtheorem{corollary}[theorem]{Corollary}

\newtheorem{definition}[theorem]{Definition}

\newtheorem{example}[theorem]{Example}

\newtheorem{proposition}[theorem]{Proposition}

\newenvironment{proof}[1][Proof]{\textbf{#1.} }{\ \rule{0.5em}{0.5em}}

\renewcommand\appendix{\par
  \setcounter{section}{0}
  \setcounter{subsection}{0}
  \setcounter{figure}{0}
  \setcounter{table}{0}
  \renewcommand\thesection{Appendix \Alph{section}}
  \renewcommand\thefigure{\Alph{section}\arabic{figure}}
  \renewcommand\thetable{\Alph{section}\arabic{table}}
}

\begin{document}

\title{Simulations of multivariate gamma distributions and multifactor gamma distributions}

\author{Philippe Bernardoff
\thanks{Corresponding author: \texttt{philippe.bernardoff@univ-pau.fr}}
\and B\'en\'edicte Puig
\thanks{\texttt{benedicte.puig@univ-pau.fr}}
\\Universit\'{e} de Pau et des Pays de l'Adour
\\Laboratoire de Math\'{e}matiques et de leurs Applications, UMR 5142
\\avenue de l'Universit\'{e}, 64000 Pau, France
}

\maketitle

\begin{abstract}
This article provides a general expression for infinitely divisible
multivariate gamma distributions defined by their Laplace transforms, as well as the conditional Laplace transform of infinitely divisible multivariate gamma distributions. We give algorithms for simulating infinitely divisible gamma distributions
and infinitely divisible multifactor gamma distributions in dimension $2,3,4$ and for all dimensions greater than $2$ in the Markovian case. We give examples of simulations in dimension $2,3,4$ and in dimension $5$ in the Markovian case.\newline

\rule[0.02cm]{10cm}{0.02cm}



\end{abstract}

\section{Introduction}

The aim of this paper is to extend simulations of bivariate gamma
distributions, see \cite{Bernardoff(2023)}, to multivariate gamma
distributions and multifactor gamma distributions defined by their Laplace transforms. In this paper, we consider the following definitions given in
\cite{Bernardoff(2018)}. For more details see also \cite{Bar-Lev(1994)},
\cite{JKB(2000)} and \cite{Bernardoff(2006)}. We use the extension of the classical univariate definition to $\mathbb{R}^{n}$ obtained as follows: we consider an affine polynomial $P_{n}\left(
\boldsymbol{\theta}\right)  $ in $\boldsymbol{\theta}=\left(  \theta
_{1},\ldots,\theta_{n}\right)  $ where `affine' means that, for $j=1,\ldots
,n,$ $\partial^{2}P_{n}/\partial\theta_{j}^{2}=0$. We also assume that
$P_{n}\left(  \mathbf{0}\right)  =1.$ For instance, for $n=2,$ we have
$P_{2}\left(  \theta_{1},\theta_{2}\right)  =1+p_{\left\{  1\right\}  }%
\theta_{1}+p_{\left\{  2\right\}  }\theta_{2}+p_{\left\{  1,2\right\}  }%
\theta_{1}\theta_{2}$. We denote by $\mathfrak{P}_{n}=\mathfrak{P}\left(  \left[  n\right]  \right)
$ the family of all subsets of $\left[  n\right]  =\left\{  1,\ldots
,n\right\}  $ and $\mathfrak{P}_{n}^{\ast}$ the family of non-empty subsets of
$\left[  n\right]  .$ For simplicity, if $n$ is fixed and if there is no
ambiguity, we denote these families by $\mathfrak{P}$ and $\mathfrak{P}^{\ast
}$, respectively. Similarly, we denote by $\mathfrak{P}_{T}=\mathfrak{P}%
\left(  T\right)  $ the family of all subsets of $T=\left\{  1,\ldots
,n\right\}  $ and $\mathfrak{P}_{T}^{\ast}$ the family of non-empty subsets of
$T$. Similarly, if there is no ambiguity, we denote $P_{n}$ by $P$. We denote by $\mathbb{N}$ the set of non-negative integers. If $\boldsymbol{z}%
=\left(  z_{1},\ldots,z_{n}\right)  \in\mathbb{R}^{n}$ and $\boldsymbol{\alpha
}=\left(  \alpha_{1},\ldots,\alpha_{n}\right)  \in\mathbb{N}^{n},$ then
$\boldsymbol{\alpha}!=\alpha_{1}!\ldots\alpha_{n}!$, $\left\vert
\boldsymbol{\alpha}\right\vert =\alpha_{1}+\ldots+\alpha_{n}$,
$a_{\boldsymbol{\alpha}}=a_{\alpha_{1},\ldots,\alpha_{n}}$ and
\begin{equation}
\boldsymbol{z}^{\boldsymbol{\alpha}}=\prod\limits_{i=1}^{n}z_{i}^{\alpha_{i}%
}=z_{1}^{\alpha_{1}}\ldots z_{n}^{\alpha_{n}}. \label{zalpha2}%
\end{equation}
For $T\ $in $\mathfrak{P}_{n},$ we simplify the above notation by writing
$\boldsymbol{z}^{T}=\prod\nolimits_{t\in T}z_{t}$ instead of $\boldsymbol{z}%
^{\mathbf{1}_{T}}$ where
\begin{equation}
\mathbf{1}_{T}=\left(  \alpha_{1},\ldots,\alpha_{n}\right)  \text{ with
}\alpha_{i}=1\text{ if }i\in T\text{ and }\alpha_{i}=0\text{ if }i\notin T.
\label{1_T}%
\end{equation}
We also write $\boldsymbol{z}^{-T}$ for $\prod\nolimits_{t\in T}1/z_{t}$ if
$z_{t}\neq0$, $\forall t\in T.$ For a mapping $a:\mathfrak{P}\rightarrow
\mathbb{R}$, we shall use the notation $a:\mathfrak{P}\rightarrow\mathbb{R}$,
$T\mapsto a_{T}.$ In this notation, an affine polynomial with constant term
equal to $1$ is $P\left(  \boldsymbol{\theta}\right)  =\sum_{T\in\mathfrak{P}%
}p_{T}\boldsymbol{\theta}^{T},$ with $p_{\varnothing}=1.$ Now, if there is no
ambiguity, for simplicity, we omit the braces and, if $T=\left\{  t_{1}%
,\ldots,t_{k}\right\}  ,$ we denote $a_{\left\{  t_{1},\ldots,t_{k}\right\}
}=a_{t_{1},\ldots,t_{k}}$ and $a_{\emptyset}=a_{0}.$ The indicator function of a set $S$ is denoted by $\mathbf{1}_{S},$ that is, $\mathbf{1}_{S}\left(  x\right)  =1$ for $x\in S$ and $0$ for $x\notin S$. We fix $\lambda>0.$ If a random vector $\mathbf{X=}\left(  X_{1},\ldots
,X_{n}\right)  $ on $\mathbb{R}^{n}$ with probability distribution (pd)
$\mu_{\boldsymbol{X}}$ is such that its Laplace transform (Lt) is%
\begin{equation}
\mathbb{E}\left\{  \exp\left[  -\left(  \theta_{1}X_{1}+\cdots+\theta_{n}%
X_{n}\right)  \right]  \right\}  =\left[  P\left(  \boldsymbol{\theta}\right)
\right]  ^{-\lambda}, \label{TLMGD}%
\end{equation}
where $\mathbb{E}$ denotes the expectation, for a set of $\boldsymbol{\theta}$
with non-empty interior, then we denote $\mu_{\boldsymbol{X}}=$%
$\boldsymbol{\gamma}$$_{\left(  P,\lambda\right)  },$ and $\boldsymbol{\gamma
}$$_{\left(  P,\lambda\right)  }$ will be called the \textit{multivariate
gamma distribution }(\textit{mgd) associated with} $\left(  P,\lambda\right)
.$ If $\mathbf{X}$ has pd $\boldsymbol{\gamma}$$_{\left(  P,\lambda\right)
},$ we denote it by $\mathbf{X}\sim\boldsymbol{\gamma}$$_{\left(
P,\lambda\right)  },$ and $P,\lambda$ is called respectively the scale
parameter, the shape parameter. These mgds occur naturally in the
classification of natural exponential families in $\mathbb{R}^{n}$
\cite{Bar-Lev(1994)}. The marginal distributions of the \textit{mgd associated with} $\left(
P,\lambda\right)  $ are univariate
gamma distributions (ugd) of parameters $\left(  p_{i},\lambda\right)  $
for $i=1,\ldots,n,$ with Lt%
\begin{equation}
\left[  P\left(  0,\ldots,0,\theta_{i},0,\ldots,0\right)  \right]  ^{-\lambda
}=\left(  1+p_{i}\theta_{i}\right)  ^{-\lambda}, \label{marginalMGD}%
\end{equation}
and pd
\begin{equation}
\gamma_{\left(  p_{i},\lambda\right)  }\left(  \text{\emph{d}}x\right)
=x^{\lambda-1}p_{i}^{-\lambda}[\Gamma\left(  \lambda\right)  ]^{-1}\exp\left(
-x/p_{i}\right)  \mathbf{1}_{\left(  0,\infty\right)  }\left(  x\right)
\emph{d}x. \label{gamma_dim_1}%
\end{equation}
As in \cite{Bernardoff(2018)}, we extend the first definition to
the\textit{\ multifactor gamma distribution (mfgd) associated with} $\left(
P,\boldsymbol{\Lambda}\right)  $ where $\boldsymbol{\Lambda}=\left(
\lambda,\lambda_{1},\ldots,\lambda_{n}\right)  $ and $\lambda_{i}%
\geqslant\lambda>0$ for all $i=1,\ldots,n$ by its Lt%
\begin{equation}
\mathbb{E}\left\{  \exp\left[  -\left(  \theta_{1}X_{1}+\cdots+\theta_{n}%
X_{n}\right)  \right]  \right\}  =\left[  P\left(  \boldsymbol{\theta}\right)
\right]  ^{-\lambda}\prod_{i=1}^{n}\left(  1+p_{i}\theta_{i}\right)
^{-\left(  \lambda_{i}-\lambda\right)  }. \label{TLMMGD}%
\end{equation}
Using (\ref{marginalMGD}), the marginal distributions of the \textit{mgd
associated with} $\left(  P,\boldsymbol{\Lambda}\right)  $ are  ugds of parameters $\left(  p_{i},\lambda_{i}\right)  $
for $i=1,\ldots,n,$ with Lt $\left(  1+p_{i}\theta_{i}\right)  ^{-\lambda_{i}%
}$, and pd \newline $\gamma_{\left(  p_{i},\lambda_{i}\right)  }\left(  \text{\emph{d}%
}x\right)  =x^{\lambda_{i}-1}p_{i}^{-\lambda_{i}}[\Gamma\left(  \lambda
_{i}\right)  ]^{-1}\exp\left(  -x/p_{i}\right)  \mathbf{1}_{\left(
0,\infty\right)  }\left(  x\right)  $\emph{d}$x.$ We can denote either
$\gamma_{\left(  p_{i},\lambda_{i}\right)  }$ or $\gamma_{\left(
1+p_{i}\theta_{i},\lambda_{i}\right)  }.$

\cite{Bernardoff(2018)} gives a proposition for building a random vector whose
distribution is the mfgd associated with $\left(  P,\boldsymbol{\Lambda
}\right)  $.

\begin{proposition}
\label{buildMFGD}A random vector $\boldsymbol{X}$ with distribution
$\boldsymbol{\gamma}$$_{P,\boldsymbol{\Lambda}}$ can be obtained in the
following way: \newline Let $\boldsymbol{Y}$ be a random vector with
distribution $\boldsymbol{\gamma}$$_{\left(  P,\lambda\right)  }.$ Let
$\boldsymbol{Z}=\left(  Z_{1},\dots,Z_{n}\right)  $ be a random vector of
independent components for which its pds are $\gamma_{\left(  p_{i}%
,\lambda_{i}-\lambda\right)  },$ and such that $\boldsymbol{Z}$ and
$\boldsymbol{Y}$ are independent random vectors. Then the random vector
$\boldsymbol{X}=\boldsymbol{Y}+\boldsymbol{Z}$ has Lt (\ref{TLMMGD}), and
consequently has the\textit{\ mfgd }associated with $\left(
P,\boldsymbol{\Lambda}\right)  .$
\end{proposition}

\begin{remark}
This construction clearly allows us to simulate $\mathbf{\gamma}_{\left(
P,\mathbf{\Lambda}\right)  }$ by simulating $\boldsymbol{Z}\sim\mathbf{\gamma
}_{\left(  P,\lambda\right)  }$ and $\boldsymbol{Y}\mathbf{.}$
\end{remark}

For the bivariate case, see also \cite{Dussauchoy(1972)}. The Lt of mgd associated with $\left(  P,\lambda\right)  $ and the Lt of mfgd
associated with $\left(  P,\boldsymbol{\Lambda}\right)  $ are known by
definition. But, its pdfs are unknown, except for the bivariate gamma
distribution (bgd) associated with $\left(  P,\lambda\right)  $ and the
bivariate mfgd associated with $\left(  P,\left(  \lambda,\lambda_{1}%
,\lambda_{2}\right)  \right)  $. Let us recall in the Proposition below, the
known results. Let $F_{m}^{p}$ be the generalized hypergeometric function (see
\cite{Slater(1966)}) defined by%

\begin{equation}
F_{m}^{p}\left(  \alpha_{1},\ldots,\alpha_{p};\beta_{1},\ldots,\beta
_{m};z\right)  =\sum_{k=0}^{\infty}\frac{\left(  \alpha_{1}\right)  _{k}%
\cdots\left(  \alpha_{p}\right)  _{k}}{\left(  \beta_{1}\right)  _{k}%
\cdots\left(  \beta_{m}\right)  _{k}}\frac{z^{k}}{k!}, \label{Hypergeomconf}%
\end{equation}
where $\left(  a\right)  _{k}=\Gamma\left(  a+k\right)  /\Gamma\left(
a\right)  $ for $a>0$ and $k\in\mathbb{N}$ (or more generally by $\forall
n\in\mathbb{N}$, $\forall a\in\mathbb{R}$, $(a)_{0}=1$, $(a)_{n+1}%
=(a+n)(a)_{n}$) is the Pochhammer's symbol. For simplification, we denote
$F_{m}^{0}$ by $F_{m}$. \cite{Bernardoff(2006)} gives the following proposition.

\begin{proposition}
Let $P\left(  \theta_{1},\theta_{2}\right)  =1+p_{1}\theta_{1}+p_{2}\theta
_{2}+p_{12}\theta_{1}\theta_{2}$ be an affine polynomial where $p_{1},p_{2}>0$
and $p_{1,2}>0.$ Let $\mu=$$\boldsymbol{\gamma}$$_{\left(  P,\lambda\right)
}$ be the bgd associated with $\left(  P,\lambda\right)  .$ The measure $\mu$
exists if and only if $c=\left(  p_{1}p_{2}-p_{1,2}\right)  /p_{1,2}^{2}%
=p_{1}p_{2}/p_{1,2}^{2}\rho_{1,2}>0,$ where $\rho_{1,2}=1-p_{1,2}/(p_{1}%
p_{2})$ is the correlation coefficient between margins. Then we have%
\begin{equation}
\boldsymbol{\gamma}_{\left(  P,\lambda\right)  }\left(  \text{\emph{d}}%
x_{1},\text{\emph{d}}x_{2}\right)  =\frac{p_{1,2}^{-\lambda}}{\Gamma\left(
\lambda\right)  ^{2}}\mathbf{e}^{-\frac{p_{2}}{p_{12}}x_{1}-\frac{p_{1}%
}{p_{12}}x_{2}}\left(  x_{1}x_{2}\right)  ^{\lambda-1}F_{1}\left(
\lambda;cx_{1}x_{2}\right)  \mathbf{1}_{\left(  0,\infty\right)  ^{2}}\left(
\boldsymbol{x}\right)  \,\text{\emph{d}}\boldsymbol{x}. \label{mud2}%
\end{equation}
with 
$
F_{1}\left(  \lambda;z\right)  =\sum_{k=0}^{\infty}\frac{1}{\left(
\lambda\right)  _{k}}\frac{z^{k}}{k!}=\Gamma\left(  \lambda\right)
z^{-(\lambda-1)/2}I_{\lambda-1}\left(  2\sqrt{z}\right)  ,
$
where $I_{\lambda}$is the modified Bessel function of order $\lambda.$
\end{proposition}

For the case $\boldsymbol{\Lambda}=\left(  \lambda,\lambda,\lambda_{2}\right)
$, the \textit{mfgd associated with} $\left(  P,\boldsymbol{\Lambda}\right)  $
is named by \cite{Chatelainetal2008} the\textit{\ multisensor gamma
distribution associated with} $\left(  P,\lambda,\lambda_{2}\right)  $ and
they have proved that its pd is given by the equality 
\begin{align}
\boldsymbol{\gamma}_{\left(  P,\boldsymbol{\Lambda}\right)  }\left(
\text{d}x_{1},\text{d}x_{2}\right)   &  =\frac{p_{12}^{-\lambda}%
p_{2}^{-\left(  \lambda_{2}-\lambda\right)  }1}{\Gamma\left(  \lambda\right)
\Gamma\left(  \lambda_{2}\right)  }x_{1}^{\lambda-1}x_{2}^{\lambda_{2}%
-1}\mathbf{e}^{-\frac{p_{2}}{p_{12}}x_{1}-\frac{p_{1}}{p_{12}}x_{2}}\Phi
_{3}(\lambda_{2}-\lambda;\lambda_{2};c\frac{p_{12}}{p_{2}}x_{2};cx_{1}%
x_{2})\nonumber\\
&  \times\mathbf{1}_{\left(  0,\infty\right)  ^{2}}\left(  x_{1},x_{2}\right)
\,\text{d}x_{1}\text{d}x_{2}\text{,} \label{BMGD}%
\end{align}
where 
$
\Phi_{3}\left(  a;b;x,y\right)  =\sum_{m,n\geqslant0}\frac{\left(  a\right)
_{m}}{\left(  b\right)  _{m+n}}\frac{x^{m}}{m!}\frac{y^{n}}{n!}
$ 
is the Horn function. For the bivariate general case, we have the following Theorem in
\cite{Bernardoff(2018)}. Let $F_{I}$ be the function defined by%
\begin{equation}
F_{I}\left(  a,b,c,z_{1},z_{2},z_{3}\right)  =\sum_{m_{1},m_{2},m_{3}=0}^{\infty
}\frac{\left(  a\right)  _{m_{1}}\left(  b\right)  _{m_{2}}\left(  c\right)
_{m_{3}}}{\left(  a+c\right)  _{m_{1}+m_{3}}\left(  b+c\right)  _{m_{2}+m_{3}%
}}\frac{z_{1}^{m_{1}}}{m_{1}!}\frac{z_{2}^{m_{2}}}{m_{2}!}\frac{z_{3}^{m_{3}}%
}{m_{3}!}; \label{genLauricella1}%
\end{equation}
it is a particular generalized Lauricella function defined, by example, in
\cite{Panda1973}.

\begin{theorem}
\label{Proposition_multifactorbig2}The pd of $\boldsymbol{\gamma}$$_{\left(
P,\left(  \lambda,\lambda_{1},\lambda_{2}\right)  \right)  }$ is given by the
equality%
\begin{align}
\boldsymbol{\gamma}_{\left(  P,\left(  \lambda,\lambda_{1},\lambda_{2}\right)
\right)  }\left(  \text{\emph{d}}x_{1},\text{\emph{d}}x_{2}\right)   &
=\frac{p_{12}^{-\lambda}p_{1}^{-\left(  \lambda_{1}-\lambda\right)  }%
p_{2}^{-\left(  \lambda_{2}-\lambda\right)  }}{\Gamma\left(  \lambda
_{1}\right)  \Gamma\left(  \lambda_{2}\right)  }x_{1}^{\lambda_{1}-1}%
x_{2}^{\lambda_{2}-1}\mathbf{e}^{-\frac{p_{2}}{p_{12}}x_{1}-\frac{p_{1}%
}{p_{12}}x_{2}}\times\nonumber\\
&  \hspace*{-3cm} F_{I}\left(  \lambda_{1}-\lambda,\lambda_{2}%
-\lambda,\lambda,\tfrac{p_{12}}{p_{1}}x_{1},\tfrac{p_{12}}{p_{2}}x_{2}%
,cx_{1}x_{2}\right)  \mathbf{1}_{\left(  0,\infty\right)  ^{2}}\left(
x_{1},x_{2}\right)  \,\text{\emph{d}}x_{1}\text{\emph{d}}x_{2}\text{,}
\label{mulbigamma2}%
\end{align}

\end{theorem}

If we get $\lambda_{1}=\lambda$ in the equality (\ref{mulbigamma2}), we obtain
Chatelain and Tourneret's result (\ref{BMGD}) because%
\begin{align*}
F_{I}\left(  0,\lambda_{2}-\lambda,\lambda,z_{1},z_{2},z_{3}\right)    =
\sum_{m_{2},m_{3}=0}^{\infty}\frac{\left(  b\right)  _{m_{2}}%
}{\left(  b+c\right)  _{m_{2}+m_{3}}}\frac{z_{2}^{m_{2}}}{m_{2}!}\frac
{z_{3}^{m_{3}}}{m_{3}!}  &  =\Phi_{3}\left(  b;b+c;z_{2},z_{3}\right)  .
\end{align*}

\cite{Bernardoff(2006)} gives the following Proposition:

\begin{proposition}
Let $\mu$ be a \textit{mgd }on $\mathbb{R}^{n}$ associated with $\left(
P,\lambda\right)  .$ Assume that $\mu$ is not concentrated on a linear
subspace of $\mathbb{R}^{n}$ of the form $\{\boldsymbol{x}\in\mathbb{R}%
^{n};\ x_{k}=0\}$ for some $k$ in $\left[  n\right]  =\{1,\ldots,n\}.$ Then:

\begin{enumerate}
\item[\emph{(i)}] For all $i\in\left[  n\right]  ,$ $p_{i}\neq0$.

\item[\emph{(ii)}] If $p_{1},\ldots,p_{k}<0$ and $p_{k+1},\ldots,p_{n}>0,$
then $\text{Supp}\left(  \mu\right)  \subset\left(  -\infty,0\right]
^{k}\times\left[  0,\infty\right)  ^{n-k}$.

\item[\emph{(iii)}] If $p_{1},\ldots,p_{n}>0$ then $p_{\left[  n\right]
}\geqslant0.$
\end{enumerate}
\end{proposition}

\cite{Bernardoff(2006)} gives a necessary and sufficient condition for
infinite divisibility of the \textit{mgd associated with} $\left(
P,\lambda\right)  $, in the sense that the Lt of $\boldsymbol{\gamma}%
$$_{\left(  P,\lambda\right)  }$ power $t$ for all positive $t$ is still the
Lt of a positive measure, by the following theorem using the notation
$\widetilde{b}_{S}$ from the notation $b_{S}$ defined in
\cite{Bernardoff(2003)}:

\begin{theorem}
\label{mainresult} Let $\mu=\boldsymbol{\gamma}_{P,\lambda}$ be a mgd
associated with $\left(  P,\lambda\right)  ,$ where $\lambda>0$ and $P\left(
\boldsymbol{\theta}\right)  =\sum\nolimits_{T\in\mathfrak{P}_{n}}%
p_{T}\boldsymbol{\theta}^{T}$ is such that $p_{i}>0$, for all $i\in\left[
n\right]  ,$ and $p_{\left[  n\right]  }>0$. Let $\widetilde{P}\left(
\boldsymbol{\theta}\right)  =\sum\nolimits_{T\in\mathfrak{P}_{n}}%
\widetilde{p}_{T}\boldsymbol{\theta}^{T}$ be the affine polynomial such that
$\widetilde{p}_{T}=-p_{\overline{T}}/p_{\left[  n\right]  }$ for all
$T\in\mathfrak{P}_{n},$ where $\overline{T}=\left[  n\right]  \smallsetminus
T$. Let%
\begin{equation}
\widetilde{b}_{S}=b_{S}(\widetilde{P})=\sum_{k=1}^{\left\vert S\right\vert
}\left(  k-1\right)  !\sum_{\mathcal{T\in}\Pi_{S}^{k}}\prod_{T\in\mathcal{T}%
}\widetilde{p}_{T}, \label{b_tildeS}%
\end{equation}
with%
\begin{equation}
b_{S}(P)=\sum_{k=1}^{\left\vert S\right\vert }\left(  k-1\right)
!\sum_{\mathcal{T\in}\Pi_{S}^{k}}\prod_{T\in\mathcal{T}}p_{T}, \label{b_S}%
\end{equation}
where $\left\vert S\right\vert $ is the cardinality of the set $S$. Then the
measure $\mu$ is infinitely divisible if and only if%
\begin{equation}
\widetilde{\text{$p$}}_{i}=\widetilde{b}_{\{i\}}<0\text{ for all }i\in\left[
n\right]  , \label{cond1}%
\end{equation}
and%
\begin{equation}
\widetilde{b}_{S}\geqslant0\text{ for all }S\in\mathfrak{P}_{n}^{\ast}\text{
such that }\left\vert S\right\vert \geqslant2. \label{cond2}%
\end{equation}

\end{theorem}

\begin{corollary}
By the properties of infinite divisible distributions we conclude that the
necessary and sufficient conditions for infinite divisibility of a \textit{mgd
}associated with $\left(  P,\lambda\right)  $ of theorem (\ref{mainresult}),
are also necessary and sufficient conditions for infinite divisibility
of\textit{ mfgd }associated with $\left(  P,\boldsymbol{\Lambda}\right)  .$
\end{corollary}

To illustrate the difficulty to calculate the \textit{mgd associated with}
$\left(  P,\lambda\right)  $ we recall, for the \textit{trivariate gamma
distribution (tgd)}, the following theorem given in \cite{Bernardoff(2018)}.
Let $F_{II}$ be the function defined by%
\begin{equation}
F_{II}\left(  \lambda_{1},\lambda_{2},z_{1},z_{2},z_{3},z_{4}\right)
=\sum_{m_{1},\ldots,m_{4}=0}^{\infty}\frac{1}{\left(  \lambda_{1}\right)
_{m_{1}+m_{2}+m_{3}}\left(  \lambda_{2}\right)  _{2m_{1}+m_{2}+m_{4}}}%
\frac{z_{1}^{m_{1}}}{m_{1}!}\frac{z_{2}^{m_{2}}}{m_{2}!}\frac{z_{3}^{m_{3}}%
}{m_{3}!}\frac{z_{4}^{m_{4}}}{m_{4}!}; \label{genLauricella2}%
\end{equation}
it is still a particular generalized Lauricella function. We note that $c$ in (\ref{mud2}) is $\widetilde{b}_{1,2}.$

\begin{theorem}
\label{gamma3}In the case $n=3$, $p_{i}>0$ for $i\in\left[  3\right]  ,$
$p_{ij}>0$ for $i,j\in\left[  3\right]  $, $\widetilde{b}_{ij}=-\frac{b_{k}%
}{p_{123}}+\frac{p_{jk}p_{ik}}{p_{123}^{2}}\geqslant0$ for $i\neq j$ and
$\left\{  i,j,k\right\}  =\left[  3\right]  ,$ $p_{123}>0,$ and
$
\widetilde{b}_{123}=-\frac{1}{p_{123}}+\frac{p_{12}p_{1}}{p_{123}^{2}}%
+\frac{p_{13}p_{2}}{p_{123}^{2}}+\frac{p_{23}p_{1}}{p_{123}^{2}}+2\frac
{p_{12}p_{13}p_{23}}{p_{123}^{3}}\geqslant0,
$
the infinitely divisible \textit{tgd }$\boldsymbol{\gamma}_{\left(
P,\lambda\right)  }$ \textit{associated with} $\left(  P,\lambda\right)  ,$ is
given by the formula%
\begin{align}
\boldsymbol{\gamma}_{\left(  P,\lambda\right)  }\left(  \text{\emph{d}%
}\boldsymbol{x}\right)   &  =\frac{p_{123}^{-\lambda}}{\left[  \Gamma\left(
\lambda\right)  \right]  ^{3}}\exp(\widetilde{p}_{1}x_{1}+\widetilde{p}%
_{2}x_{2}+\widetilde{p}_{3}x_{3})\left(  x_{1}x_{2}x_{3}\right)  ^{\lambda
-1} \times\nonumber\\
&  F_{II}(\lambda,\lambda,\widetilde{b}_{13}x_{1}x_{3}\widetilde{b}%
_{23}x_{2}x_{3},\widetilde{b}_{123}x_{1}x_{2}x_{3},\widetilde{b}_{12}%
x_{1}x_{2},\widetilde{b}_{13}x_{1}x_{3}+\widetilde{b}_{23}x_{2}x_{3}%
)\mathbf{1}_{\left(  0,\infty\right)  ^{3}}\left(  \boldsymbol{x}\right)
\,\text{\emph{d}}\boldsymbol{x}. \label{MGD3}%
\end{align}

\end{theorem}

\begin{remark}
The case $p_{123}=0$ is solved by \cite{LetacWesolowsky2008}.
\end{remark}

\begin{remark}
If $\widetilde{b}_{12}=\widetilde{b}_{13}=\widetilde{b}_{23}=0,$ Theorem
\ref{gamma3} gives
\begin{equation}
\boldsymbol{\gamma}_{\left(  P,\lambda\right)  }\left(  \text{\emph{d}%
}\boldsymbol{x}\right)  =\tfrac{p_{123}^{-\lambda}}{\left[  \Gamma\left(
\lambda\right)  \right]  ^{3}}\exp(\widetilde{p}_{1}x_{1}+\widetilde{p}%
_{2}x_{2}+\widetilde{p}_{3}x_{3})\left(  x_{1}x_{2}x_{3}\right)  ^{\lambda
-1}F_{2}(  \lambda,\lambda;\widetilde{b}_{123}x_{1}x_{2}x_{3})
\mathbf{1}_{\left(  0,\infty\right)  ^{3}}\left(  \boldsymbol{x}\right)
\,\text{\emph{d}}\boldsymbol{x}, \label{particular_n_3}%
\end{equation}
and if we put $\lambda=1$ in (\ref{particular_n_3}), we obtain the Kibble and
Moran distribution given in \cite{JKB(2000)}.
\end{remark}

After giving a general expression for the pd of infinitely divisible mgds,
this paper provides their conditional Laplace transforms. These results allow
us to extend the result of \cite{Bernardoff(2023)} from the simulation of
\textit{bgds} to \textit{mgds and mfgds.} These results allow us to achieve
the aim of this paper for $n=3,4$ in the general case and, for two particular
cases for any $n$. This paper is organised as follows. Let $n\in\mathbb{N}^{\ast}\smallsetminus
\left\{  1\right\}  $ and let $  \mathbf{X}_{n}  =\left(
X_{1},\ldots,X_{n}\right)  $ be a real random vector of infinitely divisible
mgd, Section 2 gives a general expression for the pd $\mu_{\mathbf{X}_{n}%
}=\gamma_{\left(  P_{n},\lambda\right)  }$. Let $k\in\mathbb{N}^{\ast
}\smallsetminus\left\{  1,n\right\}  $ and $\left(  x_{1},\ldots,x_{k}\right)
\in\mathbb{R}^{k}$, Section 3 gives the conditional Lt of $\left(
X_{k},\ldots,X_{n}\right)  $ given $\left(  X_{1},\ldots,X_{k}\right)
=\left(  x_{1},\ldots,x_{k}\right)  $ denoted by $L_{\left(  X_{k}%
,\ldots,X_{n}\right)  }^{\left(  X_{1},\ldots,X_{k}\right)  =\left(
x_{1},\ldots,x_{k}\right)  }$. An important particular case is given. Section
4 gives a simpler expression of $L_{\left(  X_{2},\ldots,X_{n}\right)
}^{X_{1}=x_{1}}$. Section 5 apply this last result for $n=2,3,4$, and for the
case $n=2$ , we obtain the result of \cite{Bernardoff(2023)}, see also
\cite{Walker(2021)}. For the case $n=3,4$, we obtain a new general result and algorithms for simulating tgds and quadrivariate gamma distributions. For $n=5$, we could give an expression for $L_{\left(
X_{2},\ldots,X_{n}\right)  }^{X_{1}=x_{1}}$ and apply the same method as for
cases $n=2,3,4$. Unfortunately, the computations seem long and arduous.
Therefore, we study the simpler Markovian case in Section 6. Section 7 presents simulations of mgds for $n=2,3,4$, simulations of mfgds for $n=2,3$, and a simulation of Mmgd for $n=5$. All simulations are performed using the R software, \cite{RCoreTeam}. In
order to facilitate the fluent presentation of the paper, proofs are collected
in   \ref{Annexe1} in order of appearance.

\section{Probability distributions of multivariate gamma distributions}

For a simple example of applying the main results, we will need the following
example in dimension $n$.

\begin{example}
\label{equic_n}Let $P_{n}$ be the affine polynomial defined in
\cite{Bernardoff(2006)} by%
\begin{equation}
P_{n}\left(  \boldsymbol{\theta}\right)  =\frac{-q}{p}+\frac{1}{p}\prod
_{i=1}^{n}\left(  1+p\theta_{i}\right)  \label{gamma_bernardoff2006}%
\end{equation}
where $0<p=1-q<1.$ Let $\mu=\boldsymbol{\gamma}_{P,1}=\varphi_{n,p,1}$ be the
infinitely divisible mgd associated to $\left(  P,1\right)  .$ Let
$\boldsymbol{\gamma}_{P,\lambda}=\varphi_{n,p,\lambda}$ the mgd associated to
$\left(  P,\lambda\right)  .$ For $\boldsymbol{x}%
=\left(  x_{1},\ldots,x_{n}\right)  \in \mathbb{R}^{n},$ we have
\begin{equation}
\boldsymbol{\gamma}_{\left(  P_{n},\lambda\right)  }\left(  \mathtt{d}%
\mathbf{x}\right)  =\tfrac{p^{-\left(  n-1\right)  \lambda}}{\left(
\Gamma\left(  \lambda\right)  \right)  ^{n}}\mathbf{e}^{-\frac{x_{1+\cdots+}x_{n}}{p}%
}\left( \mathbf{x}^{\left[  n\right]  }\right)  ^{\lambda-1}%
F_{n-1}\left(  \lambda,\ldots,\lambda;qp^{-n}\mathbf{x}^{\left[  n\right]
}\right)  \mathbf{1}_{\left(  0,\infty\right)  ^{n}}\left(  \mathbf{x}\right)
\mathtt{d}\mathbf{x} \label{Phy2}%
\end{equation}

\end{example}

We will give an expression of the pd $\boldsymbol{\gamma}_{\left(
P_{n},\lambda\right)  }\left(  \mathtt{d}\mathbf{x}\right)  $ in the general case. Let us denote \newline$\boldsymbol{\theta}_{\left[  n\right]  }\mathbf{=}\left(
\theta_{1},\ldots,\theta_{n}\right)  $ , $P_{n}\left(  \boldsymbol{\theta
}_{\left[  n\right]  }\right)  =\sum_{T\in\mathfrak{P}_{n}}p_{T}\left(
P_{n}\right)  \boldsymbol{\theta}_{\left[  n\right]  }^{T},$ for $1\leqslant
k<n,$ $\boldsymbol{\theta}_{\left[  k\right]  }\mathbf{=}\left(  \theta
_{1},\ldots,\theta_{k}\right)  ,$ $P_{[k]}\left(  \boldsymbol{\theta}_{\left[
k\right]  }\right)  =\sum_{T\in\mathfrak{P}_{k}}p_{T}\left(  P_{k}\right)
\boldsymbol{\theta}_{\left[  k\right]  }^{T}=P_{n}\left(  \left(  \theta
_{1},\ldots,\theta_{k},0,\ldots,0\right)  \right)  $, let us define for
$T\in\mathfrak{P}_{n},$ $\widetilde{p}_{T}\left(  P_{n}\right)  =-\frac
{p_{\left[  n\right]  \smallsetminus T}}{p_{\left\{  n\right\}  }},$ and for
$T\in\mathfrak{P}_{k}$, $\widetilde{p}_{T}\left(  P_{k}\right)  =-\frac
{p_{\left[  k\right]  \smallsetminus T}}{p_{\left\{  k\right\}  }}$, and
$\boldsymbol{\theta}_{P_{n}}=\left(  \widetilde{p}_{1}\left(  P_{n}\right)
,\ldots,\widetilde{p}_{n}\left(  P_{n}\right)  \right)  ,$ as well as
$\boldsymbol{\theta}_{P_{k}}=\left(  \widetilde{p}_{1}\left(  P_{k}\right)
,\ldots,\widetilde{p}_{k}\left(  P_{k}\right)  \right)  .$ If there is no
ambiguity, we denote $\boldsymbol{\theta}_{\left[  n\right]  }$ by 
$\boldsymbol{\theta},$ $\boldsymbol{\theta}_{P_{n}}$ by $\boldsymbol{\theta}_{P}$,
$\widetilde{p}_{T}\left(  P_{n}\right)  =\widetilde{p}_{T}$, and $P_{n}\left(
\boldsymbol{\theta}_{n}\right)$ by $P\left(  \boldsymbol{\theta}\right)  $. We also denote $\widetilde{p}_{i}\left(  P_{n}\right)  =\widetilde{p}_{i}\left(
P\right)  =\widetilde{p}_{i},$ $i\in\left[  n\right]  .$ We also denote by
$\widetilde{\mathbf{p}}$ the vector $\boldsymbol{\theta}_{P_{n}}=\left(
\widetilde{p}_{1},\ldots,p_{n}\right)  .$ If $n>1,$ let $\boldsymbol{\theta
}_{P_{n}}=\left(  \widetilde{p}_{1}\left(  P_{n}\right)  ,\ldots
,\widetilde{p}_{n}\left(  P_{n}\right)  \right)  $, so that $\left(
\partial / \partial\theta\right)  ^{\overline{\left\{  i\right\}  }%
}\left(  P_{n}\right)  \left(  \boldsymbol{\theta}_{P_{n}}\right)  =0,\forall
i\in\left[  n\right]  $, since $P_{n}$ is an affine polynomial, using the
Taylor formula in $\boldsymbol{\theta}_{P_{n}},$ we get the following
proposition which define the affine polynomial $R_{n}.$

\begin{proposition}
\label{PropPnRn}With the above definition, we have
\begin{equation}
P_{n}\left(  \boldsymbol{\theta}_{n}\right)  =p_{\left[  n\right]  }\left(
\boldsymbol{\theta}_{\left[  n\right]  }-\boldsymbol{\theta}_{P_{n}}\right)
^{\left[  n\right]  }\left\{  1-R_{n}\left[  \left(  \boldsymbol{\theta
}_{\left[  n\right]  }-\boldsymbol{\theta}_{P_{n}}\right)  ^{-1}\right]
\right\}  . \label{P_nTaylor}%
\end{equation}
with
\begin{equation}
R_{n}\left(  \mathbf{z}_{\left[  n\right]  }\right)  =R_{n}\left(
z_{1},\ldots,z_{n}\right)  =\sum_{T\in\mathfrak{P}_{n},\left\vert T\right\vert
\geqslant2}r_{T}\mathbf{z}_{\left[  n\right]  }^{T}, \label{Rn}%
\end{equation}
and
\begin{equation}
r_{T}=\tfrac{-1}{p_{\left[  n\right]  }}\left(  \tfrac{\partial}{\partial\theta
}\right)  ^{\overline{T}}\left(  P_{n}\right)  \left(  \boldsymbol{\theta
}_{P_{n}}\right)  ,T\in\mathfrak{P}_{n},\left\vert T\right\vert \geqslant2
\label{rT}%
\end{equation}
Since $R_{n}$ depends of $P_{n},$ if necessary we will write $R_{n}\left(
P_{n}\right)  .$ If there is no ambiguity, denoting $R_{n}\left(
\mathbf{z}_{\left[  n\right]  }\right)  $ by $R\left(  \mathbf{z}\right)  $,
we have
\begin{equation}
P\left(  \boldsymbol{\theta}\right)  =p_{\left[  n\right]  }\left(
\boldsymbol{\theta}-\boldsymbol{\theta}_{P}\right)  ^{\left[  n\right]
}\left\{  1-R\left[  \left(  \boldsymbol{\theta}-\boldsymbol{\theta}%
_{P}\right)  ^{-\mathbf{1}}\right]  \right\}  \label{P_Taylor}%
\end{equation}
This last equality is still true for $n=1,$with $R\left(  z_{1}\right)
=0.$\newline More specifically, let $T\in\mathfrak{P}_{n},\left\vert
T\right\vert \geqslant2,$ we have
\begin{equation}
r_{T}=\sum_{T^{\prime}\in\mathfrak{P}_{T}}\widetilde{p}_{T\smallsetminus
T^{\prime}}\widetilde{\mathbf{p}}^{T^{\prime}}, \label{rT1}%
\end{equation}
and for any $n\in\mathbb{N}\smallsetminus\left\{  0,1\right\}  $ if
$\left\vert T\right\vert =2\leqslant n,$%
\begin{equation}
r_{T}=\widetilde{b}_{T}, \label{rT2}%
\end{equation}
if $\left\vert T\right\vert =3\leqslant n,$%
\begin{equation}
r_{T}=\widetilde{b}_{T}, \label{rT3}%
\end{equation}
if $\left\vert T\right\vert =4\leqslant n,$%
\begin{equation}
r_{T}=\widetilde{b}_{T}-\sum_{\left\{  U,V\right\}  \in\Pi_{T}^{2},\left\vert
U\right\vert =2,\left\vert V\right\vert =2}\widetilde{b}_{U}\widetilde{b}_{V},
\label{rT4}%
\end{equation}
if $\left\vert T\right\vert =5\leqslant n,$%
\begin{equation}
r_{T}=\widetilde{b}_{T}-\sum_{\left\{  U,V\right\}  \in\Pi_{T}^{2},\left\vert
U\right\vert =3,\left\vert V\right\vert =2}\widetilde{b}_{U}\widetilde{b}_{V},
\label{rT5}%
\end{equation}
Later, we will need the following definition and results. In the sequel, we
suppose that $\forall T\in\mathfrak{P}_{n},$ $p_{T}\neq0.$
\end{proposition}

\begin{definition}
Let $T\in\mathfrak{P}_{n},$ and $\overline{T}=\mathfrak{[}n]\smallsetminus T,$
and if there is no ambiguity, for simplicity we denote $S_{T}%
(\boldsymbol{\theta}_{T})$ by $S_{T},$ the polynomial defined by: if
$p_{\overline{T}}\neq0,$ then
\begin{equation}
S_{T}(\boldsymbol{\theta}_{T})=\tfrac{1}{p_{\overline{T}}}(\tfrac{\partial
}{\partial\boldsymbol{\theta}})^{\overline{T}}[P_{n}[\boldsymbol{\theta)}].
\label{STdef}%
\end{equation}
If $q_{U}$ is the numbers such that $S_{T}\left(  \boldsymbol{\theta}%
_{T}\right)  =\sum_{U\in\mathfrak{P}_{T}}q_{U}\boldsymbol{\theta}^{U}$, we
have
\begin{equation}
q_{U}=p_{\overline{T}\cup U}p_{\overline{T}}^{-1}. \label{qU}%
\end{equation}
Since $S_{T}$ depends on $P_{n}$, if necessary we will write $S_{T}\left(
P_{n}\right)  $.
\end{definition}

For example, if $n=3,$ we have $S_{2,3}\left(  \theta_{2},\theta_{3}\right)
=1+\frac{p_{1,2}}{p_{1}}\theta_{2}+\frac{p_{1,3}}{p_{1}}\theta_{3}%
+\frac{p_{1,2,3}}{p_{1}}\theta_{2}\theta_{3},$ $S_{i}\left(  \theta
_{i}\right)  =1+\frac{p_{1,2,3}}{p_{\left[  3\right]  \smallsetminus\left\{
i\right\}  }}\theta_{i},$ $i=1,2,3.$ We will need the following proposition.

\begin{proposition}
\label{PropSU(ST)}For $U\in\mathfrak{P}_{T}$, we have $,$
\begin{equation}
S_{U}\left(  S_{T}\right)  =S_{U}\left(  P_{n}\right)  . \label{SU(ST)=SU(Pn)}%
\end{equation}

\end{proposition}

We also have the following proposition.

\begin{proposition}
\label{PropqtildeUbtildeUrU(ST)}With the above definition, we have $\forall
U\in\mathfrak{P}_{T}$
\begin{equation}
\widetilde{q}_{U}=\widetilde{p}_{U}, \label{q_tilde=p_tilde}%
\end{equation}
therefore, we have
\begin{equation}
\widetilde{b}_{U}\left(  S_{T}\right)  =\widetilde{b}_{U}, \label{b_U(ST)=b_U}%
\end{equation}
and
\begin{equation}
r_{U}(S_{T})=r_{U}(P_{n}). \label{rU(ST)=rU(Pn)}%
\end{equation}
If $\boldsymbol{\gamma}_{\left(  P,\lambda\right)  }$ is an infinitely
divisible mgd, $\boldsymbol{\gamma}_{\left(  S_{T},\lambda\right)  }$ is also
an infinitely divisible mgd from Theorem (\ref{mainresult}). We have the
following equalities
\begin{equation}
S_{T}=(-\widetilde{p}_{T})^{-1}[\left(  -\widetilde{\mathbf{p}}\right)
^{T}\mathbf{S}^{T}-\sum_{T^{\prime}\in\mathfrak{P}_{T},\left\vert T^{\prime
}\right\vert >1}r_{T^{\prime}}(-\widetilde{\mathbf{p}})^{T\smallsetminus
T^{\prime}}\mathbf{S}^{T\smallsetminus T^{\prime}}] \label{ST}%
\end{equation}
or%
\begin{equation}
(-\widetilde{\mathbf{p}})^{T}\mathbf{S}^{T}=(-\widetilde{p}_{T})S_{T}%
+\sum_{T^{\prime}\in\mathfrak{P}_{T},\left\vert T^{\prime}\right\vert
>1}r_{T^{\prime}}(-\widetilde{\mathbf{p}})^{T\smallsetminus T^{\prime}%
}\mathbf{S}^{T\smallsetminus T^{\prime}}. \label{S^T}%
\end{equation}

\end{proposition}

Now, we can give the following expression of $\boldsymbol{\gamma}
_{(P_{n},\lambda)}$ by the following theorem.

\begin{theorem}
\label{Th_gamma_P_lambda}Let $c_{\boldsymbol{\alpha},\lambda}\left(  R\right)
$ such that
\begin{equation}
\left[  1-R\left(  \mathbf{z}\right)  \right]  ^{-\lambda}=\sum
_{\boldsymbol{\alpha}\in\mathbb{N}^{n}}c_{\boldsymbol{\alpha},\lambda}\left(
R\right)  \mathbf{z}^{\boldsymbol{\alpha}}, \label{c_alpha_lambda_R}%
\end{equation}
then
\begin{equation}
\boldsymbol{\gamma}_{\left(  P,\lambda\right)  }\left(  \mathtt{d}%
\mathbf{x}\right)  =\frac{p_{\left[  n\right]  }^{-\lambda}}{\left[
\Gamma\left(  \lambda\right)  \right]  ^{n}}\exp\left(  \boldsymbol{\theta
}_{P},\mathbf{x}\right)  \mathbf{x}^{\left(  \lambda-1\right)  \mathbf{1}_{n}%
}[\sum_{\boldsymbol{\alpha}\in\mathbb{N}^{n}}\frac{c_{\boldsymbol{\alpha
},\lambda}\left(  R\right)  }{\left(  \lambda\right)  _{\boldsymbol{\alpha}}%
}\mathbf{x}^{\boldsymbol{\alpha}}]\mathbf{1}_{\left(  0,\infty\right)  ^{n}%
}\left(  \mathbf{x}\right)  \left(  \mathtt{d}\mathbf{x}\right)  .
\label{gamma_P_lambda}%
\end{equation}
or more specifically%
\begin{equation}
\boldsymbol{\gamma}_{\left(  P,\lambda\right)  }\left(  \mathtt{d}%
\mathbf{x}\right)  =\frac{p_{\left[  n\right]  }^{-\lambda}}{\left[
\Gamma\left(  \lambda\right)  \right]  ^{n}}\exp\left(  \boldsymbol{\theta
}_{P},\mathbf{x}\right)  \mathbf{x}^{\left(  \lambda-1\right)  \mathbf{1}_{n}%
}[\sum_{\boldsymbol{\alpha}\in\mathbb{N}^{n},c_{\boldsymbol{\alpha},\lambda
}\left(  R\right)  \neq0}\frac{c_{\boldsymbol{\alpha},\lambda}\left(
R\right)  }{\left(  \lambda\right)  _{\boldsymbol{\alpha}}}\mathbf{x}%
^{\boldsymbol{\alpha}}]\mathbf{1}_{\left(  0,\infty\right)  ^{n}}\left(
\mathbf{x}\right)  \left(  \mathtt{d}\mathbf{x}\right)  .
\label{gamma_P_lambda2}%
\end{equation}

\end{theorem}

We deduce a result given in \cite{Bernardoff(2006)} in the form of the
following corollary.

\begin{corollary}
\label{cor_cor}For $P_{n}\left(  \boldsymbol{\theta}_{\left[  n\right]
}\right)  =\frac{-q}{p}+\frac{1}{p}\prod_{i=1}^{n}\left(  1+p\theta
_{i}\right)  $, for $\mathbf{x}=\left(  x_{1},\ldots,x_{n}\right)  $ in
$\mathbb{R}^{n},$ we have the result (\ref{Phy2}).
\end{corollary}

For $n=2,$ we can give the following corollary.

\begin{corollary}
\label{cor_cor_n_2}For $n=2,$ we have
\begin{equation}
R_{2}\left(  \mathbf{z}\right)  =\widetilde{b}_{1,2}z_{1}z_{2}, \label{R2}%
\end{equation}
and%
\begin{equation}
c_{\boldsymbol{\alpha},\lambda}\left(  R_{2}\right)  =\frac{\left(
\lambda\right)  _{l}}{l!}\widetilde{b}_{1,2}^{l}, \label{calphaR2}%
\end{equation}
and $c_{\boldsymbol{\alpha},\lambda}\left(  R_{2}\right)  =0$ otherwise.
Hence, we have%
\begin{equation}
\boldsymbol{\gamma}_{\left(  P_{2},\lambda\right)  }\left(  \mathtt{d}%
\mathbf{x}\right)  =\frac{p_{1,2}^{-\lambda}}{\left[  \Gamma\left(
\lambda\right)  \right]  ^{2}}\exp(-\frac{p_{2}}{p_{1,2}}\boldsymbol{x}%
_{1}-\frac{p_{1}}{p_{1,2}}x_{2})\left(  x_{1}x_{2}\right)  ^{\left(
\lambda-1\right)  }F_{1}\left(  \lambda,\widetilde{b}_{1,2}x_{1}x_{2}\right)
\mathbf{1}_{\left(  0,\infty\right)  ^{2}}\left(  \mathbf{x}\right)  \left(
\mathtt{d}\mathbf{x}\right)  . \label{gammaP2_lambda}%
\end{equation}
We obtain the formula (\ref{mud2}).
\end{corollary}

For $n=3,$ we can give the following corollary.

\begin{corollary}
\label{cor_cor_n_3}For $n=3,$ we have
\begin{equation}
R_{3}\left(  \mathbf{z}\right)  =\widetilde{b}_{1,2}z_{1}z_{2}+\widetilde{b}%
_{1,3}z_{1}z_{3}+\widetilde{b}_{2,3}z_{2}z_{3}+\widetilde{b}_{1,2,3}z_{1}%
z_{2}z_{3} \label{R3}%
\end{equation}
If $\boldsymbol{\alpha}=\left(  \mathbb{\alpha}_{1},\mathbb{\alpha}%
_{2},\mathbb{\alpha}_{3}\right)  \in\mathbb{N}^{3},$ and $\max\left(
\mathbb{\alpha}_{1},\mathbb{\alpha}_{2},\mathbb{\alpha}_{3}\right)
=\left\Vert \boldsymbol{\alpha}\right\Vert _{\infty}$, if $\left\Vert
\boldsymbol{\alpha}\right\Vert _{\infty}\leqslant\frac{\left\vert
\boldsymbol{\alpha}\right\vert }{2},$ we have
\begin{equation}
c_{\boldsymbol{\alpha},\lambda}\left(  R_{3}\right)  =\sum_{\left\Vert
\boldsymbol{\alpha}\right\Vert _{\infty}\leqslant k\leqslant\frac{\left\vert
\boldsymbol{\alpha}\right\vert }{2},k\in\mathbb{N}}\tfrac{\left(
\lambda\right)  _{k}\widetilde{b}_{1,2}^{k-\alpha_{3}}\widetilde{b}%
_{1,3}^{k-\alpha_{2}}\widetilde{b}_{2,3}^{k-\alpha_{1}}\widetilde{b}%
_{1,2,3}^{\alpha_{1}+\alpha_{2}+\alpha_{3}-2k}}{\left(  k-\alpha_{3}\right)
!\left(  k-\alpha_{2}\right)  !\left(  k-\alpha_{1}\right)  !\left(
\alpha_{1}+\alpha_{2}+\alpha_{3}-2k\right)  !}. \label{calphaR3}%
\end{equation}
and $c_{\boldsymbol{\alpha},\lambda}\left(  R_{3}\right)  =0$ if $\left\Vert
\boldsymbol{\alpha}\right\Vert _{\infty}>\frac{\left\vert \boldsymbol{\alpha
}\right\vert }{2}$, in particular if $\left\vert \boldsymbol{\alpha
}\right\vert =1,$ $c_{\boldsymbol{\alpha},\lambda}\left(  R_{3}\right)  =0$.
Thus we have for $\mathbf{x}=\left(  x_{1},x_{2},x_{3}\right)  $
\begin{align}
\boldsymbol{\gamma}_{\left(  P_{3},\lambda\right)  }\left(  \mathtt{d}%
\mathbf{x}\right)   &  =\frac{p_{\left[  3\right]  }^{-\lambda}}{\left[
\Gamma\left(  \lambda\right)  \right]  ^{3}}\exp\left(  \boldsymbol{\theta
}_{P},\mathbf{x}\right)  \mathbf{x}^{\left(  \lambda-1\right)  \mathbf{1}_{3}%
}\nonumber\\
&  \times\{\sum_{\boldsymbol{\alpha}\in\mathbb{N}^{3},\left\Vert
\boldsymbol{\alpha}\right\Vert _{\infty}\leqslant\frac{\left\vert
\boldsymbol{\alpha}\right\vert }{2}}[\sum_{\left\Vert \boldsymbol{\alpha
}\right\Vert _{\infty}\leqslant k\leqslant\frac{\left\vert \boldsymbol{\alpha
}\right\vert }{2},k\in\mathbb{N}}\tfrac{\left(  \lambda\right)  _{k}%
\widetilde{b}_{1,2}^{k-\alpha_{3}}\widetilde{b}_{1,3}^{k-\alpha_{2}%
}\widetilde{b}_{2,3}^{k-\alpha_{1}}\widetilde{b}_{1,2,3}^{\alpha_{1}%
+\alpha_{2}+\alpha_{3}-2k}}{\left(  k-\alpha_{3}\right)  !\left(  k-\alpha
_{2}\right)  !\left(  k-\alpha_{1}\right)  !\left(  \alpha_{1}+\alpha
_{2}+\alpha_{3}-2k\right)  !}]\frac{\mathbf{x}^{\boldsymbol{\alpha}}}{\left(
\lambda\right)  _{\boldsymbol{\alpha}}}\}\mathbf{1}_{\left(  0,\infty\right)
^{3}}\left(  \mathbf{x}\right)  \left(  \mathtt{d}\mathbf{x}\right)  .
\label{gammaP3lambda}%
\end{align}
Or, with for $\mathbf{z}_{\left[4\right]}=\left(  z_{1},z_{2},z_{3},z_{4}\right)
\in\mathbb{R}^{4},$ and $_{1}\mathbf{F}_{3}$ is a generalized multivariate
Lauricella function defined by%
\begin{equation}
_{1}\mathbf{F}_{3}(\lambda;\mathbf{z}_{\left[4\right]})=\sum_{\boldsymbol{l}\in
\mathbb{N}^{4}}\frac{\left(  \lambda\right)  _{l_{1}+l_{2}+l_{3}+l_{4}}%
}{\left(  \lambda\right)  _{l_{2}+l_{3}+l_{4}}\left(  \lambda\right)
_{l_{1}+l_{3}+l_{4}}\left(  \lambda\right)  _{l_{1}+l_{2}+l_{4}}}%
\frac{\mathbf{z}_{\left[4\right]}^{\boldsymbol{l}}}{\boldsymbol{l}!}. \label{gLauricella_4}%
\end{equation}%
\begin{align}
\boldsymbol{\gamma}_{\left(  P_{3},\lambda\right)  }\left(  \mathtt{d}%
\mathbf{x}\right)  =\frac{p_{\left[  3\right]  }^{-\lambda}}{\left[
\Gamma\left(  \lambda\right)  \right]  ^{3}}\exp\left(  \boldsymbol{\theta
}_{P_{3}},\mathbf{x}\right) \times\nonumber\\
& \hspace{-3cm}\mathbf{x}^{\left(  \lambda-1\right)
\mathbf{1}_{3}}{}_{1}\mathbf{F}_{3}(\lambda;\widetilde{b}_{1,2}x_{1}%
x_{2},\widetilde{b}_{1,3}x_{1}x_{3},\widetilde{b}_{2,3}x_{2}x_{3}%
,\widetilde{b}_{1,2,3}x_{1}x_{2}x_{3})\mathbf{1}_{\left(  0,\infty\right)
^{3}}\left(  \mathbf{x}\right)  \left(  \mathtt{d}\mathbf{x}\right)  .
\label{gammaP3lambda2}%
\end{align}

\end{corollary}

We note that the latter result compared with (\ref{MGD3}) gives with
$z_{1}=\widetilde{b}_{1,2}x_{1}x_{2},z_{2}=\widetilde{b}_{1,3}x_{1}x_{3}%
,z_{3}=\widetilde{b}_{2,3}x_{2}x_{3},z_{4}=\widetilde{b}_{1,2,3}x_{1}%
x_{2}x_{3}$: 
$
\sum_{\boldsymbol{l}%
\in\mathbb{N}^{4}}\frac{\left(  \lambda\right)  _{l_{1}+l_{2}+l_{3}+l_{4}}%
}{\left(  \lambda\right)  _{l_{2}+l_{3}+l_{4}}\left(  \lambda\right)
_{l_{1}+l_{3}+l_{4}}\left(  \lambda\right)  _{l_{1}+l_{2}+l_{4}}}%
\frac{\mathbf{z}_{4}^{\boldsymbol{l}}}{\boldsymbol{l}!}
=\sum_{\boldsymbol{l}\in\mathbb{N}^{4}}\frac{1}{\left(  \lambda\right)
_{l_{1}+l_{3}+l_{4}}\left(  \lambda\right)  _{l_{2}+2l_{3}+l_{4}}}%
\frac{\left(  z_{1},z_{2}+z_{3,},z_{2}z_{3},z_{4}\right)  ^{\boldsymbol{l}}%
}{\boldsymbol{l}!}.
$

In the particular three cases $\widetilde{b}_{1,2},\widetilde{b}%
_{1,3},\widetilde{b}_{2,3},\widetilde{b}_{1,2,3}>0,$ $\widetilde{b}%
_{1,2},\widetilde{b}_{1,3},\widetilde{b}_{2,3}>0$, $\widetilde{b}_{1,2,3}=0$
and $\widetilde{b}_{1,2}=\widetilde{b}_{1,3}=\widetilde{b}_{2,3}=0$,
$\widetilde{b}_{1,2,3}>0,$ we can give the following remarks.

\begin{remark}
\label{cor_cor_n_3_Rem1}If $\widetilde{b}_{1,2},\widetilde{b}_{1,3}%
,\widetilde{b}_{2,3},\widetilde{b}_{1,2,3}>0,$ if $\left\Vert
\boldsymbol{\alpha}\right\Vert _{\infty}\leqslant\frac{\left\vert
\boldsymbol{\alpha}\right\vert }{2},$
\begin{equation}
c_{\boldsymbol{\alpha},\lambda}\left(  R_{3}\right)  =\sum_{\left\Vert
\boldsymbol{\alpha}\right\Vert _{\infty}\leqslant k\leqslant\frac{\left\vert
\boldsymbol{\alpha}\right\vert }{2},,k\in\mathbb{N}}\tfrac{\left(
\lambda\right)  _{k}\widetilde{b}_{2,3}^{k-\alpha_{1}}\widetilde{b}%
_{1,3}^{k-\alpha_{2}}\widetilde{b}_{1,2}^{k-\alpha_{3}}\widetilde{b}%
_{1,2,3}^{\alpha_{1}+\alpha_{2}+\alpha_{3}-2k}}{\left(  k-\alpha_{1}\right)
!\left(  k-\alpha_{2}\right)  !\left(  k-\alpha_{3}\right)  !\left(
\alpha_{1}+\alpha_{2}+\alpha_{3}-2k\right)  !}>0 \label{calphaR3bTpositive}%
\end{equation}
and $c_{\boldsymbol{\alpha},\lambda}\left(  R_{3}\right)  =0$ if $\left\Vert
\boldsymbol{\alpha}\right\Vert _{\infty}>\frac{\left\vert \boldsymbol{\alpha
}\right\vert }{2}$. Thus we have
\begin{align}
\boldsymbol{\gamma}_{\left(  P_{3},\lambda\right)  }\left(  \mathtt{d}%
\mathbf{x}\right)   &  =\frac{p_{\left[  3\right]  }^{-\lambda}}{\left[
\Gamma\left(  \lambda\right)  \right]  ^{3}}\exp\left(  \boldsymbol{\theta
}_{P},\mathbf{x}\right)  \mathbf{x}^{\left(  \lambda-1\right)  \mathbf{1}_{3}%
}\times\nonumber\\
&  \{\sum_{\boldsymbol{\alpha}\in\mathbb{N}^{3},\left\Vert
\boldsymbol{\alpha}\right\Vert _{\infty}\leqslant\frac{\left\vert
\boldsymbol{\alpha}\right\vert }{2}}[\sum_{\left\Vert \boldsymbol{\alpha
}\right\Vert _{\infty}\leqslant k\leqslant\frac{\left\vert \boldsymbol{\alpha
}\right\vert }{2},k\in\mathbb{N}}\tfrac{\left(  \lambda\right)  _{k}%
\widetilde{b}_{2,3}^{k-\alpha_{1}}\widetilde{b}_{1,3}^{k-\alpha_{2}%
}\widetilde{b}_{1,2}^{k-\alpha_{3}}\widetilde{b}_{1,2,3}^{\alpha_{1}%
+\alpha_{2}+\alpha_{3}-2k}}{\left(  k-\alpha_{1}\right)  !\left(  k-\alpha
_{2}\right)  !\left(  k-\alpha_{3}\right)  !\left(  \alpha_{1}+\alpha
_{2}+\alpha_{3}-2k\right)  !}]\frac{\mathbf{x}^{\boldsymbol{\alpha}}}{\left(
\lambda\right)  _{\boldsymbol{\alpha}}}\}\mathbf{1}_{\left(  0,\infty\right)
^{3}}\left(  \mathbf{x}\right)  \left(  \mathtt{d}\mathbf{x}\right)  .
\label{gammaP3lambda0}%
\end{align}

\end{remark}

\begin{remark}
\label{cor_cor_n_3_Rem2}If $\widetilde{b}_{1,2},\widetilde{b}_{1,3}%
,\widetilde{b}_{2,3}>0$ and $\widetilde{b}_{1,2,3}=0,$ if $\left\Vert
\boldsymbol{\alpha}\right\Vert _{\infty}\leqslant\frac{\left\vert
\boldsymbol{\alpha}\right\vert }{2}=k\in\mathbb{N},$ then
\begin{equation}
c_{\boldsymbol{\alpha},\lambda}\left(  R_{3}\right)  =\tfrac{\left(
\lambda\right)  _{k}\widetilde{b}_{2,3}^{k-\alpha_{1}}\widetilde{b}%
_{1,3}^{k-\alpha_{2}}\widetilde{b}_{1,2}^{k-\alpha_{3}}}{\left(  k-\alpha
_{1}\right)  !\left(  k-\alpha_{2}\right)  !\left(  k-\alpha_{3}\right)  !}>0,
\label{calphaR3bTpb123z}%
\end{equation}
and $c_{\boldsymbol{\alpha},\lambda}\left(  R_{3}\right)  =0$ if $\left\Vert
\boldsymbol{\alpha}\right\Vert _{\infty}>\frac{\left\vert \boldsymbol{\alpha
}\right\vert }{2}\in\mathbb{N}$ or $\frac{\left\vert \boldsymbol{\alpha
}\right\vert }{2}\notin\mathbb{N}$. Thus we have
\begin{equation}
\boldsymbol{\gamma}_{\left(  P_{3},\lambda\right)  }\left(  \mathtt{d}%
\mathbf{x}\right)  =\tfrac{p_{\left[  3\right]  }^{-\lambda}}{\left[
\Gamma\left(  \lambda\right)  \right]  ^{3}}\exp\left(  \boldsymbol{\theta
}_{P},\mathbf{x}\right)  \mathbf{x}^{\left(  \lambda-1\right)  \mathbf{1}_{3}%
}[\sum_{\boldsymbol{\alpha}\in\mathbb{N}^{3},\left\Vert \boldsymbol{\alpha
}\right\Vert _{\infty}\leqslant\frac{\left\vert \boldsymbol{\alpha}\right\vert
}{2}=k\in\mathbb{N}}\tfrac{\left(  \lambda\right)  _{k}\widetilde{b}%
_{2,3}^{k-\alpha_{1}}\widetilde{b}_{1,3}^{k-\alpha_{2}}\widetilde{b}%
_{1,2}^{k-\alpha_{3}}}{\left(  k-\alpha_{1}\right)  !\left(  k-\alpha
_{2}\right)  !\left(  k-\alpha_{3}\right)  !}\frac{\mathbf{x}%
^{\boldsymbol{\alpha}}}{\left(  \lambda\right)  _{\boldsymbol{\alpha}}%
}]\mathbf{1}_{\left(  0,\infty\right)  ^{3}}\left(  \mathbf{x}\right)  \left(
\mathtt{d}\mathbf{x}\right)  . \label{gammaP3lambda1}%
\end{equation}

\end{remark}

\begin{remark}
\label{cor_cor_n_3_Rem3}If $\widetilde{b}_{1,2},\widetilde{b}_{1,3}%
,\widetilde{b}_{2,3}=0$, $\widetilde{b}_{1,2,3}>0,$ then for
$\boldsymbol{\alpha=}k\mathbf{1}_{3}$ $,k\in\mathbb{N},$
\begin{equation}
c_{k\mathbf{1}_{3},\lambda}\left(  R_{3}\right)  =\frac{\left(  \lambda
\right)  _{k}}{k!}\widetilde{b}_{1,2,3}^{k} \label{calphaR3bTzb123positive}%
\end{equation}
and $c_{\boldsymbol{\alpha},\lambda}\left(  R_{3}\right)  =0$ otherwise. Thus
we have%
\begin{equation}
\boldsymbol{\gamma}_{\left(  P_{3},\lambda\right)  }\left(  \mathtt{d}%
\mathbf{x}\right)  =\frac{p_{\left[  3\right]  }^{-\lambda}}{\left[
\Gamma\left(  \lambda\right)  \right]  ^{3}}\exp\left(  \boldsymbol{\theta
}_{P},\mathbf{x}\right)  \mathbf{x}^{\left(  \lambda-1\right)  \mathbf{1}_{3}%
}F_{2}(\lambda,\lambda;\widetilde{b}_{1,2,3}\mathbf{x}^{\left[  3\right]
})\mathbf{1}_{\left(  0,\infty\right)  ^{3}}\left(  \mathbf{x}\right)  \left(
\mathtt{d}\mathbf{x}\right)  \label{gammaP3lambda3}%
\end{equation}

\end{remark}

\section{Conditional Laplace transform}

We now assume that $P_{n}$ is an affine polynomial such that $P_{n}\left(
\boldsymbol{\theta}\right)  =\sum_{T\in\mathfrak{P}_{n}}p_{T}%
\boldsymbol{\theta}^{T}=1+\sum_{T\in\mathfrak{P}_{n}^{\ast}}p_{T}%
\boldsymbol{\theta}^{T}$, with $p_{\left[  k\right]  }>0$ for $k=1,\ldots,n,$
and is such that the mgd $\boldsymbol{\gamma}_{\left(  P_{n},\lambda\right)
}$ associated with $\left(  P_{n},\lambda\right)  $ is infinitely divisible.
If $k\in\left[  n\right]  $, we denote by $\boldsymbol{\theta}_{\left[
k\right]  }=\left(  \theta_{1},\ldots,\theta_{k}\right)  $,
$\boldsymbol{\theta}_{\left[  n\right]  \smallsetminus\left[  k\right]
}=\left(  \theta_{k+1},\ldots,\theta_{n}\right)  $, and $P_{k}$ the affine
polynomial $P_{k}\left(  \boldsymbol{\theta}_{\left[  k\right]  }\right)
=\sum_{T\in\mathfrak{P}_{k}}p_{T}\boldsymbol{\theta}_{\left[  k\right]  }%
^{T}=P_{n}\left(  \boldsymbol{\theta}_{\left[  k\right]  },\mathbf{0}%
_{n-k}\right)  $. Similarly, if $T=\left\{  t_{1},\ldots,t_{k}\right\}
\subset\left[  n\right]  ,$ $t_{1}<\ldots<t_{k},$ we denote by
$\boldsymbol{\theta}_{T}=\left(  \theta_{t_{1}},\ldots,\theta_{t_{k}}\right)
$, and if $\left[  n\right]  \smallsetminus T=\overline{T}=\left\{
t_{k+1},\ldots t_{n}\right\}  ,$ $\boldsymbol{\theta}_{\left[  n\right]
\smallsetminus T}=\left(  \theta_{t_{k+1}},\ldots,\theta_{t_{n}}\right)  $,
and $P_{T}$ the affine polynomial $P_{T}\left(  \boldsymbol{\theta}%
_{T}\right)  =\sum_{S\in\mathfrak{P}\left(  T\right)  }p_{S}\boldsymbol{\theta
}_{T}^{S}=P_{n}\left(  \boldsymbol{\theta}_{T}^{\prime}\right)  $, where
$\left(  \boldsymbol{\theta}_{T}^{\prime}\right)  _{i}=\allowbreak
\theta_{t_{i}}$ if $t_{i}\in T$ and $\left(  \boldsymbol{\theta}_{T}^{\prime
}\right)  _{i}=0$ if $t_{i}\notin T.$ For all $\boldsymbol{\alpha}=\left(
\alpha_{1},\ldots,\alpha_{n}\right)  \in\mathbb{N}^{n}$, we introduce the
notation  
$
(\partial/\partial\boldsymbol{\theta})^{\boldsymbol{\alpha}}%
=\partial^{\left\vert \boldsymbol{\alpha}\right\vert }/\partial
\theta_{1}^{\alpha_{1}}\cdots\partial\theta_{n}^{\alpha_{n}}.
$
For all $T\in\mathfrak{P}_{n}$, we also define $\left(  \partial
/\partial\boldsymbol{\theta}\right)  ^{T}=\left(  \partial/\partial
\boldsymbol{\theta}\right)  ^{\mathbf{1}_{T}}$.

Let $\mathbf{X}=\left(  X_{1},\ldots,X_{n}\right)  $ be a random real vector
such that has distribution $\boldsymbol{\gamma}_{\left(  P_{n},\lambda\right)
},$ denoted by $\mathbf{X}\sim\boldsymbol{\gamma}_{\left(  P_{n}%
,\lambda\right)  },$ we give a formula for the Lt of $X_{\left[
n\smallsetminus\left[  k\right]  \right]  }=\left(  X_{k+1},\ldots
,X_{n}\right)  $ given $\mathbf{X}_{\left[  k\right]  }=\mathbf{x}_{\left[
k\right]  }$, an important conditional distribution for the simulations of
$\mathbf{X}$, in the following main theorem.

\begin{theorem}
\label{L_X1_n_0} Let $\mathbf{X}=\left(  X_{1},\ldots,X_{n}\right)  $ be a
random real vector such that $\mathbf{X}\sim\boldsymbol{\gamma}_{\left(
P_{n},\lambda\right)  },$ with the notation of Theorem \ref{Th_gamma_P_lambda}%
. Let $1<n\in\mathbb{N},$ $1\leqslant k<n,$ and $Q_{\left[  n\right]
\smallsetminus\left[  k\right]  }$ the affine polynomials with respect to the
$n-k$ variables $\theta_{k+1},\ldots,\theta_{n}$ defined by%
\begin{equation}
Q_{\left[  n\right]  \smallsetminus\left[  k\right]  }\left(
\boldsymbol{\theta}_{\left[  n\right]  \smallsetminus\left[  k\right]
}\right)  =\prod_{i=k+1}^{n}[1+\theta_{i}\left(  -\boldsymbol{\theta}_{P_{n}%
}\right)  _{i}^{-1}] \label{Qn_k}%
\end{equation}
If $\mathbf{x}=\left(  x_{1},\ldots,x_{n}\right)  \in\mathbb{R}^{n},$ we
denote $\mathbf{x}_{\left[  k\right]  }=\left(  x_{1},\ldots,x_{k}\right)
\in\mathbb{R}^{k}.$ If $\mathbf{y}=\left(  y_{1},\ldots,y_{n}\right)
\in\mathbb{R}^{n},$ let $\mathbf{y}_{\left[  k\right]  }\mathbf{x}_{\left[
k\right]  }=\left(  y_{1}x_{1},\ldots,y_{k}x_{k}\right)  $, then the Lt of
$\mathbf{X}_{\left[  n\right]  \smallsetminus\left[  k\right]  }$ given
$\mathbf{X}_{\left[  k\right]  }=\mathbf{x}_{\left[  k\right]  }$ is
\begin{equation}
L_{\mathbf{X}_{\left[  n\right]  \smallsetminus\left[  k\right]  }%
}^{\mathbf{X}_{\left[  k\right]  }=\mathbf{x}_{\left[  k\right]  }}\left(
\boldsymbol{\theta}_{\left[  n\right]  \smallsetminus\left[  k\right]
}\right)  =[Q_{\left[  n\right]  \smallsetminus\left[  k\right]  }\left(
\boldsymbol{\theta}_{\left[  n\right]  \smallsetminus\left[  k\right]
}\right)  ]^{-\lambda}\frac{\mathbf{F}_{k}\left(  \lambda,R_{n},\mathbf{x}%
_{\left[  k\right]  },\boldsymbol{\theta}_{\left[  n\right]  \smallsetminus
\left[  k\right]  }\right)  }{\mathbf{F}_{k}\left(  \lambda,R_{n}%
,\mathbf{x}_{\left[  k\right]  },\boldsymbol{0}_{\left[  n\right]
\smallsetminus\left[  k\right]  }\right)  } \label{L_Zk_xk}%
\end{equation}
with%
\begin{equation}
\mathbf{F}_{k}\left(  \lambda,R_{n},\mathbf{x}_{\left[  k\right]
},\boldsymbol{\theta}_{\left[  n\right]  \smallsetminus\left[  k\right]
}\right)  =\sum_{\boldsymbol{\alpha=}\left(  \boldsymbol{\alpha}_{\left[
k\right]  },\boldsymbol{\alpha}_{\left[  n\right]  \smallsetminus\left[
k\right]  }\right)  \in\mathbb{N}^{n},c_{\boldsymbol{\alpha},\lambda}\left(
R_{n}\right)  \neq0}\frac{\mathbf{x}_{\left[  k\right]  }%
^{_{_{\boldsymbol{\alpha}_{\left[  k\right]  }}}}}{\left(  \lambda\right)
_{\boldsymbol{\alpha}_{\left[  k\right]  }}}c_{\boldsymbol{\alpha},\lambda
}\left(  R_{n}\right)  \left(  \boldsymbol{\theta}_{\left[  n\right]
}-\boldsymbol{\theta}_{P_{n}}\right)  _{\left[  n\right]  \smallsetminus
\left[  k\right]  }^{-\boldsymbol{\alpha}_{\left[  n\right]  \smallsetminus
\left[  k\right]  }} \label{Fkdef}%
\end{equation}

\end{theorem}

\section{Conditional Laplace transform in the particular case $k=1$}

Another form of Theorem \ref{L_X1_n_0} can be given for $k=1$ by the following theorem.

\begin{theorem}
\label{L_X1_n_1} Let $\mathbf{X}=\left(  X_{1},\ldots,X_{n}\right)  $ be a
random real vector such that $\mathbf{X}\sim\boldsymbol{\gamma}_{\left(
P_{n},\lambda\right)  },$ with the notation of Theorem \ref{Th_gamma_P_lambda}%
. Let $1<n\in\mathbb{N},$ $k=1,$ and $S_{n-1},$ $B_{n-1}$ the affine
polynomials with respect to the $n-1$ variables $\theta_{2},\ldots,\theta_{n}$
defined by

\begin{align}
&S_{\left[  n\right]  \smallsetminus\left[  1\right]  }\left(
\boldsymbol{\theta}_{\left[  n\right]  \smallsetminus\left[  1\right]
}\right)  =\frac{p_{\left[  n\right]  }}{p_{1}}\left(  \boldsymbol{\theta}
_{n}-\boldsymbol{\theta}_{P_{n}}\right)  _{\left[  n\right]  \smallsetminus
\left[  1\right]  }^{\mathbf{1}_{\left[  n\right]  \smallsetminus\left[
1\right]  }}\left[  1-R_{n}\left(  0,\left(  \boldsymbol{\theta}%
_{n}-\boldsymbol{\theta}_{P_{n}}\right)  _{\left[  n\right]  \smallsetminus
\left[  1\right]  }^{-1}\right)  \right]  ,\label{S} \\
&B_{n-1}\left(  \boldsymbol{\theta}_{\left[  n\right]  \smallsetminus\left[
1\right]  }\right)  =\frac{p_{\left[  n\right]  }}{p_{1}}\left(
\boldsymbol{\theta}_{n}-\boldsymbol{\theta}_{P_{n}}\right)  _{\left[
n\right]  \smallsetminus\left[  1\right]  }^{\mathbf{1}_{\left[  n\right]
\smallsetminus\left[  1\right]  }}\frac{\partial}{\partial z_{1}}R_{n}\left(
0,\left(  \boldsymbol{\theta}_{n}-\boldsymbol{\theta}_{P_{n}}\right)
_{\left[  n\right]  \smallsetminus\left[  1\right]  }^{-1}\right)  ,\label{B} 
\end{align}

we have

\begin{align}
&S_{\left[  n\right]  \smallsetminus\left[  1\right]  }\left(
\boldsymbol{\theta}_{\left[  n\right]  \smallsetminus\left[  1\right]
}\right)  =\frac{1}{p_{1}}\frac{\partial}{\partial\theta_{1}}P_{n}\left(
\boldsymbol{\theta}_{n}\right)  , \label{S2} \\
&B_{n-1}\left(  \boldsymbol{\theta}_{\left[  n\right]  \smallsetminus\left[
1\right]  }\right)  =-\frac{1}{p_{1}}P_{n}\left(  \widetilde{p}_{1}%
,\boldsymbol{\theta}_{\left[  n\right]  \smallsetminus\left[  1\right]
}\right)  ,\label{B2}
\end{align}

and%
\begin{equation}
S_{\left[  n\right]  \smallsetminus\left[  1\right]  }\left(
\boldsymbol{\theta}_{\left[  n\right]  \smallsetminus\left[  1\right]
}\right)  =\sum_{T\subset\left[  n\right]  \smallsetminus\left[  1\right]
}\frac{p_{\left[  1\right]  \cup T}}{p_{\left[  1\right]  }}\boldsymbol{\theta
}_{\left[  n\right]  \smallsetminus\left[  1\right]  }^{T}. \label{S3}%
\end{equation}
Let $\mathfrak{z}_{n-1}$ the function with respect to the $n-1$ variables
$\theta_{2},\ldots,\theta_{n}$ defined by%
\begin{equation}
\mathfrak{z}_{n-1}\left(  \boldsymbol{\theta}_{\left[  n\right]
\smallsetminus\left[  1\right]  }\right)    =\frac{B_{n-1}\left(
\boldsymbol{\theta}_{\left[  n\right]  \smallsetminus\left[  1\right]
}\right)  }{S_{\left[  n\right]  \smallsetminus\left[  1\right]  }\left(
\boldsymbol{\theta}_{\left[  n\right]  \smallsetminus\left[  1\right]
}\right)  }
 =\frac{\frac{\partial}{\partial z_{1}}R_{n}\left(  0,\left(
\boldsymbol{\theta}_{\left[  n\right]  }-\boldsymbol{\theta}_{P_{n}}\right)
_{\left[  n\right]  \smallsetminus\left[  1\right]  }^{-1}\right)  }%
{1-R_{n}\left(  0,\left(  \boldsymbol{\theta}_{\left[  n\right]
}-\boldsymbol{\theta}_{P_{n}}\right)  _{\left[  n\right]  \smallsetminus
\left[  1\right]  }^{-1}\right)  }. \label{dzeta_}%
\end{equation}
Let us define
\begin{equation}
I_{1}\left(  R_{n}\right)  =\left\{  \alpha_{1}\in\mathbb{N},\exists
\boldsymbol{\alpha}=\left(  \alpha_{1},\boldsymbol{\alpha}_{\left[  n\right]
\smallsetminus\left\{  1\right\}  }\right)  \in\mathbb{N}^{n}%
,c_{\boldsymbol{\alpha},\lambda}\left(  R_{n}\right)  \neq0\right\}  ,
\label{I1_Rn}%
\end{equation}
then we define $\mathbf{G},$ a function with respect to the variable $u_{1}$,
by
\begin{equation}
\mathbf{G}\left(  R_{n},u_{1}\right)  =\sum_{\alpha_{1}\in I_{1}\left(
R_{n}\right)  }\frac{u_{1}^{\alpha_{1}}}{\alpha_{1}!} \label{F2}%
\end{equation}
If $\mathbf{x}=\left(  x_{1},\ldots,x_{n}\right)  \in\mathbb{R}^{n},$ we
denote $\mathbf{x}_{\left[  k\right]  }=\left(  x_{1},\ldots,x_{k}\right)
\in\mathbb{R}^{k}.$ If $\mathbf{y}=\left(  y_{1},\ldots,y_{n}\right)
\in\mathbb{R}^{n},$ let $\boldsymbol{y}_{\left[  k\right]  }\mathbf{x}%
_{\left[  k\right]  }=\left(  y_{1}x_{1},\ldots,y_{k}x_{k}\right)  $, then the
Lt of $\mathbf{X}_{\left[  n\right]  \smallsetminus\left[  1\right]  }=\left(
X_{2},\ldots,X_{n}\right)  $ given $X_{1}=x_{1}$ is
\begin{equation}
L_{\mathbf{X}_{\left[  n\right]  \smallsetminus\left[  1\right]  }}%
^{X_{1}=x_{1}}\left(  \boldsymbol{\theta}_{\left[  n\right]  \smallsetminus
\left[  1\right]  }\right)  =[S_{\left[  n\right]  \smallsetminus\left[
1\right]  }\left(  \boldsymbol{\theta}_{\left[  n\right]  \smallsetminus
\left[  1\right]  }\right)  ]^{-\lambda}\frac{\mathbf{G}\left(  R_{n}%
,\mathfrak{z}_{n-1}\left(  \boldsymbol{\theta}_{\left[  n\right]
\smallsetminus\left[  1\right]  }\right)  x_{1}\right)  }{\mathbf{G}\left(
R_{n},\mathfrak{z}_{n-1}\left(  \boldsymbol{0}_{\left[  n\right]
\smallsetminus\left[  1\right]  }\right)  x_{1}\right)  }
\label{LXn_1_X1_Yn_1}
\end{equation}

\end{theorem}

Before prove Theorem \ref{L_X1_n_1}, we give the following remark for the
affine polynomial \newline $S_{\left[  n\right]  \smallsetminus\left[  1\right]
}\left(  \boldsymbol{\theta}_{\left[  n\right]  \smallsetminus\left[
1\right]  }\right)  =\sum_{T\subset\left[  n\right]  \smallsetminus\left[
1\right]  }\frac{p_{\{1\}\cup T}}{p_{1}}(\boldsymbol{\theta}_{\left[
n\right]  \smallsetminus\left[  1\right]  })^{T}.$

\begin{remark}
\label{RemL_X1_n_1}For $T\in\mathfrak{P}_{\left[  n\right]  \smallsetminus
\left[  1\right]  },$ we have $\widetilde{p}_{T}(S_{\left[  n\right]
\smallsetminus\left[  1\right]  })=\widetilde{p}_{T}(R_{n}),$ hence we have
$\widetilde{b}_{T}\left(  S_{\left[  n\right]  \smallsetminus\left[  1\right]
}\right)  =\widetilde{b}_{T}\left(  P_{n}\right)  .$ Therefore, if
$\mathbf{\gamma}_{(P_{n},\lambda)}$ is infinitely divisible $\mathbf{\gamma
}_{(S_{\left[  n\right]  \smallsetminus\left[  1\right]  },\lambda)}$ is also
infinitely divisible by Theorem \ref{mainresult}. We also have
\begin{equation}
r_{T}\left(  S_{\left[  n\right]  \smallsetminus\left[  1\right]  }\right)
=r_{T}\left(  P_{n}\right)  =r_{T} \label{r_T_S_n_1}%
\end{equation}
and therefore
\begin{equation}
S_{\left[  n\right]  \smallsetminus\left[  1\right]  }\left(
\boldsymbol{\theta}_{\left[  n\right]  \smallsetminus\left[  1\right]
}\right)  =\frac{p_{\left[  n\right]  }}{p_{1}}\left(  \boldsymbol{\theta
}_{\left[  n\right]  }-\boldsymbol{\theta}_{P_{n}}\right)  ^{\left[  n\right]
\smallsetminus\left\{  1\right\}  }\{1-\sum_{T\subset\left[  n\right]
\smallsetminus\left\{  1\right\}  }r_{T}\left(  \boldsymbol{\theta}_{\left[
n\right]  }-\boldsymbol{\theta}_{P_{n}}\right)  ^{-T}\} \label{S_n_1_rT}%
\end{equation}

\end{remark}

Now, we can prove Theorem \ref{L_X1_n_1}. If there is no ambiguity, we denote
$S_{\left[  n\right]  \smallsetminus\left[  1\right]  },B_{n-1}$ and
$\mathfrak{z}_{n-1}$ by $S,B$ and $\mathfrak{z}$.

Before giving the main theorem, we prove by finite induction the following lemma.

\begin{lemma}
\label{I1(Rn)}With the notations of Theorem (\ref{L_X1_n_1}), unless $Rn=0$,
we have,%
\[
I_{1}\left(  R_{n}\right)  =\left\{  \alpha_{1}\in\mathbb{N},\exists
\boldsymbol{\alpha}=\left(  \alpha_{1},\boldsymbol{\alpha}_{\left[  n\right]
\smallsetminus\left\{  1\right\}  }\right)  \in\mathbb{N}^{n}%
,c_{\boldsymbol{\alpha},\lambda}\left(  R_{n}\right)  \neq0\right\}
=\mathbb{N},
\]
therefore $G$ defined in Theorem (\ref{L_X1_n_1}) by (\ref{F2}) is $\exp$.
\end{lemma}

Now, we can give the following main Theorem.

\begin{theorem}
\label{ThLXn_1_X1}With the notations of Theorem (\ref{L_X1_n_1}), unless
$Rn=0$, we have%
\begin{align}
L_{\mathbf{X}_{\left[  n\right]  \smallsetminus\left[  1\right]  }}%
^{X_{1}=x_{1}}\left(  \boldsymbol{\theta}_{\left[  n\right]  \smallsetminus
\left[  1\right]  }\right)   &  =S_{\left[  n\right]  \smallsetminus\left[
1\right]  }^{-\lambda}\left(  \boldsymbol{\theta}_{\left[  n\right]
\smallsetminus\left[  1\right]  }\right)  \exp\{-[\tfrac{P_{n}\left(
0,\boldsymbol{\theta}_{\left[  n\right]  \smallsetminus\left[  1\right]
}\right)  }{S_{\left[  n\right]  \smallsetminus\left[  1\right]  }\left(
\boldsymbol{\theta}_{\left[  n\right]  \smallsetminus\left[  1\right]
}\right)  }-1]\frac{x_{1}}{p_{1}}\}\label{LXn_1_Yn_1s}\\
&  =S_{\left[  n\right]  \smallsetminus\left[  1\right]  }^{-\lambda}\left(
\boldsymbol{\theta}_{\left[  n\right]  \smallsetminus\left[  1\right]
}\right)  \exp\{-[\frac{\sum_{T\in\mathfrak{P}\left(  \left[  n\right]
\smallsetminus\left[  1\right]  \right)  }(p_{T}-\frac{p_{\{1\}\cup T\}}%
}{p_{1}})\boldsymbol{\theta}_{\left[  n\right]  \smallsetminus\left[
1\right]  }^{T}}{S_{\left[  n\right]  \smallsetminus\left[  1\right]  }\left(
\boldsymbol{\theta}_{\left[  n\right]  \smallsetminus\left[  1\right]
}\right)  }]\frac{x_{1}}{p_{1}}\}\label{LXn_1_Yn_1s_d}\\
&  =S_{\left[  n\right]  \smallsetminus\left[  1\right]  }^{-\lambda}\left(
\boldsymbol{\theta}_{\left[  n\right]  \smallsetminus\left[  1\right]
}\right)\times \nonumber\\
&  \mathbf{e}^{\{\frac{p_{\left[  n\right]  }}{p_{1}}x_{1}\sum
_{T\subset\left[  n\right]  \ \smallsetminus\left\{  1\right\}  ,0<\left\vert
T\right\vert }r_{\left\{  1\right\}  \cup T}\left(  -\widetilde{\mathbf{p}%
}\right)  ^{\left\{  \left[  n\right]  \smallsetminus\left\{  1\right\}
\right\}  \smallsetminus T}[\mathbf{S}^{\left[  n\right]  \smallsetminus
\left\{  1\right\}  \smallsetminus T}\left(  \boldsymbol{\theta}_{\left[
n\right]  \mathbf{\smallsetminus}\left[  1\right]  }\right)  S_{\left[
n\right]  \smallsetminus\left[  1\right]  }^{-1}\left(  \boldsymbol{\theta
}_{\left[  n\right]  \mathbf{\smallsetminus}\left[  1\right]  }\right)
-1]\}}\label{LXn_1_X1_Yn_1s_d_main}
\end{align}
We denote by $\mathbf{S}=\left(  S_{1},\ldots,S_{n}\right)  ,$ with
$S_{i}\left(  \theta_{i}\right)  =1+\left(  -\widetilde{p}_{i}\right)
^{-1}\theta_{i}$, $i=1,\ldots,n.$ For simplicity, we denote $L_{\mathbf{X}%
_{\left[  n\right]  \smallsetminus\left[  1\right]  }}^{X_{1}=x_{1}}\left(
\boldsymbol{\theta}_{\left[  n\right]  \smallsetminus\left[  1\right]
}\right)  $ by $L_{\mathbf{X}_{\left[  n\right]  \smallsetminus\left[
1\right]  }}^{X_{1}=x_{1}}$, $S_{\left[  n\right]  \smallsetminus\left[
1\right]  }\left(  \boldsymbol{\theta}_{\left[  n\right]
\mathbf{\smallsetminus}\left[  1\right]  }\right)  $ by $S_{\left[  n\right]
\smallsetminus\left[  1\right]  }$, and $\mathbf{S}\left(  \boldsymbol{\theta
}_{\left[  n\right]  \mathbf{\smallsetminus}\left[  1\right]  }\right)  $ by
$\mathbf{S}$ and we can write
\begin{equation}
L_{\mathbf{X}_{\left[  n\right]  \smallsetminus\left[  1\right]  }}%
^{X_{1}=x_{1}}=S_{\left[  n\right]  \smallsetminus\left[  1\right]
}^{-\lambda}\mathbf{e}^{\{(-\widetilde{p}_{\left[  n\right]
\mathbf{\smallsetminus}\left[  1\right]  })^{-1}x_{1}\sum_{T\subset\left[
n\right]  \ \smallsetminus\left\{  1\right\}  ,0<\left\vert T\right\vert
}r_{\left\{  1\right\}  \cup T}\left(  -\widetilde{\mathbf{p}}\right)
^{\left\{  \left[  n\right]  \smallsetminus\left\{  1\right\}  \right\}
\smallsetminus T}[\mathbf{S}^{\left[  n\right]  \smallsetminus\left\{
1\right\}  \smallsetminus T}S_{\left[  n\right]  \smallsetminus\left[
1\right]  }^{-1}-1]\}}. \label{LXn_1_X1_Yn_1s_d_main_s}%
\end{equation}
We can also write%
\begin{equation}
L_{\mathbf{X}_{\left[  n\right]  \smallsetminus\left[  1\right]  }}%
^{X_{1}=x_{1}}=S_{\left[  n\right]  \smallsetminus\left[  1\right]
}^{-\lambda}\mathbf{e}^{\{(-\widetilde{p}_{\left[  n\right]
\mathbf{\smallsetminus}\left[  1\right]  })^{-1}x_{1}[\sum_{T\subset\left[
n\right]  \ \smallsetminus\left\{  1\right\}  ,0<\left\vert T\right\vert
}r_{\left\{  1\right\}  \cup T}\left(  -\widetilde{\mathbf{p}}\right)
^{\left\{  \left[  n\right]  \smallsetminus\left\{  1\right\}  \right\}
\smallsetminus T}\mathbf{S}^{\left[  n\right]  \smallsetminus\left\{
1\right\}  \smallsetminus T}S_{\left[  n\right]  \smallsetminus\left[
1\right]  }^{-1}-C]\}}, \label{LXn_1_X1_Yn_1s_d_main_s_C}%
\end{equation}
where $C=\sum_{T\subset\left[  n\right]  \ \smallsetminus\left\{  1\right\}
,0<\left\vert T\right\vert }r_{\left\{  1\right\}  \cup T}\left(
-\widetilde{\mathbf{p}}\right)  ^{\left\{  \left[  n\right]  \smallsetminus
\left\{  1\right\}  \right\}  \smallsetminus T}$ is such that $L_{\mathbf{X}%
_{\left[  n\right]  \smallsetminus\left[  1\right]  }}^{X_{1}=x_{1}}\left(
\mathbf{0}_{n-1}\right)  =1.$
\end{theorem}

\section{Applications of the main results}

\subsection{A particular case in the case $k=1$}

Now let us apply the result of Corollary \ref{cor_cor} to get the following Corollary.

\begin{corollary}
\label{cor1}For $P_{n}\left(  \boldsymbol{\theta}\right)  =\frac{-q}{p}%
+\frac{1}{p}\prod_{i=1}^{n}\left(  1+p\theta_{i}\right)  ,$ we have
\begin{equation}
L_{\mathbf{X}_{\left[  n\right]  \mathbf{\smallsetminus}\left[  k\right]  }%
}^{\mathbf{X}_{\left[  k\right]  }=\mathbf{x}_{\left[  k\right]  }}\left(
\boldsymbol{\theta}_{\left[  n\right]  \mathbf{\smallsetminus}\left[
k\right]  }\right)  =[\prod_{i=k+1}^{n}\left(  1+p\theta_{i}\right)
]^{-\lambda}\frac{F_{k-1}\left(  \lambda,\ldots,\lambda,qp^{-k}\mathbf{x}%
_{k}^{\left[  k\right]  }\prod_{i=k+1}^{n}\left(  1+p\theta_{i}\right)
^{-1}\right)  }{F_{k-1}\left(  \lambda,\ldots,\lambda,qp^{-k}\mathbf{x}%
_{k}^{\left[  k\right]  }\right)  } \label{LCp}%
\end{equation}

\end{corollary}

The case $k=1$ is simpler because $F_{0}=\exp,$ so we can give the following corollary.

\begin{corollary}
\label{cor1_k_1}For $n>1,k=1,$ we have%
\begin{equation}
L_{\mathbf{X}_{\left[  n\right]  \mathbf{\smallsetminus}\left[  1\right]  }%
}^{X_{1}=x_{1}}\left(  \boldsymbol{\theta}_{\left[  n\right]
\mathbf{\smallsetminus}\left[  1\right]  }\right)  =[\prod_{i=2}^{n}\left(
1+p\theta_{i}\right)  ]^{-\lambda}\exp\{qp^{-1}x_{1}[\prod_{i=2}^{n}\left(
1+p\theta_{i}\right)  ^{-1}-1]\}, \label{LCp_k_1}%
\end{equation}
or%
\begin{equation}
L_{\mathbf{X}_{\left[  n\right]  \mathbf{\smallsetminus}\left[  1\right]  }%
}^{X_{1}=x_{1}}\left(  \boldsymbol{\theta}_{\left[  n\right]
\mathbf{\smallsetminus}\left[  1\right]  }\right)  =\sum_{k=0}^{\infty}%
\frac{\left(  qp^{-1}x_{1}\right)  ^{k}}{k!}\exp\left(  -qp^{-1}x_{1}\right)
[\prod_{i=2}^{n}\left(  1+p\theta_{i}\right)  ]^{-\left(  \lambda+k\right)  },
\label{LCp_k_1_2}%
\end{equation}

\end{corollary}

As a result, Formula (\ref{LCp_k_1_2}) gives a simulation of $\mathbf{X}%
_{\left[  n\right]  }$. Let us denote by $\mathcal{P}\left(  \mu\right)  $ the
Poisson distribution (Pd) of parameter $\mu$, we derive the following theorem.

\begin{theorem}
\label{Thequicor}Let $X_{1}\sim\gamma_{\left(  p,\lambda\right)  },$ let
$V_{1}\sim\mathcal{P}\left(  qp^{-1}X_{1}\right)  ,$ let $\mathbf{X}_{\left[
n\right]  \smallsetminus\left[  1\right]  }=\left(  X_{2},\ldots,X_{n}\right)
,$ and\newline$\mathbf{X}_{\left[  n\right]  \smallsetminus\left[  1\right]
}|\left(  X_{1}=x_{1}\right)  \sim\mathbf{\gamma}_{(\prod_{i=2}^{n}\left(
1+p\boldsymbol{\theta}_{i}\right)  ,\lambda+V_{1})}$, then $\mathbf{X}%
_{\left[  n\right]  }=(X_{1},X_{2},\ldots,X_{n})\sim\boldsymbol{\gamma
}_{(P_{n},\lambda)},$ with $P_{n}\left(  \boldsymbol{\theta}_{\left[
n\right]  }\right)  =\frac{-q}{p}+\frac{1}{p}\prod_{i=1}^{n}\left(
1+p\theta_{i}\right)  .$
\end{theorem}

We derive the following algorithm to simulate $\mathbf{X}_{\left[
n\right]  }\sim\boldsymbol{\gamma}_{(P_{n},\lambda)}.$

\begin{algorithm}
\label{Al14Ber(2023) copy(1)}Simulation of an infinitely divisible mgd
$\mathbf{\gamma}_{\left(  P_{n},\lambda\right)  },$ with $P_{n}\left(
\boldsymbol{\theta}_{\left[  n\right]  }\right)  =\frac{-q}{p}+\frac{1}%
{p}\prod_{i=1}^{n}\left(  1+p\theta_{i}\right)  $

\begin{enumerate}
\item Simulate $X_{1}\sim\gamma_{(p,\lambda)}$

\item Simulate $V_{1}\sim\mathcal{P}\left(  qp^{-1}X_{1}\right)  $

\item Simulate independently $X_{i}\sim\gamma_{(p,\lambda+V_{1})}$

\item Then $\mathbf{X}_{\left[  n\right]  }=(X_{1},X_{2},\ldots,X_{n})$
simulate $\boldsymbol{\gamma}_{(P_{n},\lambda)}.$
\end{enumerate}
\end{algorithm}

\subsection{The case $n=2$ and $k=1$}

In this case, we give another proof of Theorem 14 in \cite{Bernardoff(2023)}:

\begin{theorem}
\label{Th14Ber(2023)}Let $P_{2}\left(  \theta_{1},\theta_{2}\right)
=1+p_{1}\theta_{1}+p_{2}\theta_{2}+p_{1,2}\theta_{1}\theta_{2}$, with
$p_{1}>0,$ $p_{2}>0$, $p_{1,2}>0$ $\widetilde{b}_{1,2}=\widetilde{b}%
_{1,2}\left(  P_{2}\right)  =p_{1}p_{2}/p_{1,2}^{2}-1/p_{1,2}>0$ Let
$P_{1}\left(  \theta_{1}\right)  =P_{2}\left(  \theta_{1},0\right)  ,$
$X_{1}\sim\gamma_{\left(  P_{1},\lambda\right)  }.$ Let $\alpha_{1}%
=\frac{\widetilde{b}_{1,2}}{(-\widetilde{p}_{2})}$, and $V_{1}\sim
\mathcal{P}\left(  \alpha_{1}X_{1}\right)  $. We have $S_{2}\left(  \theta
_{2}\right)  =1+\frac{p_{1,2}}{p_{1}}\theta_{2}$. Let $X_{2}\sim
\gamma_{\left(  S_{2},\lambda+V_{1}\right)  }$, then $\mathbf{X}_{\left[
2\right]  }=\left(  X_{1},X_{2}\right)  \sim\mathbf{\gamma}_{\left(
P_{2},\lambda\right)  }.$
\end{theorem}

We derive the following algorithm to simulate $\mathbf{X}_{\left[
2\right]  }=\left(  X_{1},X_{2}\right)  \sim\mathbf{\gamma}_{\left(
P_{2},\lambda\right)  }$, see \cite{Bernardoff(2023)}:

\begin{algorithm}
\label{Al14Ber(2023)}Simulation of an infinitely divisible bgd $\mathbf{\gamma
}_{\left(  P_{2},\lambda\right)  }$

\begin{enumerate}
\item Simulate $X_{1}\sim\gamma_{(p_{1},\lambda)}$

\item Simulate $V_{1}\sim\mathcal{P}(\frac{p_{1,2}}{p_{1}}\widetilde{b}%
_{1,2}X_{1})$

\item Simulate $X_{2}\sim\gamma_{(\frac{p_{1,2}}{p_{1}},\lambda+V_{1})}$

\item Then $\mathbf{X}_{\left[  2\right]  }=\left(  X_{1},X_{2}\right)  $
simulates $\mathbf{\gamma}_{\left(  P_{2},\lambda\right)  }.$
\end{enumerate}
\end{algorithm}

\subsection{The case $n=3$ and $k=1$}

In this case, we give the following theorem

\begin{theorem}
\label{Th1_n=3_Bern2024}Let $P_{3}\left(  \theta_{1},\theta_{2},\theta
_{3}\right)  =1+p_{1}\theta_{1}+p_{2}\theta_{2}+p_{3}\theta_{3}+p_{1,2}%
\theta_{1}\theta_{2}+p_{1,3}\theta_{1}\theta_{3}+p_{2,3}\theta_{2}\theta
_{3}+p_{1,2,3}\theta_{1}\theta_{2}\theta_{3}$ with $p_{i}>0,i=1,2,3,$
$p_{1,2}>0,p_{2,3}>0$, $p_{1,3}>0$, $p_{1,2,3}>0$, and $\widetilde{b}%
_{1,2}\left(  P_{3}\right)  =p_{1,3}p_{2,3}/p_{1,2,3}^{2}-p_{3}/p_{1,2,3}>0$,
$\widetilde{b}_{1,3}\left(  P_{3}\right)  =p_{1,2}p_{2,3}/p_{1,2,3}^{2}%
-p_{2}/p_{1,2,3}>0,$ $\widetilde{b}_{2,3}\left(  P_{3}\right)  =p_{1,3}%
p_{1,2}/p_{1,2,3}^{2}-p_{1}/p_{1,2,3}>0$ ($\widetilde{p}_{1}=-p_{2,3}%
/p_{1,2,3}<0$, $\widetilde{p}_{2}=-p_{1,3}/p_{1,2,3}<0$,$\widetilde{p}%
_{3}=-p_{1,2}/p_{1,2,3}<0$) and $\widetilde{b}_{1,2,3}\left(  P_{3}\right)
=\widetilde{p}_{1,2,3}+\widetilde{p}_{1}\widetilde{p}_{2,3}+\widetilde{p}%
_{2}\widetilde{p}_{1,3}+\widetilde{p}_{3}\widetilde{p}_{1,2}+2\widetilde{p}%
_{1}\widetilde{p}_{2}\widetilde{p}_{3}=-1/p_{1,2,3}+p_{2,3}p_{1}/p_{1,2,3}%
^{2}+p_{1,3}p_{2}/p_{1,2,3}^{2}+p_{1,2}p_{3}/p_{1,2,3}^{2}-2p_{2,3}%
p_{1,3}p_{1,2}/p_{1,2,3}^{3}>0$. Let $P_{1}\left(  \theta_{1}\right)
=P_{3}\left(  \theta_{1},0,0\right)  =1+p_{1}\theta_{1},$ let $X_{1}\sim
\gamma_{\left(  P_{1},\lambda\right)  }.$ Let $S_{2,3}\left(  \theta
_{2},\theta_{3}\right)  =1+\frac{p_{1,2}}{p_{1}}\theta_{2}+\frac{p_{1,3}%
}{p_{1}}\theta_{3}+\frac{p_{1,2,3}}{p_{1}}\theta_{2}\theta_{3},S_{2}\left(
\theta_{2}\right)  =1+\left(  -\widetilde{p}_{2}\right)  ^{-1}\theta_{2}$ and
$S_{3}\left(  \theta_{3}\right)  =1+\left(  -\widetilde{p}_{3}\right)
^{-1}\theta_{3}.$ Let $\alpha_{1}=\frac{\widetilde{b}_{1,2}}{(-\widetilde{p}%
_{2})},\alpha_{2}=\frac{\widetilde{b}_{1,3}}{(-\widetilde{p}_{3})},\alpha
_{3}=\frac{\widetilde{b}_{1,2,3}}{(-\widetilde{p}_{2,3})},\alpha_{4}%
=\frac{\widetilde{b}_{1,2}\widetilde{b}_{2,3}}{(-\widetilde{p}_{2}%
)(-\widetilde{p}_{2,3})},\alpha_{5}=\frac{\widetilde{b}_{1,3}\widetilde{b}%
_{2,3}}{(-\widetilde{p}_{3})(-\widetilde{p}_{2,3})}$, and $V_{i}%
\sim\mathcal{P}\left(  \alpha_{i}X_{1}\right)  ,i\in\left[  5\right]  $
independent, with the notation $\mathbf{v}=\left(  v_{1},\ldots v_{5}\right)
\in\mathbb{N},$ we have%
\begin{equation}
L_{(X_{2},X_{3})}^{X_{1}=x_{1}}=\sum_{\mathbf{v}\in\mathbb{N}^{5}}\prod
_{i=1}^{5}\mathbf{P}\left(  V_{i}=v_{i}\right)  S_{2}^{-(v_{1}+v_{4})}%
S_{3}^{-(v_{2}+v_{5})}S_{2,3}^{-(\lambda+v_{3}+v_{4}+v_{5})}. \label{LCn=3}%
\end{equation}
Let $Y_{2}\sim\gamma_{\left(  S_{2},V_{1}+V_{4}\right)  },Y_{3}\sim
\gamma_{\left(  S_{3},V_{2}+V_{5}\right)  },\left(  Z_{2},Z_{3}\right)
\sim\mathbf{\gamma}_{\left(  S_{2,3},\lambda+V_{3}+V_{4}+V_{5}\right)  }$
independent, and let $\left(  X_{2},X_{3}\right)  =\left(  Y_{2}+Z_{2}%
,Y_{3}+Z_{3}\right)  $, then $\mathbf{X}_{\left[  3\right]  }=\left(
X_{1},X_{2},X_{3}\right)  \sim\mathbf{\gamma}_{\left(  P_{3},\lambda\right) }.$
\end{theorem}

From Theorem \ref{Th1_n=3_Bern2024}, we derive the following algorithm for
simulating $\mathbf{X}_{\left[  3\right]  }=\left(  X_{1},X_{2},X_{3}\right)
\sim\mathbf{\gamma}_{\left(  P_{3},\lambda\right)  }$

\begin{algorithm}
\label{Al1_n=3_Ber(2024)}Simulation of an infinitely divisible tgd $\mathbf{\gamma}_{\left(  P_{3},\lambda\right)  }$

\begin{enumerate}
\item Simulate $X_{1}\sim\gamma_{\left(  p_{1},\lambda\right)  .}$

\item Compute $\alpha_{i},i\in\left[  5\right]  ,$ defined
in Theorem \ref{Th1_n=3_Bern2024}. Simulate independently
$V_{i}\sim\mathcal{P}\left(  \alpha_{i}X_{1}\right) .$

\item Simulate independently $Y_{2}\sim\gamma_{\left(  \left(  -\widetilde{p}%
_{2}\right)  ^{-1},V_{1}+V_{4}\right)  },Y_{3}\sim\gamma_{\left(  \left(
-\widetilde{p}_{3}\right)  ^{-1},V_{2}+V_{5}\right)  },\left(  Z_{2}%
,Z_{3}\right)  \sim\mathbf{\gamma}_{\left(  S_{2,3},\lambda+V_{3}+V_{4}%
+V_{5}\right)  }$.

\item Then $\mathbf{X}_{\left[  3\right]  }  =\left(  X_{1}%
,Y_{2}+Z_{2},Y_{3}+Z_{3}\right)  $ simulates $\mathbf{\gamma}_{\left(  P_{3},\lambda\right)}.$
\end{enumerate}
\end{algorithm}

We can notice that $\boldsymbol{W=}\left(  W_{1},W_{2},W_{3}\right)  =\left(
V_{1}+V_{4},V_{2}+V_{5},V_{3}+V_{4}+V_{5}\right)  $ satisfied for
$\boldsymbol{t}=\left(  t_{1},t_{2},t_{3}\right)  \in\left(  0,\infty\right)
^{3},$
$
\mathbb{E}\left(  \boldsymbol{t}^{\boldsymbol{W}}|X_{1}=x_{1}\right)
=\exp[(\alpha_{1}t_{1}+\alpha_{2}t_{2}+\alpha_{3}t_{3}+\alpha_{4}t_{1}%
t_{3}+\alpha_{5}t_{2}t_{3})x_{1}-(\alpha_{1}+\alpha_{2}+\alpha_{3}+\alpha
_{4}+\alpha_{5})x_{1}],
$
and $\boldsymbol{W}|X_{1}=x_{1}$ is a trivariate Poisson distribution, see
\cite{DwassTeicher(1957)}.

From Theorem (\ref{Th14Ber(2023)}), we derive the following theorem and algorithm.

\begin{theorem}
\label{Th2_n=3_Bern2024}Let $X_{1}\sim\gamma_{\left(  p_{1},\lambda\right)  }%
$. 
Let $\alpha_{i},i\in\left[  5\right]  ,$ defined
in Theorem \ref{Th1_n=3_Bern2024}, $\alpha_{6}=\frac
{\widetilde{b}_{2,3}}{(-\widetilde{p}_{3})}$, and $V_{i}\sim\mathcal{P}\left(
\alpha_{i}X_{1}\right)  ,i\in\left[  5\right]  $. Let $Z_{2}^{\prime}%
\sim\gamma_{\left(  \frac{p_{1,2}}{p_{1}},\lambda+V_{3}+V_{4}+V_{5}\right)  }$
and $V_{6}\sim\mathcal{P}\left(  \alpha_{6}Z_{2}^{\prime}\right)  $, let
$Y_{2}\sim\gamma_{\left(  (-\widetilde{p}_{2})^{-1},V_{1}+V_{4}\right)
},Y_{3}\sim\gamma_{\left(  -\widetilde{p}_{3})^{-1},V_{2}+V_{5}\right)
},$ $Z_{3}^{\prime}\sim\gamma_{\left(  -\widetilde{p}_{3})^{-1},\lambda
+V_{3}+V_{4}+V_{5}+V_{6}\right)  }$ independent, we have $\left(
Z_{2}^{\prime},Z_{3}^{\prime}\right)  \sim\mathbf{\gamma}_{\left(
S_{2,3},\lambda+V_{3}+V_{4}+V_{5}\right)  }$, let $\left(X_{2},X_{3}\right)=\left(
Y_{2}+Z_{2}^{^{\prime}},Y_{3}+Z_{3}^{\prime}\right)  ,$ then $\mathbf{X}%
_{\left[  3\right]  }=\left(  X_{1},X_{2},X_{3}\right)  \sim\mathbf{\gamma
}_{\left(  P_{3},\lambda\right)  }.$
\end{theorem}

\begin{algorithm}
\label{Al2_n=3_Ber(2024)}Simulation of an infinitely divisible tgd

\begin{enumerate}
\item Simulate $X_{1}\sim\gamma_{\left(  p_{1},\lambda\right)  .}.$

\item Compute $\alpha_{i},i\in\left[  6\right]  ,$ defined
in Theorem \ref{Th2_n=3_Bern2024}. Simulate independently
$V_{i}\sim\mathcal{P}\left(  \alpha_{i}X_{i}\right)  ,i\in\left[  5\right]  .$

\item Simulate $Z_{2}^{\prime}\sim\gamma_{(  \frac{p_{1,2}}{p_{1}},\lambda+V_{3}+V_{4}+V_{5})  }.$

\item Simulate $V_{6}\sim\mathcal{P}\left(  \alpha_{6}Z_{2}^{\prime}\right)
.$

\item Simulate independently $Y_{2}\sim\gamma_{\left(  \left(  -\widetilde{p}%
_{2}\right)  ^{-1},V_{4}+V_{5}\right)  },Y_{3}\sim\gamma_{\left(  \left(
-\widetilde{p}_{3}\right)  ^{-1},V_{2}+V_{5}\right)  },Z_{3}^{\prime}%
\sim\gamma_{\left(  \left(  -\widetilde{p}_{3}\right)  ^{-1},\lambda
+V_{3}+V_{4}+V_{5}+V_{6}\right)  }.$

\item Then $\mathbf{X}_{\left[3\right]}=\left(  X_{1},X_{2},X_{3}\right)  =\left(  X_{1}%
,Y_{2}+Z_{2}^{\prime},Y_{3}+Z_{3}^{\prime}\right)  $ simulates $\mathbf{\gamma
}_{\left(  P_{3},\lambda\right)  }.$
\end{enumerate}
\end{algorithm}

We can give the following remark.

\begin{remark}
In Theorem (\ref{Th1_n=3_Bern2024}) or Algorithm (\ref{Al1_n=3_Ber(2024)}), we
use 3 univariate gamma distributions, 5 Poisson distributions and one
bivariate distribution. In Theorem (\ref{Th2_n=3_Bern2024}) or Algorithm
(\ref{Al2_n=3_Ber(2024)}), we use 5 univariate gamma distributions and 6
Poisson distributions. For the simulations themselves, we can use either
method with the R software, \cite{RCoreTeam}.
\end{remark}

\subsection{The case $n=4$ and $k=1$}

In this case, we give the following theorem

\begin{theorem}
\label{Th1_n=4_Bern2024}Let $P_{4}\left(  \boldsymbol{\theta}_{4}\right)
=\sum_{T\in\mathfrak{P}_{4}}p_{T}\boldsymbol{\theta}_{4}^{T}$ with
$p_{T}>0,T\in\mathfrak{P}_{4}^{\ast}$, and $\widetilde{b}_{T}\left(
P_{4}\right)  >0,T\in\mathfrak{P}_{4}^{\ast},|T|>1$, Let $P_{1}\left(
\theta_{1}\right)  =P_{4}\left(  \theta_{1},0,0,0\right)  =1+p_{1}\theta_{1},$
let $X_{1}\sim\gamma_{\left(  P_{1},\lambda\right)  }.$ Let $S_{T}%
,T\in\mathfrak{P}_{4}^{\ast}$ defined by (\ref{STdef}), so that $S_{2}\left(
\theta_{2}\right)  =1+\left(  -\widetilde{p}_{2}\right)  ^{-1}\theta_{2}$ and
$S_{3}\left(  \theta_{3}\right)  =1+\left(  -\widetilde{p}_{3}\right)
^{-1}\theta_{3},$ $S_{4}\left(  \theta_{4}\right)  =1+\left(  -\widetilde{p}%
_{4}\right)  ^{-1}\theta_{4},$ $S_{2,3}\left(  \theta_{2},\theta_{3}\right)
=1+\frac{p_{1,2,4}}{p_{1,4}}\theta_{2}+\frac{p_{1,3,4}}{p_{1,4}}\theta
_{3}+\frac{p_{1,2,3,4}}{p_{1,4}}\theta_{2}\theta_{3},$ $S_{2,4}\left(
\theta_{2},\theta_{4}\right)  =1+\frac{p_{1,2,3}}{p_{1,3}}\theta_{2}%
+\frac{p_{1,3,4}}{p_{1,3}}\theta_{4}+\frac{p_{1,2,3,4}}{p_{1,3}}\theta
_{2}\theta_{4},$ $S_{3,4}\left(  \theta_{3},\theta_{4}\right)  =1+\frac
{p_{1,2,3}}{p_{1,2}}\theta_{3}+\frac{p_{1,2,4}}{p_{1,2}}\theta_{4}%
+\frac{p_{1,2,3,4}}{p_{1,2}}\theta_{3}\theta_{4},$ and $S_{2,3,4}\left(
\theta_{2},\theta_{3},\theta_{4}\right)  =1+\frac{p_{1,2}}{p_{1}}\theta
_{2}+\frac{p_{1,3}}{p_{1}}\theta_{3}+\frac{p_{1,4}}{p_{1}}\theta_{4}%
+\frac{p_{1,2,3}}{p_{1}}\theta_{2}\theta_{3}+\frac{p_{1,2,4}}{p_{1}}\theta
_{2}\theta_{4}+\frac{p_{1,3,4}}{p_{1}}\theta_{3}\theta_{4}+\frac{p_{1,2,3,4}%
}{p_{1}}\theta_{2}\theta_{3}\theta_{4}.$ \newline Let $
\alpha_{1}  =\tfrac{\widetilde{b}_{1,2}}{\left(  -\widetilde{p}_{2}\right)
},\alpha_{2}=\tfrac{\widetilde{b}_{1,3}}{\left(
-\widetilde{p}_{3}\right)  },\alpha_{3}=\tfrac{\widetilde{b}_{1,4}}{\left(
-\widetilde{p}_{4}\right)  },\alpha_{4}=\tfrac{\widetilde{b}_{1,2,3}}{\left(
-\widetilde{p}_{2,3}\right)  },\alpha_{5}=\tfrac{\widetilde{b}_{1,2,4}%
}{\left(  -\widetilde{p}_{2,4}\right)  },\alpha_{6}=\tfrac{\widetilde{b}%
_{1,3,4}}{\left(  -\widetilde{p}_{3,4}\right)  },
\alpha_{7} =\tfrac{\widetilde{b}_{1,2}\widetilde{b}_{2,3}}{\left(
-\widetilde{p}_{2}\right)  \left(  -\widetilde{p}_{2,3}\right)  },\\ \alpha
_{8}=\tfrac{\widetilde{b}_{1,2}\widetilde{b}_{2,4}}{\left(  -\widetilde{p}%
_{2}\right)  \left(  -\widetilde{p}_{2,4}\right)  },\alpha_{9}=\tfrac
{\widetilde{b}_{1,3}\widetilde{b}_{2,3}}{\left(  -\widetilde{p}_{3}\right)
\left(  -\widetilde{p}_{2,3}\right)  },\alpha_{10}=\tfrac{\widetilde{b}%
_{1,3}\widetilde{b}_{3,4}}{\left(  -\widetilde{p}_{3}\right)  \left(
-\widetilde{p}_{3,4}\right)  },
\alpha_{11}    =\tfrac{\widetilde{b}_{1,4}\widetilde{b}_{2,4}}{\left(
-\widetilde{p}_{4}\right)  \left(  -\widetilde{p}_{2,4}\right)  },\alpha
_{12}=\tfrac{\widetilde{b}_{1,4}\widetilde{b}_{3,4}}{\left(  -\widetilde{p}%
_{4}\right)  \left(  -\widetilde{p}_{3,4}\right)  },\\ \alpha_{13}=\tfrac
{\widetilde{b}_{1,2,3,4}}{\left(  -\widetilde{p}_{2,3,4}\right)  },\alpha
_{14}=\tfrac{\widetilde{b}_{1,2}\widetilde{b}_{2,3,4}+\widetilde{b}%
_{2,3}\widetilde{b}_{1,2,4}+\widetilde{b}_{2,4}\widetilde{b}_{1,2,3}}{\left(
-\widetilde{p}_{2}\right)  \left(  -\widetilde{p}_{2,3,4}\right)  },
\alpha_{15}   =\tfrac{\widetilde{b}_{1,3}\widetilde{b}_{2,3,4}%
+\widetilde{b}_{2,3}\widetilde{b}_{1,3,4}+\widetilde{b}_{3,4}\widetilde{b}%
_{1,2,3}}{\left(  -\widetilde{p}_{3}\right)  \left(  -\widetilde{p}%
_{2,3,4}\right)  },\alpha_{16}=\tfrac{\widetilde{b}_{1,4}\widetilde{b}%
_{2,3,4}+\widetilde{b}_{2,4}\widetilde{b}_{1,3,4}+\widetilde{b}_{3,4}%
\widetilde{b}_{1,2,4}}{\left(  -\widetilde{p}_{4}\right)  \left(
-\widetilde{p}_{2,3,4}\right)  },\\
\alpha_{17} =\tfrac{2\widetilde{b}_{1,2}\widetilde{b}_{2,3}\widetilde{b}%
_{2,4}}{\left(  -\widetilde{p}_{2}\right)  ^{2}\left(  -\widetilde{p}%
_{2,3,4}\right)  },\alpha_{18}=\tfrac{2\widetilde{b}_{1,3}\widetilde{b}%
_{2,3}\widetilde{b}_{3,4}}{\left(  -\widetilde{p}_{3}\right)  ^{2}\left(
-\widetilde{p}_{2,3,4}\right)  },\alpha_{19}=\tfrac{2\widetilde{b}%
_{1,4}\widetilde{b}_{2,4}\widetilde{b}_{3,4}}{\left(  -\widetilde{p}%
_{4}\right)  ^{2}\left(  -\widetilde{p}_{2,3,4}\right)  },
\alpha_{20}   =\tfrac{\widetilde{b}_{2,3}\left(  \widetilde{b}%
_{1,2}\widetilde{b}_{3,4}+\widetilde{b}_{1,3}\widetilde{b}_{2,4}\right)
}{\left(  -\widetilde{p}_{2}\right)  \left(  -\widetilde{p}_{3}\right)
\left(  -\widetilde{p}_{2,3,4}\right)  },\\ \alpha_{21}=\tfrac{\widetilde{b}%
_{2,4}\left(  \widetilde{b}_{1,2}\widetilde{b}_{3,4}+\widetilde{b}%
_{1,4}\widetilde{b}_{2,3}\right)  }{\left(  -\widetilde{p}_{2}\right)  \left(
-\widetilde{p}_{4}\right)  \left(  -\widetilde{p}_{2,3,4}\right)  }%
,\alpha_{22}=\tfrac{\widetilde{b}_{3,4}\left(  \widetilde{b}_{1,3}%
\widetilde{b}_{2,4}+\widetilde{b}_{1,4}\widetilde{b}_{2,3}\right)  }{\left(
-\widetilde{p}_{3}\right)  \left(  -\widetilde{p}_{4}\right)  \left(
-\widetilde{p}_{2,3,4}\right)  },
\alpha_{23}    =\tfrac{\widetilde{b}_{1,2,3}\widetilde{b}_{2,3,4}}{\left(
-\widetilde{p}_{2,3}\right)  \left(  -\widetilde{p}_{2,3,4}\right)  }%
,\alpha_{24}=\tfrac{\widetilde{b}_{1,2,4}\widetilde{b}_{2,3,4}}{\left(
-\widetilde{p}_{2,4}\right)  \left(  -\widetilde{p}_{2,3,4}\right)  }%
,\\ \alpha_{25}=\tfrac{\widetilde{b}_{1,3,4}\widetilde{b}_{2,3,4}}{\left(
-\widetilde{p}_{3,4}\right)  \left(  -\widetilde{p}_{2,3,4}\right)  },
\alpha_{26}  =\tfrac{\widetilde{b}_{2,3}\left(  \widetilde{b}%
_{1,2}\widetilde{b}_{2,3,4}+\widetilde{b}_{2,4}\widetilde{b}_{1,2,3}\right)
}{\left(  -\widetilde{p}_{2}\right)  \left(  -\widetilde{p}_{2,3}\right)
\left(  -\widetilde{p}_{2,3,4}\right)  },\alpha_{27}=\tfrac{\widetilde{b}%
_{2,4}\left(  \widetilde{b}_{1,2}\widetilde{b}_{2,3,4}+\widetilde{b}%
_{2,3}\widetilde{b}_{1,2,4}\right)  }{\left(  -\widetilde{p}_{2}\right)
\left(  -\widetilde{p}_{2,4}\right)  \left(  -\widetilde{p}_{2,3,4}\right)
},\alpha_{28}=\tfrac{\widetilde{b}_{2,3}\left(  \widetilde{b}_{1,3}%
\widetilde{b}_{2,3,4}+\widetilde{b}_{3,4}\widetilde{b}_{1,2,3}\right)
}{\left(  -\widetilde{p}_{3}\right)  \left(  -\widetilde{p}_{2,3}\right)
\left(  -\widetilde{p}_{2,3,4}\right)  },\\
\alpha_{29}    =\tfrac{\widetilde{b}_{3,4}\left(  \widetilde{b}%
_{1,3}\widetilde{b}_{2,3,4}+\widetilde{b}_{2,3}\widetilde{b}_{1,3,4}\right)
}{\left(  -\widetilde{p}_{3}\right)  \left(  -\widetilde{p}_{3,4}\right)
\left(  -\widetilde{p}_{2,3,4}\right)  },\alpha_{30}=\tfrac{\widetilde{b}%
_{2,4}\left(  \widetilde{b}_{1,4}\widetilde{b}_{2,3,4}+\widetilde{b}%
_{3,4}\widetilde{b}_{1,2,4}\right)  }{\left(  -\widetilde{p}_{4}\right)
\left(  -\widetilde{p}_{2,4}\right)  \left(  -\widetilde{p}_{2,3,4}\right)
},\alpha_{31}=\tfrac{\widetilde{b}_{3,4}\left(  \widetilde{b}_{1,4}%
\widetilde{b}_{2,3,4}+\widetilde{b}_{2,4}\widetilde{b}_{1,3,4}\right)
}{\left(  -\widetilde{p}_{4}\right)  \left(  -\widetilde{p}_{3,4}\right)
\left(  -\widetilde{p}_{2,3,4}\right)  },
\alpha_{32}    =\tfrac{\widetilde{b}_{1,2}\widetilde{b}_{2,3}^{2}%
\widetilde{b}_{2,4}}{\left(  -\widetilde{p}_{2}\right)  ^{2}\left(
-\widetilde{p}_{2,3}\right)  \left(  -\widetilde{p}_{2,3,4}\right)  }%
,\\ \alpha_{33}=\tfrac{\widetilde{b}_{1,2}\widetilde{b}_{2,3}\widetilde{b}%
_{2,4}^{2}}{\left(  -\widetilde{p}_{2}\right)  ^{2}\left(  -\widetilde{p}%
_{2,4}\right)  \left(  -\widetilde{p}_{2,3,4}\right)  },\alpha_{34}%
=\tfrac{\widetilde{b}_{1,3}\widetilde{b}_{2,3}^{2}\widetilde{b}_{3,4}}{\left(
-\widetilde{p}_{3}\right)  ^{2}\left(  -\widetilde{p}_{2,3}\right)  \left(
-\widetilde{p}_{2,3,4}\right)  },
\alpha_{35}    =\tfrac{\widetilde{b}_{1,3}\widetilde{b}_{2,3}\widetilde{b}%
_{3,4}^{2}}{\left(  -\widetilde{p}_{3}\right)  ^{2}\left(  -\widetilde{p}%
_{3,4}\right)  \left(  -\widetilde{p}_{2,3,4}\right)  },\alpha_{36}%
=\tfrac{\widetilde{b}_{1,4}\widetilde{b}_{3,4}\widetilde{b}_{2,4}^{2}}{\left(
-\widetilde{p}_{4}\right)  ^{2}\left(  -\widetilde{p}_{2,4}\right)  \left(
-\widetilde{p}_{2,3,4}\right)  },\\ \alpha_{37}=\tfrac{\widetilde{b}%
_{1,4}\widetilde{b}_{2,4}\widetilde{b}_{3,4}^{2}}{\left(  -\widetilde{p}%
_{4}\right)  ^{2}\left(  -\widetilde{p}_{3,4}\right)  \left(  -\widetilde{p}%
_{2,3,4}\right)  },
\alpha_{38}    =\tfrac{\widetilde{b}_{2,3}^{2}\left(  \widetilde{b}%
_{1,2}\widetilde{b}_{3,4}+\widetilde{b}_{1,3}\widetilde{b}_{2,4}\right)
}{\left(  -\widetilde{p}_{2}\right)  \left(  -\widetilde{p}_{3}\right)
\left(  -\widetilde{p}_{2,3}\right)  \left(  -\widetilde{p}_{2,3,4}\right)
},\alpha_{39}=\tfrac{\widetilde{b}_{2,4}^{2}\left(  \widetilde{b}%
_{1,2}\widetilde{b}_{3,4}+\widetilde{b}_{1,4}\widetilde{b}_{2,3}\right)
}{\left(  -\widetilde{p}_{2}\right)  \left(  -\widetilde{p}_{4}\right)
\left(  -\widetilde{p}_{2,4}\right)  \left(  -\widetilde{p}_{2,3,4}\right)
},\\ \alpha_{40}=\tfrac{\widetilde{b}_{3,4}^{2}\left(  \widetilde{b}%
_{1,3}\widetilde{b}_{2,4}+\widetilde{b}_{1,4}\widetilde{b}_{2,3}\right)
}{\left(  -\widetilde{p}_{3}\right)  \left(  -\widetilde{p}_{4}\right)
\left(  -\widetilde{p}_{3,4}\right)  \left(  -\widetilde{p}_{2,3,4}\right)  }.
$ Let $V_{i}\sim\mathcal{P}\left(  \alpha_{i}X_{1}\right)  ,i\in\left[
40\right]  ,$ be independent Poisson distributions, we have with the above definitions, with $\mathbf{v}=\left(  v_{1},\ldots,v_{40}\right) 
\in\mathbb{N}^{40} ,$ and 
$z_{1}=v_{1}+v_{7}+v_{8}+v_{14}+2v_{17}+v_{20}+v_{21}+v_{26}%
+v_{27}+2v_{32}+2v_{33}+v_{38}+v_{39}$, 
$z_{2}=v_{2}+v_{9}+v_{10}+v_{15}+2v_{18}+v_{20}+v_{22}+v_{28}%
+v_{29}+2v_{34}+2v_{35}+v_{38}+v_{40}$,
$z_{3}=v_{3}+v_{11}+v_{12}+v_{16}+2v_{19}+v_{21}+v_{22}+v_{30}%
+v_{31}+2v_{36}+2v_{37}+v_{39}+v_{40}$,
$z_{4}=v_{4}+v_{7}+v_{9}+v_{23}+v_{26}+v_{28}+v_{32}+v_{34}+v_{38}$, 
$z_{5}=v_{5}+v_{8}+v_{11}+v_{24}+v_{27}+v_{30}+v_{33}+v_{36}+v_{39}$,
$z_{6}=(v_{6}+v_{10}+v_{12}+v_{25}+v_{29}+v_{31}+v_{35}+v_{37}+v_{40}$,
$z_{7}=\sum_{i=13}^{40}v_{i}.$

\begin{equation}
L_{\left(  X_{2},X_{3},X_{4}\right)  }^{X_{1}=x_{1}}   =\sum_{\mathbf{v}
\in\mathbb{N}^{40}}[\prod_{i=1}^{40}\mathbf{P}\left(  V_{i}=v_{i}\right)
] S_{2}^{-z_{1}} S_{3}^{-z_{2}}  S_{4}^{-(z_{3}}S_{2,3}^{-z_{4}}S_{2,4}^{-z_{5}} S_{3,4}^{-z_{6}}S_{2,3,4}^{-(\lambda+z_{7})}.\label{LCn=4}
\end{equation}

Let the following infinitely divisible independent random vectors defined by%
\begin{align}
X_{1}  &  \sim\gamma_{\left(  P_{1},\lambda\right)  },\label{X1}\\
Y_{2}  &  \sim\gamma_{\left(  S_{2},V_{1}+V_{7}+V_{8}+V_{14}+2V_{17}%
+V_{20}+V_{21}+V_{26}+V_{27}+2V_{32}+2V_{33}+V_{38}+V_{39}\right)
},\label{Y2}\\
Y_{3}  &  \sim\gamma_{\left(  S_{3},V_{2}+V_{9}+V_{10}+V_{15}+2V_{18}%
+V_{20}+V_{22}+V_{28}+V_{29}+2V_{34}+2V_{35}+V_{38}+V_{40}\right)
},\label{Y3}\\
Y_{4}  &  \sim\gamma_{\left(  S_{4},V_{3}+V_{11}+V_{12}+V_{16}+2V_{19}%
+V_{21}+V_{22}+V_{30}+V_{31}+2V_{36}+2V_{37}+V_{39}+V_{40}\right)
},\label{Y4}\\
\left(  U_{1,2},U_{1,3}\right)   &  \sim\boldsymbol{\gamma}_{\left(
S_{2,3},V_{4}+V_{7}+V_{9}+V_{23}+V_{26}+V_{28}+V_{32}+V_{34}+V_{38}\right)
},\label{U1,2,U1,3}\\
\left(  U_{2,2},U_{2,4}\right)   &  \sim\boldsymbol{\gamma}_{\left(
S_{2,4},V_{5}+V_{8}+V_{11}+V_{24}+V_{27}+V_{30}+V_{33}+V_{36}+V_{39}\right)
},\label{U2,2,U2,4}\\
\left(  U_{3,3},U_{3,4}\right)   &  \sim\boldsymbol{\gamma}_{\left(
S_{3,4},V_{6}+V_{10}+V_{12}+V_{25}+V_{29}+V_{31}+V_{35}+V_{37}+V_{40}\right)
},\label{U3,3,U3,4}\\
\left(  W_{2},W_{3},W_{4}\right)   &  \sim\boldsymbol{\gamma}_{(S_{2,3,4}%
,\lambda+\sum_{i=13}^{40}V_{i})}. \label{W2,W3,W4}%
\end{align}
Let $(X_{2},X_{3},X_{4})$ defined by
\begin{align}
X_{2}  &  =Y_{2}+U_{1,2}+U_{2,2}+W_{2},\label{X2}\\
X_{3}  &  =Y_{3}+U_{1,3}+U_{3,3}+W_{3},\label{X3}\\
X_{4}  &  =Y_{4}+U_{2,4}+U_{3,4}+W_{4}, \label{X4}%
\end{align}
then
\begin{equation}
\mathbf{X}_{\left[  4\right]  }=(X_{1},X_{2},X_{3},X_{4})\sim
\boldsymbol{\gamma}_{(P_{\left[  4\right]  },\lambda)}. \label{X[4]}%
\end{equation}

\end{theorem}

The following algorithm is derived from Theorem \ref{Th1_n=4_Bern2024} to simulate $\mathbf{X}_{\left[  4\right]  } \sim\mathbf{\gamma}_{\left(  P_{4},\lambda\right)  }$

\begin{algorithm}
\label{Al_n=4_Ber(2024)}Simulation of an infinitely divisible quadrivariate
gamma distribution $\mathbf{\gamma}_{\left(  P_{4},\lambda\right)  }$

\begin{enumerate}
\item Simulate $X_{1}\sim\gamma_{\left(  P_{1},\lambda\right)  };$

\item Compute $\alpha_{i},i\in\left[  40\right]  ,$ defined
in Theorem \ref{Th1_n=4_Bern2024} and simulate independently  $V_{i}\sim
\mathcal{P}\left(  \alpha_{i}X_{1}\right) ;$

\item Simulate independently
$Y_{2},Y_{3},Y_{4},\left(  U_{1,2},U_{1,3}\right),\left(  U_{2,2},U_{2,4}\right),\left(  U_{3,3},U_{3,4}\right),\left(  W_{2},W_{3},W_{4}\right)$ defined in Theorem \ref{Th1_n=4_Bern2024}.
  
\item Compute $X_{2},X_{3},$ and $X_{4},$ respectively defined by (\ref{X2}), (\ref{X3}) and  (\ref{X4}), then $\mathbf{X}_{\left[
4\right]  }$ simulates $\boldsymbol{\gamma}_{(P_{4},\lambda)}.$
\end{enumerate}
\end{algorithm}

\begin{remark}
We need to simulate $\boldsymbol{\gamma}_{\left(  S_{2,3},V_{4}+V_{7}%
+V_{9}+V_{23}+V_{26}+V_{28}+V_{32}+V_{34}+V_{38}\right)  },$ \newline%
$\boldsymbol{\gamma}_{\left(  S_{2,4},V_{5}+V_{8}+V_{11}+V_{24}+V_{27}%
+V_{30}+V_{33}+V_{36}+V_{39}\right)  },$ $\boldsymbol{\gamma}_{\left(
S_{3,4},V_{6}+V_{10}+V_{12}+V_{25}+V_{29}+V_{31}+V_{35}+V_{37}+V_{40}\right)
},$ and $\boldsymbol{\gamma}_{(S_{2,3,4},\lambda+\sum_{i=13}^{40}V_{i})}.$ To
do this, we use $40$ Pds and $4$ ugds, $3$ bgds and $1$ tgd. In each time for
bgd, we can use $1$ Pd and $2$ ugds, and for tgd we can use $6$ Pds and $5$
ugds. Finally, we can simulate $\boldsymbol{\gamma}_{(P_{\left[  4\right]
},\lambda)}$with $49$ Pds and $15$ ugds.
\end{remark}

We see that it is possible to simulate a mgd by induction. Unfortunately, the
complexity of the computations seems enormous from $n=5$ upwards. This is why
we study the Markovian case, which is simpler.

\section{Markovian multivariate gamma distributions}

In this section, we use the results given in \cite{Chatelainetal2009}. We suppose that $\mathbf{X}_{\left[  n\right]  }=\left(
X_{1},\ldots,X_{n}\right)  \sim\gamma_{\left(  P_{n},\lambda\right)  },$ where
$P_{n}$ is an affine polynomial and $\lambda>0$. We assume that $\gamma
_{\left(  P_{n},\lambda\right)  }$ is infinitely divisible and it satisfies
the following first-order Markov property%
$
\mathbf{P}\left(  X_{i+1}\in B|X_{i}=x_{i},\ldots,X_{1}=x_{1}\right)
=\mathbf{P}\left(  X_{i+1}\in B|X_{i}=x_{i}\right)  ,
$
for any $1\leqslant i\leqslant n-1$ and for any bounded set all $B\subset
\mathbb{R}.$ Such a distribution is called a Markovian mgd (Mmgd). Let
$f_{\mathbf{X}_{\left[  n\right]  }}$ be the probability density function
(pdf) of $\mathbf{X}_{\left[  n\right]  }$ on $\left(  0,\infty\right)  ^{n}$. See  \cite{Chatelainetal2009} for the expression of the probability density of a Mmgd. \cite{Chatelainetal2009} also give the following Theorem in the case
$p_{i}=1,i\in\left[  n\right]  .$

\begin{theorem}
\label{Th_Mmgd}Let $\mathbf{X}_{\left[  n\right]  }$ be a random vector
distributed according to a Mmgd with shape parameter $\lambda>0.$ The Lt of
$\mathbf{X}_{\left[  n\right]  }$ can be expressed for  $\boldsymbol{\theta}=\left(  \theta_{1},\ldots,\theta_{n}\right)\in\mathbb{R}^{n}$ as
$
L_{\mathbf{X}_{\left[  n\right]  }}\left(  \boldsymbol{\theta}\right)
=\det\left(  \boldsymbol{I}_{n}+\boldsymbol{D}_{\boldsymbol{\theta}%
}\boldsymbol{R}_{1/2}\right)  ^{-\lambda}
$
where $\boldsymbol{I}_{n}$ is the $n\times n$ identity matrix, $\boldsymbol{D}%
_{\boldsymbol{\theta}}$ is the diagonal matrix whose diagonal entries are the
components of vector $\boldsymbol{\theta,}$ and $\boldsymbol{R}_{1/2}=\left(
a_{i,j}\right)  _{1\leqslant i,j\leqslant n}$ is a correlation matrix such
that
$
a_{i,i}   =1,\forall1\leqslant i\leqslant n;
a_{i,i+1}   =\sqrt{\rho_{i,i+1}},\forall1\leqslant i\leqslant
n-1;
a_{i,j}   =\sqrt{\rho_{i,j}}=\sqrt{\rho_{i,k}}\sqrt{\rho_{k,j}}%
,\forall1\leqslant i<k<j\leqslant n.
$

\end{theorem}

We also have $\left(  X_{i},X_{i+l}\right)  \sim\gamma_{\left(  P_{\left[
i,i+l\right]  },\lambda\right)  }$ with $P_{\left[  i,i+l\right]  }\left(
\theta_{i},\theta_{i+1}\right)  =1+\theta_{i}+\theta_{i+1}+\left(
1-a_{i,i+l}^{2}\right)  \theta_{i}\theta_{i+l}$, and we have
$
\rho_{i,i+l}=\prod_{j=i}^{i+l-1}\rho_{j,j+1},\text{ for all }1\leqslant
i<i+l\leqslant n.
$ A straightforward consequence is that $\left(  X_{i},X_{i+1}\right)
\sim\gamma_{\left(  P_{\left[  i,i+1\right]  },\lambda\right)  }$ with
$P_{\left[  i,i+1\right]  }\left(  \theta_{i},\theta_{i+1}\right)
=1+\theta_{i}+\theta_{i+1}+\left(  1-\rho_{i,i+1}\right)  \theta_{i}%
\theta_{i+1}.$ We immediately derive the algorithm  for simulating 
$\mathbf{X}_{\left[  n\right]  }$ from Algorithm \ref{Al14Ber(2023)} and the computations $\frac{p_{i,i+1}}{p_{i}}\widetilde{b}_{i,i+1}=\frac{\rho_{i,i+1}%
}{1-\rho_{i,i+1}}$ and $\frac{p_{i,i+1}}{p_{i}}=1-\rho_{i,i+1}.$

\begin{algorithm}
\label{Algo_Mgd}Simulation of $\gamma_{\left(  P_{n}%
,\lambda\right)  },$ with $P_{n}=\det\left(  \boldsymbol{I}%
_{n}+\boldsymbol{D}_{\boldsymbol{\theta}}R_{1/2}\right)  $ under the
conditions of Theorem \ref{Th_Mmgd}
\end{algorithm}

\begin{enumerate}
\item Simulate $X_{1}\sim\gamma_{\left(  1,\lambda\right)  }$ ;

\item For $i=1$ to $n-1,$ do

\begin{itemize}
\item simulate $V_{i}\sim\mathcal{P}\left(  \frac{\rho_{i,i+1}}{1-\rho
_{i,i+1}}X_{i}\right)  ,$

\item simulate $X_{i+1}\sim\gamma_{(1-\rho_{i,i+1},\lambda+V_{i})}$ ;
\end{itemize}

\item Then $\mathbf{X}_{\left[  n\right]  }=\left(  X_{1},\ldots,X_{n}\right)
$ simulate $\boldsymbol{\gamma}_{\left(  P_{n},\lambda
\right)  }.$
\end{enumerate}

We note that if we consider $\mathbf{Y}_{\left[  n\right]  }=\left(  p_{1}X_{1}%
,\ldots,p_{n}X_{n}\right)  ,$ then $\mathbf{Y}_{\left[  n\right]  }=\left(
Y_{1},\ldots,Y_{n}\right)  \sim\gamma_{\left(  Q_{n}%
,\lambda\right)  },$ where $Q_{n}$ is an affine polynomial
with 
$
Q_{n}\left(  \boldsymbol{\theta}_{\left[  n\right]  }\right)
 =P_{n}\left(  \boldsymbol{p}_{\left[  n\right]
}\boldsymbol{\theta}_{\left[  n\right]  }\right) 
=\det\left(  \boldsymbol{I}_{n}+\boldsymbol{D}_{\boldsymbol{p}_{\left[
n\right]  }\boldsymbol{\theta}_{\left[  n\right]  }}R_{1/2}\right)
^{-\lambda}.
$

\section{Simulations} 
We present simulations for examples of mgds for $n\in \left\lbrace2,3,4\right\rbrace$, for examples of mfgds for $n\in \left\lbrace2,3\right\rbrace$, and for an example of Mmgd for $n=5$.

\subsection{Simulations in dimension 2}

Let $P_{2}\left(  \theta_{1},\theta_{2}\right)
=1+3\theta_{1}+3\theta_{2}+\theta_{1}\theta_{2}$ and $Q_{2}
\left(  \theta_{1},\theta_{2}\right)  =1+15/13\theta_{1}+3/13\theta
_{2}+1/13\theta_{1}\theta_{2}$, let $\mathbf{X}_{\left[2\right]}=\left(  X_{1},X_{2}\right)
\sim\boldsymbol{\gamma}_{(P_{ 2},2)}$ and $\mathbf{Y}_{\left[2\right]}=\left(
Y_{1},Y_{2}\right)  \sim\boldsymbol{\gamma}_{(Q_{2},2)}$ for
which the correlation coefficients $\rho_{X_{1},X_{2}}$ and $\rho_{Y_{1}%
,Y_{2}}$ are respectively, $\rho_{X_{1},X_{2}}=1-\frac{p_{1,2}}{p_{1}p_{2}%
}=\frac{8}{9}=0.889$ and $\rho_{Y_{1},Y_{2}}=\frac{32}{45}=0.711.$

Simulations for samples of size 1,000 of bgd $\boldsymbol{\gamma
}_{(P_{2},2)}$ and $\boldsymbol{\gamma}_{(Q_{2},2)}$ are illustrated respectively by the graphical representations
given in Figure \ref{fig1}.

\begin{center}
\begin{figure}[!h]
\centering
\includegraphics[width=8cm,height =10cm]{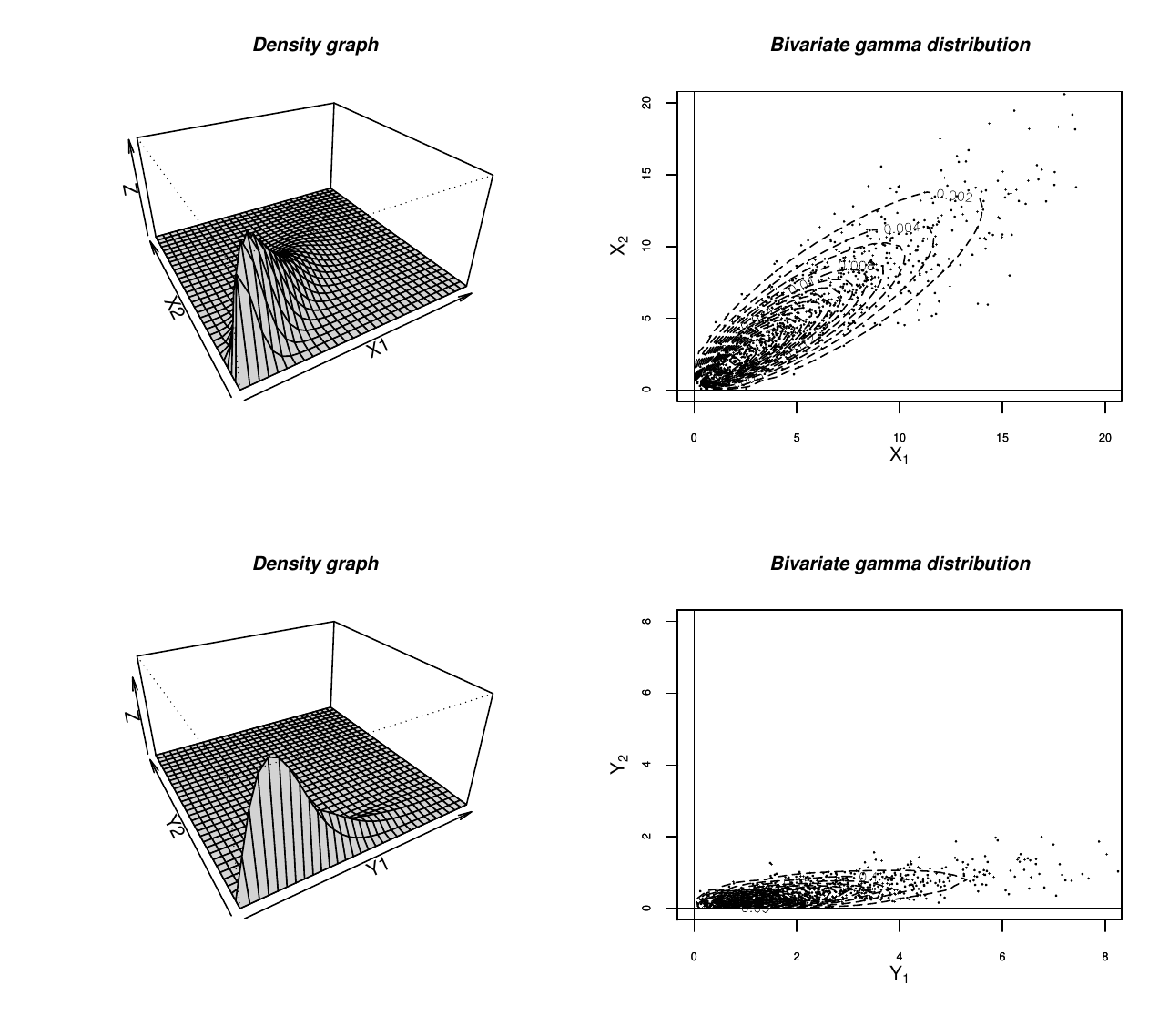}
\\
\caption{ Distributions and simulations of $\mathbf{X}_{\left[2\right]}$ and $\mathbf{Y}_{\left[2\right]}$}
\label{fig1}
\end{figure}
\end{center}

Let $\mathbf{X}^{\prime}_{\left[2\right]}=\left(  X_{1}^{\prime},X_{2}^{\prime}\right)
\sim\boldsymbol{\gamma}_{(P_{2},(2,3,4))}$ and $\mathbf{Y}%
^{\prime}_{\left[2\right]}=\left( Y_{1}^{\prime},Y_{2}^{\prime}\right)  \sim\boldsymbol{\gamma
}_{(Q_{2},(2,3,4))}$ for which the correlation coefficients
$\rho_{X_{1}^{\prime},X_{2}^{\prime}}^{\prime}$ and $\rho_{Y_{1}^{\prime
},Y_{2}^{\prime}}^{\prime}$ are respectively (by the formula $\rho
_{X_{1}^{\prime},X_{2}^{\prime}}^{\prime}=\frac{\lambda}{\sqrt{\lambda
_{1}\lambda_{2}}}\rho_{X_{1},X_{2}}$) $\rho_{X_{1}^{\prime},X_{2}^{\prime}%
}^{\prime}=\frac{2}{\sqrt{12}}\frac{8}{9}=\frac{8}{27}\sqrt{3}\simeq0.513\,$
and $\rho_{Y_{1}^{\prime},Y_{2}^{\prime}}^{\prime}=\frac{2}{\sqrt{12}}%
\frac{32}{45}=\frac{32}{135}\sqrt{3}\simeq0.411.$

Simulations for samples of size 1,000 of mfgd $\boldsymbol{\gamma
}_{(P_{2},(2,3,4))}$ and $\boldsymbol{\gamma}_{(Q_{2},(2,3,4))}$ are illustrated by the graphical representations given
in Figure \ref{fig2}.
\begin{center}
\begin{figure}[!h]
\centering
\includegraphics[width=8cm,height =10cm]{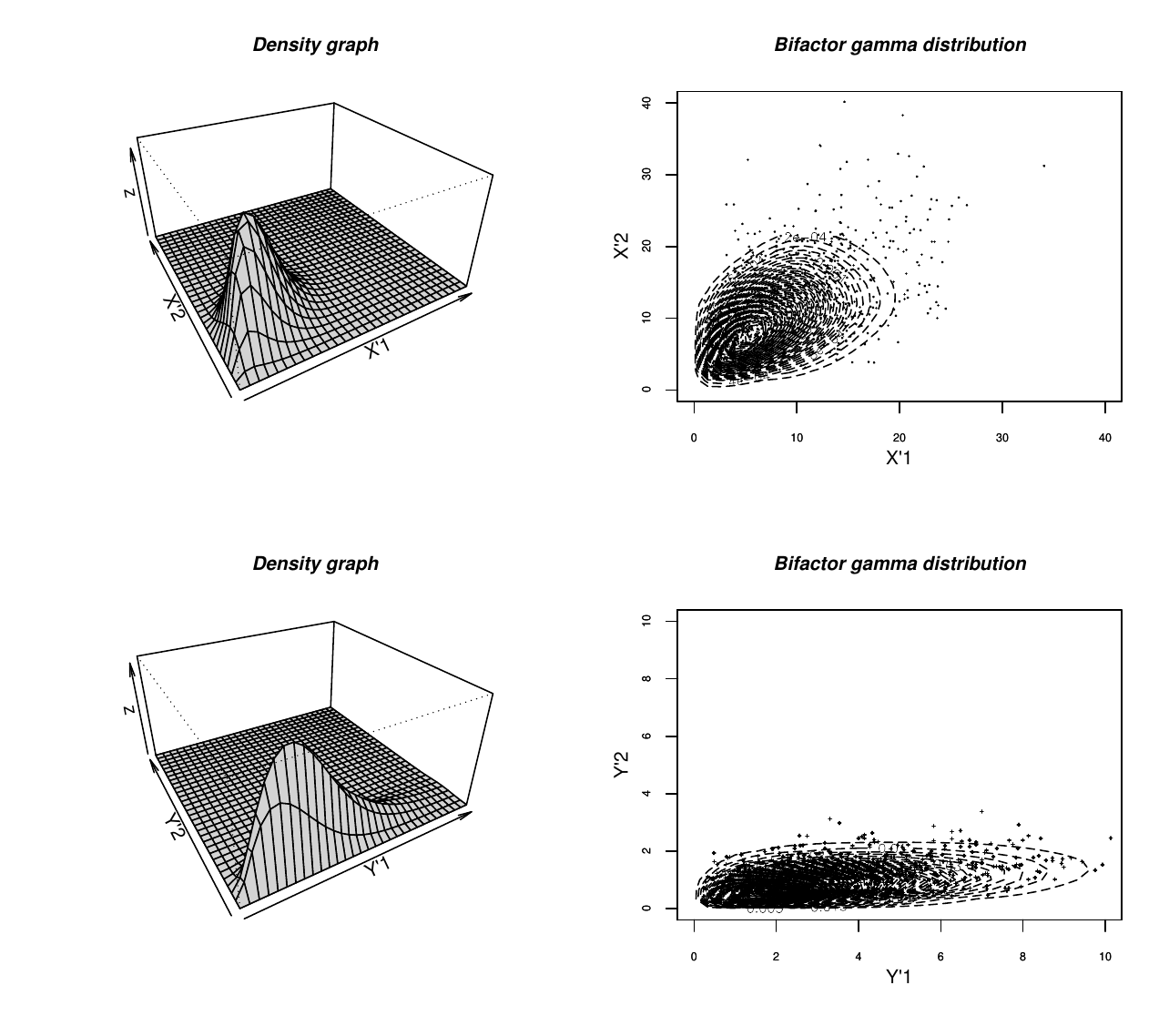}
\\
\caption{Distributions and simulations  of $\mathbf{X}^{\prime}_{\left[2\right]}$ and $\mathbf{Y}^{\prime}_{\left[2\right]}$}
\label{fig2}
\end{figure}
\end{center}

\subsection{Simulations in dimension 3}

Let $P_{3}\left(  \theta_{1},\theta_{2},\theta_{3}\right)
=1+\theta_{1}+\theta_{2}+\theta_{3}+0.55\theta_{1}\theta_{2}+0.45\theta_{1}\theta
_{3}+0.5\theta_{2}\theta_{3}+0.2\theta_{1}\theta_{2}\theta_{3}$ and
$Q_{3}\left(  \theta_{1},\theta_{2},\theta_{3}\right)
=1+\theta_{1}+4\theta_{2}+5\theta_{3}+2.2\theta_{1}\theta_{2}+2.25\theta_{1}\theta
_{3}+10\theta_{2}\theta_{3}+4\theta_{1}\theta_{2}\theta_{3}$, let
$\mathbf{X}_{\left[3\right]}=\left(  X_{1},X_{2},X_{3}\right)  \sim\boldsymbol{\gamma}_{(P_{3},2)}$ and $\mathbf{Y}_{\left[3\right]}=\left(  Y_{1},Y_{2},Y_{3}\right)\sim\boldsymbol{\gamma}_{(Q_{3},2)}$ with $\left(
Y_{1},Y_{2},Y_{3}\right)  =\left(  X_{1},4X_{2},5X_{3}\right)  $.

Simulations for samples of size 1,000 of mgd $\boldsymbol{\gamma
}_{(P_{3},2)}$ and $\boldsymbol{\gamma}_{(Q_{3},2)}$ are illustrated by the graphical representations given in
Figures \ref{fig3}, \ref{fig4} and Figures \ref{fig5}, \ref{fig6}.

\begin{center}
\begin{figure}[!h]
\centering
\includegraphics[width=15cm,height =4cm]{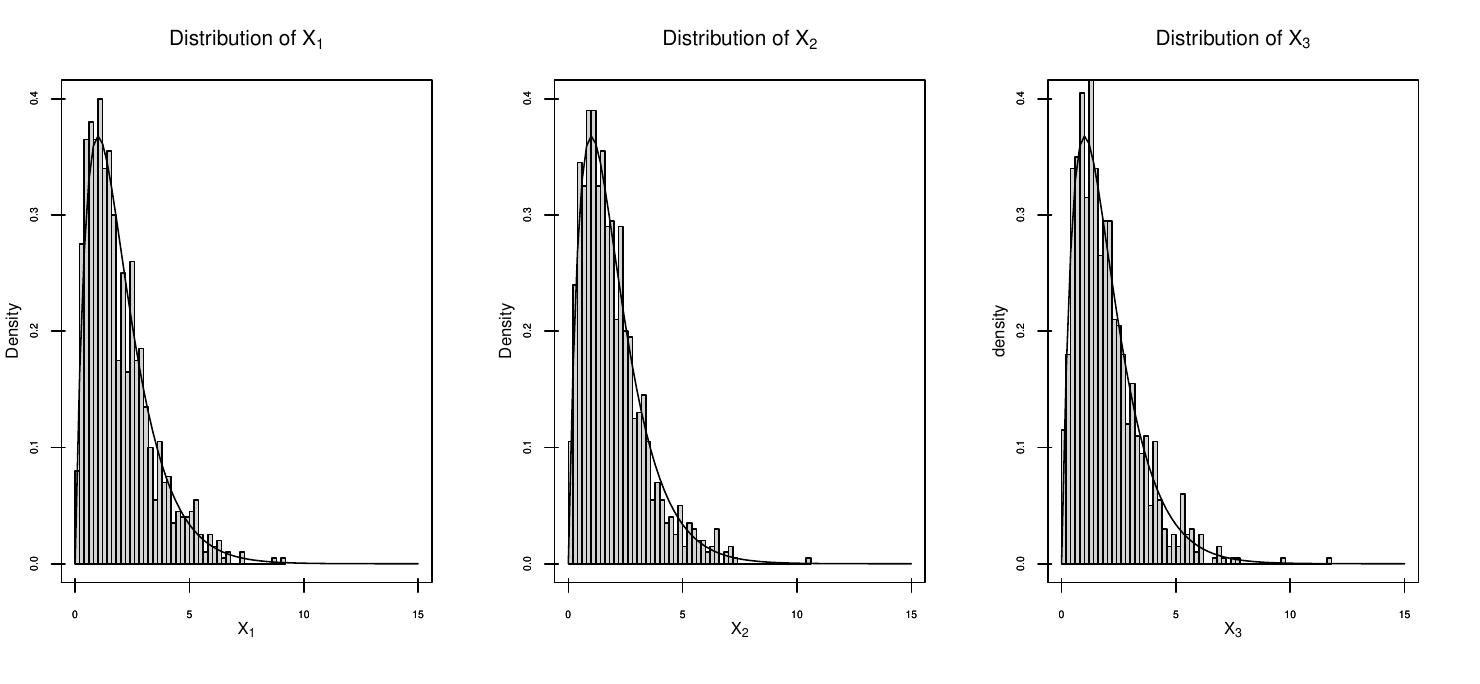}
\caption{ Distributions of $ X_{1},X_{2},X_{3}$ }
\label{fig3}
\end{figure}
\end{center}

\begin{center}
\begin{figure}[!h]
\centering
\includegraphics[width=10cm,height =12cm]{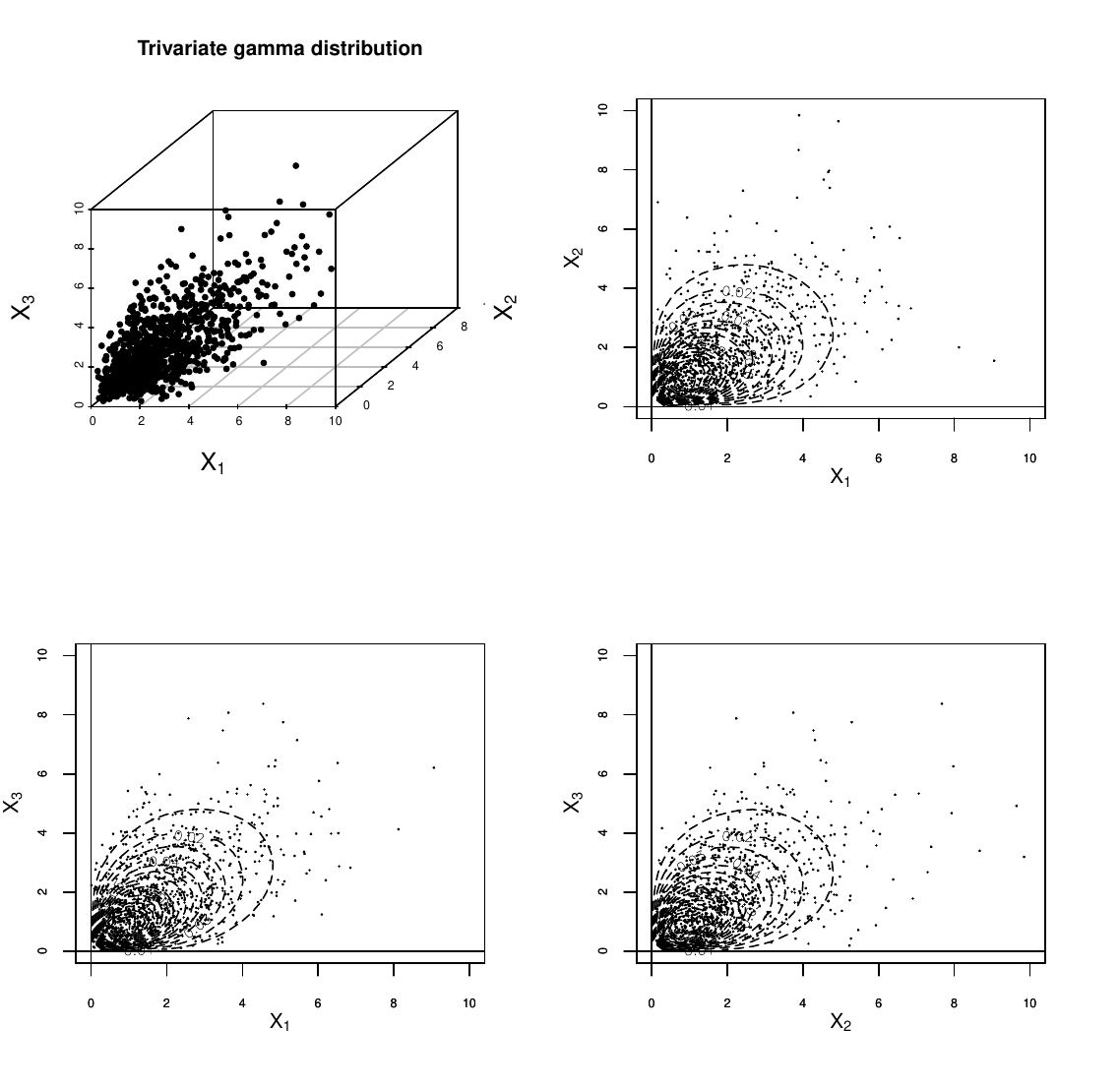}
\caption{Distribution and simulation of $\mathbf{X}_{\left[3\right]}$}
\label{fig4}
\end{figure}
\end{center}

\begin{center}
\begin{figure}[!h]
\centering
\includegraphics[width=15cm,height =4cm]{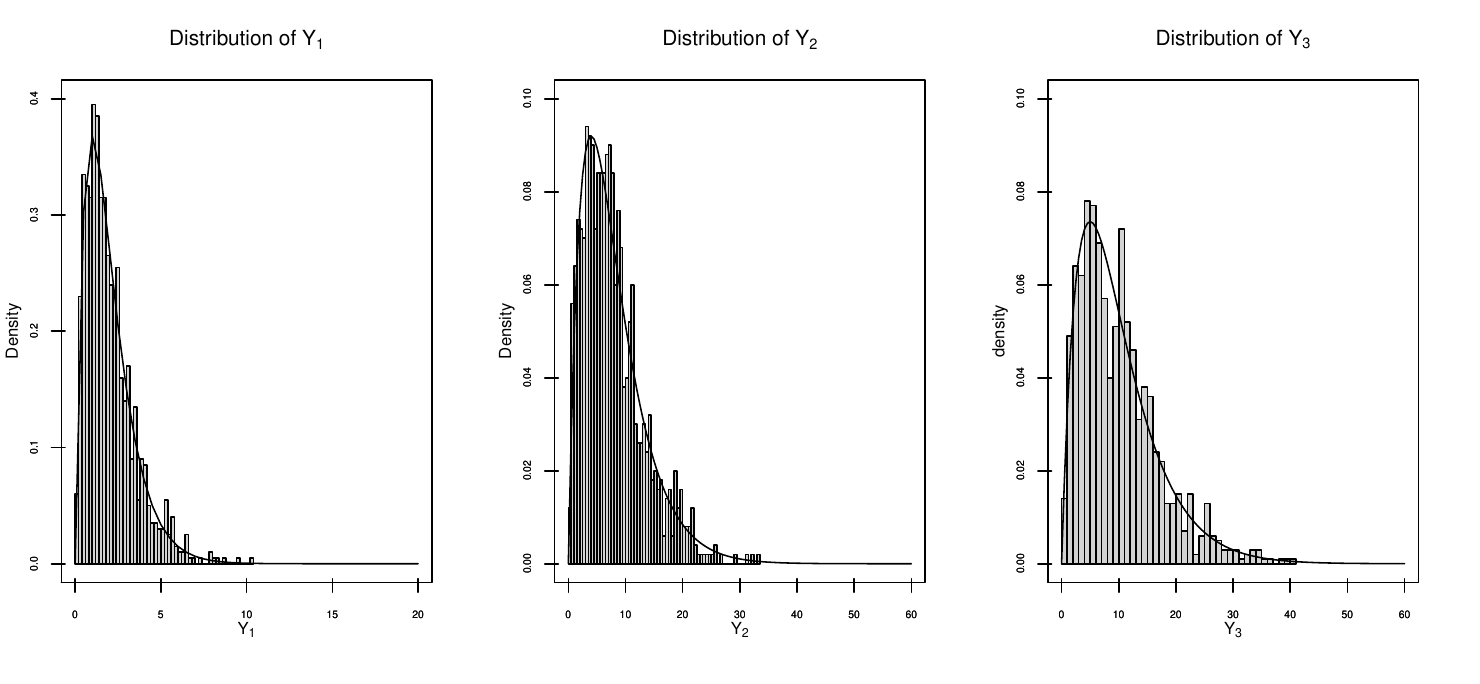}
\caption{ Distributions of $ Y_{1},Y_{2},Y_{3}$}
\label{fig5}
\end{figure}
\end{center}

\begin{center}
\begin{figure}[!h]
\centering
\includegraphics[width=10cm,height =12cm]{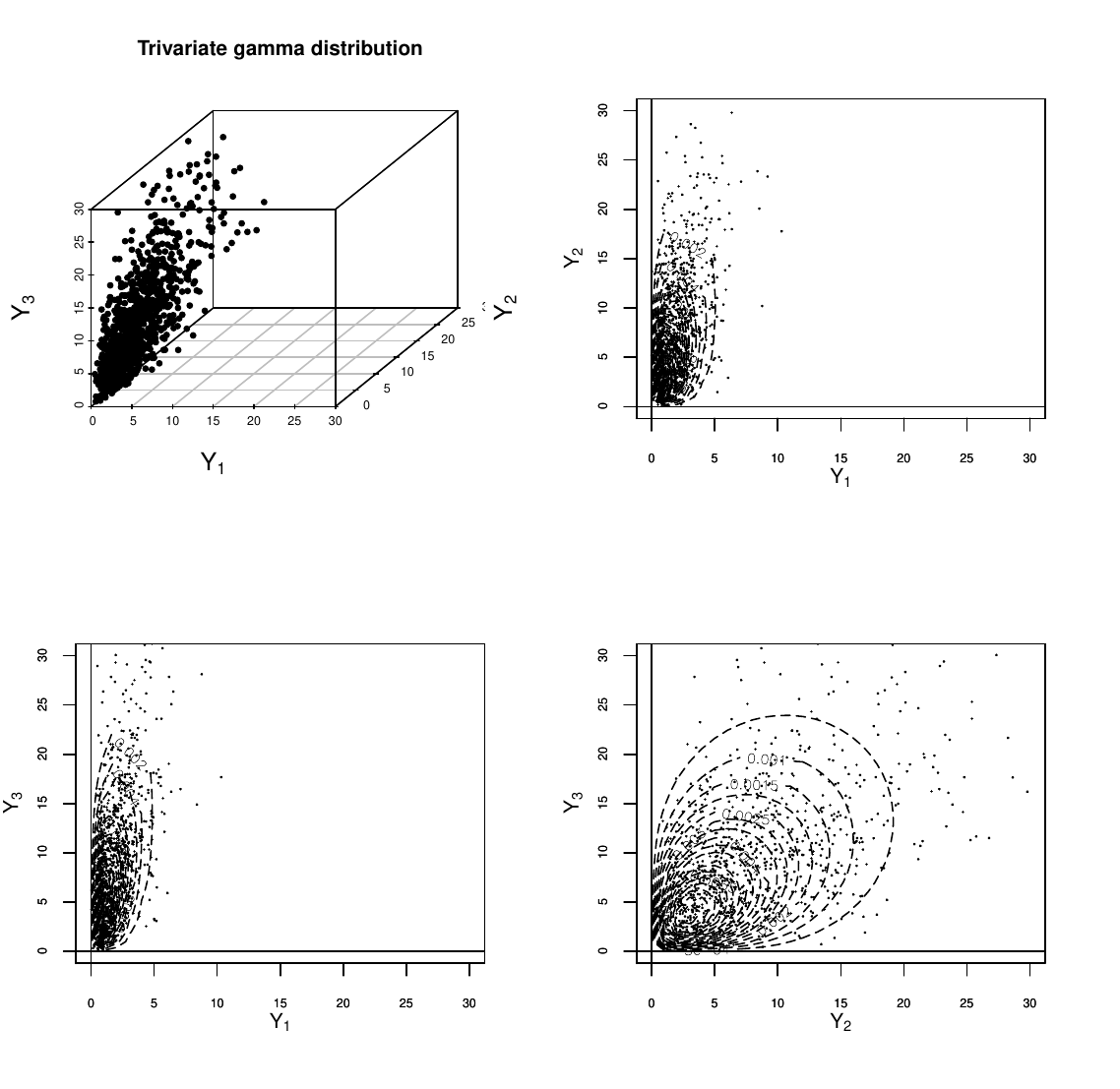}
\caption{ Distribution and simulation of $\mathbf{Y}_{\left[3\right]}$}
\label{fig6}
\end{figure}
\end{center}
 
\rule[-3cm]{0cm}{3cm}
\newline

Let $\mathbf{X}^{\prime}_{\left[3\right]}=\left(  X_{1}^{\prime},X_{2}^{\prime},X_{3}^{\prime
}\right)  \sim\boldsymbol{\gamma}_{(P_{\left[  3\right]  },(2,3,4,5))}$ and
$\mathbf{Y}^{\prime}_{\left[3\right]}=\left(  Y_{1}^{\prime},Y_{2}^{\prime},Y_{3}^{\prime
}\right)  \sim\boldsymbol{\gamma}_{(Q_{\left[  3\right]  },(2,3,4,5))}.$

Simulations for samples of size 1,000 of mfgd $\boldsymbol{\gamma
}_{(P_{3},(2,3,4,5))}$ and $\boldsymbol{\gamma}_{(Q_{3},(2,3,4,5))}$ are illustrated by the graphical representations
given in Figure \ref{fig7} and Figure \ref{fig8}.
\begin{center}
\begin{figure}[!h]
\centering
\includegraphics[width=9cm,height =9cm]{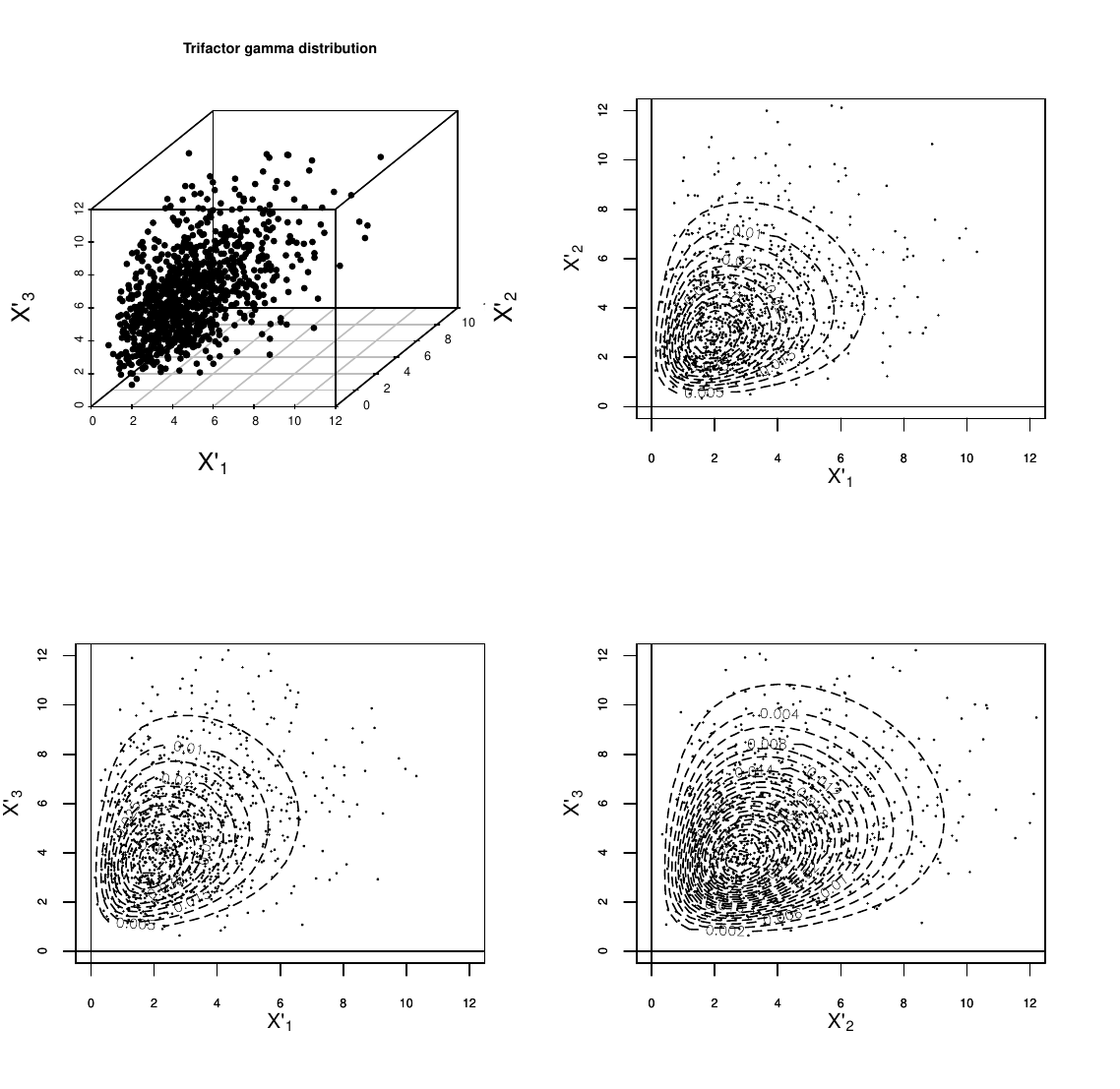}
\caption{Distribution and simulation of $\mathbf{X}^{\prime}_{\left[3\right]}$}
\label{fig7}
\end{figure}
\end{center}

\pagebreak

\begin{center}
\begin{figure}[!h]
\centering
\includegraphics[width=9cm,height =9cm]{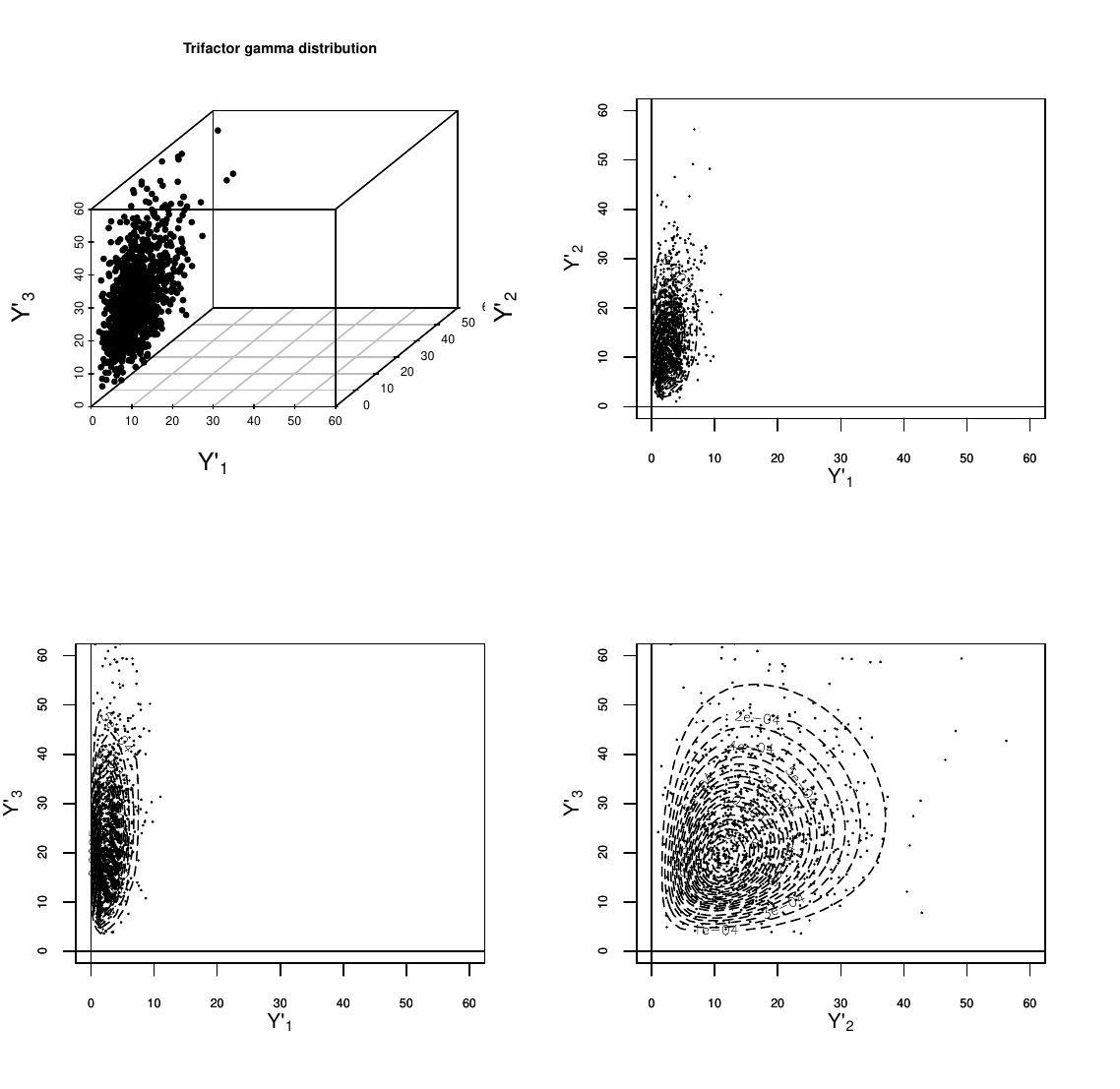}
\caption{Distribution and simulation of $\mathbf{Y}^{\prime}_{\left[3\right]}$}
\label{fig8}
\end{figure}
\end{center}

\subsection{Simulations in dimension 4}
First, we look for the symmetric case where 
$
P_{4}\left(  \theta_{1},\theta_{2},\theta_{3},\theta
_{4},\theta_{4}\right)    =1+s_{1}\left(  \theta_{1}+\theta_{2}+\theta
_{3}+\theta_{4}\right)  +\\ s_{2}\left(  \theta_{1}\theta_{2}+\theta_{1}%
\theta_{3}+\theta_{1}\theta_{4}+\theta_{2}\theta_{3}+\theta_{2}\theta
_{4}+\theta_{3}\theta_{4}\right) 
 +s_{3}\left(  \theta_{1}\theta_{2}\theta_{3}+\theta_{1}\theta_{2}\theta
_{4}+\theta_{1}\theta_{3}\theta_{4}+\theta_{2}\theta_{3}\theta_{4}\right)
+s_{4}\theta_{1}\theta_{2}\theta_{3}\theta_{4}.
$ 
According to formulas (3.13a) or (3.13b) in \cite{Bernardoff(2006)} and
\cite{Comtet(1974)} pp. 307-8, we respectively have for $n=4,$ and $\left\vert
S\right\vert =1,2,3,4,$ $\widetilde{b}_{S}=-s_{3}/s_{4},$ $(-s_{4}s_{2}%
+s_{3}^{2})/s_{4}^{2},$ $-(s_{1}s_{4}^{2}+2s_{3}^{3}-3s_{2}s_{3}s_{4}%
)/s_{4}^{3},$ $(6s_{3}^{4}-s_{4}^{3}+3s_{2}^{2}s_{4}^{2}+4s_{1}s_{3}s_{4}%
^{2}-12s_{2}s_{3}^{2}s_{4})/s_{4}^{4}.$ Without loss of generality, we can
assume that $s_{4}=1$, so we respectively have for $n=4,$ and $\left\vert
S\right\vert =1,2,3,4,$ $\widetilde{b}_{S}=-s_{3},-s_{2}+s_{3}^{2}%
,-s_{1}+3s_{3}s_{2}-2s_{3}^{3},6s_{3}^{4}-1+3s_{2}^{2}+4s_{1}s_{3}%
-12s_{2}s_{3}^{2}.$ We must simultaneously check the conditions $-s_{3}%
<0,-s_{2}+s_{3}^{2}\geqslant0,-s_{1}+3s_{3}s_{2}-2s_{3}^{3}\geqslant0$ and
$6s_{3}^{4}-1+3s_{2}^{2}+4s_{1}s_{3}-12s_{2}s_{3}^{2}\geqslant0$ for the gamma
distribution $\boldsymbol{\gamma}_{(P_{\left[  4\right]  },\lambda)}$ to be
indefinitely divisible. These conditions are equivalent to
\begin{equation}
s_{3}>0,s_{2}\leqslant s_{3}^{2},\frac{1-3s_{2}^{2}}{4s_{3}}-\frac{3}{2}%
s_{3}^{3}+3s_{2}s_{3}\leqslant s_{1}\leqslant3s_{3}s_{2}-2s_{3}^{3}.
\label{condsym4_1}%
\end{equation}
The last condition  (\ref{condsym4_1}) is equivalent to $\frac{1-3s_{2}^{2}%
}{4s_{3}}+\frac{1}{2}s_{3}^{3}+(3s_{2}s_{3}-2s_{3}^{3})\leqslant
s_{1}\leqslant3s_{3}s_{2}-2s_{3}^{3}$, which is only possible for
$\frac{1-3s_{2}^{2}}{4s_{3}}+\frac{1}{2}s_{3}^{3}\leqslant0.$ This gives us
the following condition $\sqrt{\frac{1+2s_{3}^{4}}{3}\leqslant}s_{2}$ and
(\ref{condsym4_1}) becomes
\begin{equation}
s_{3}>0,\sqrt{\frac{1+2s_{3}^{4}}{3}}\leqslant s_{2}\leqslant s_{3}^{2}%
,\frac{1-3s_{2}^{2}}{4s_{3}}-\frac{3}{2}s_{3}^{3}+3s_{2}s_{3}\leqslant
s_{1}\leqslant3s_{3}s_{2}-2s_{3}^{3}. \label{condsym4_2}%
\end{equation}
For $s_{3}=2$, we get $3.\,316\,624\,8\leqslant s_{2}\leqslant4$, and
$s_{2}=3.5$ matches. For $s_{3}=2,s_{2}=3.5$, we get $4.\,531\,25\leqslant
s_{1}\leqslant5.0$ and $s_{1}=4.75$ is a possible value. We check that for
$\left\vert S\right\vert =1,2,3,4,$ we have respectively $\widetilde{b}%
_{S}=-2,0.5,0.25,1.75$. Let  
$\mathbf{X}_{\left[4\right]}=\left(  X_{1},X_{2},X_{3},X_{4}\right)  \sim\boldsymbol{\gamma}_{(P_{4},2)}$. Simulations for samples of size 1,000 of mgd $\boldsymbol{\gamma}_{(P_{4 },2)}$ are illustrated by the graphical representations given in Figure \ref{fig9} by four one-dimensional projections and Figure \ref{fig10} by various three-dimensional projections.

\begin{center}
\begin{figure}[!h]
\centering
\includegraphics[width=10cm,height =9cm]{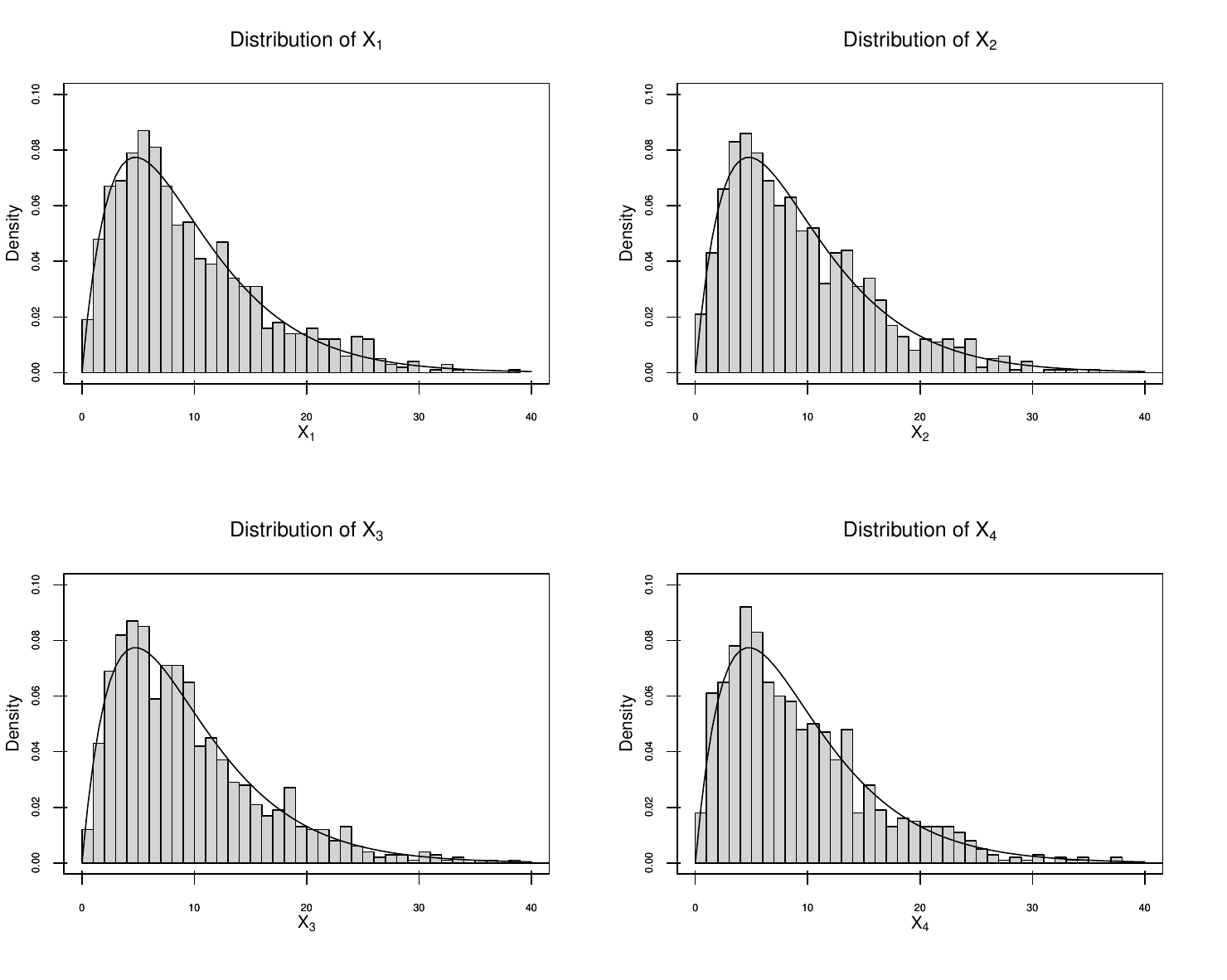}
\caption{Distribution and simulation of $X_{1}, X_{2}, X_{3}, X_{4}$}
\label{fig9}
\end{figure}
\end{center}

\pagebreak

\begin{center}
\begin{figure}[!h]
\centering
\includegraphics[width=10cm,height =18cm]{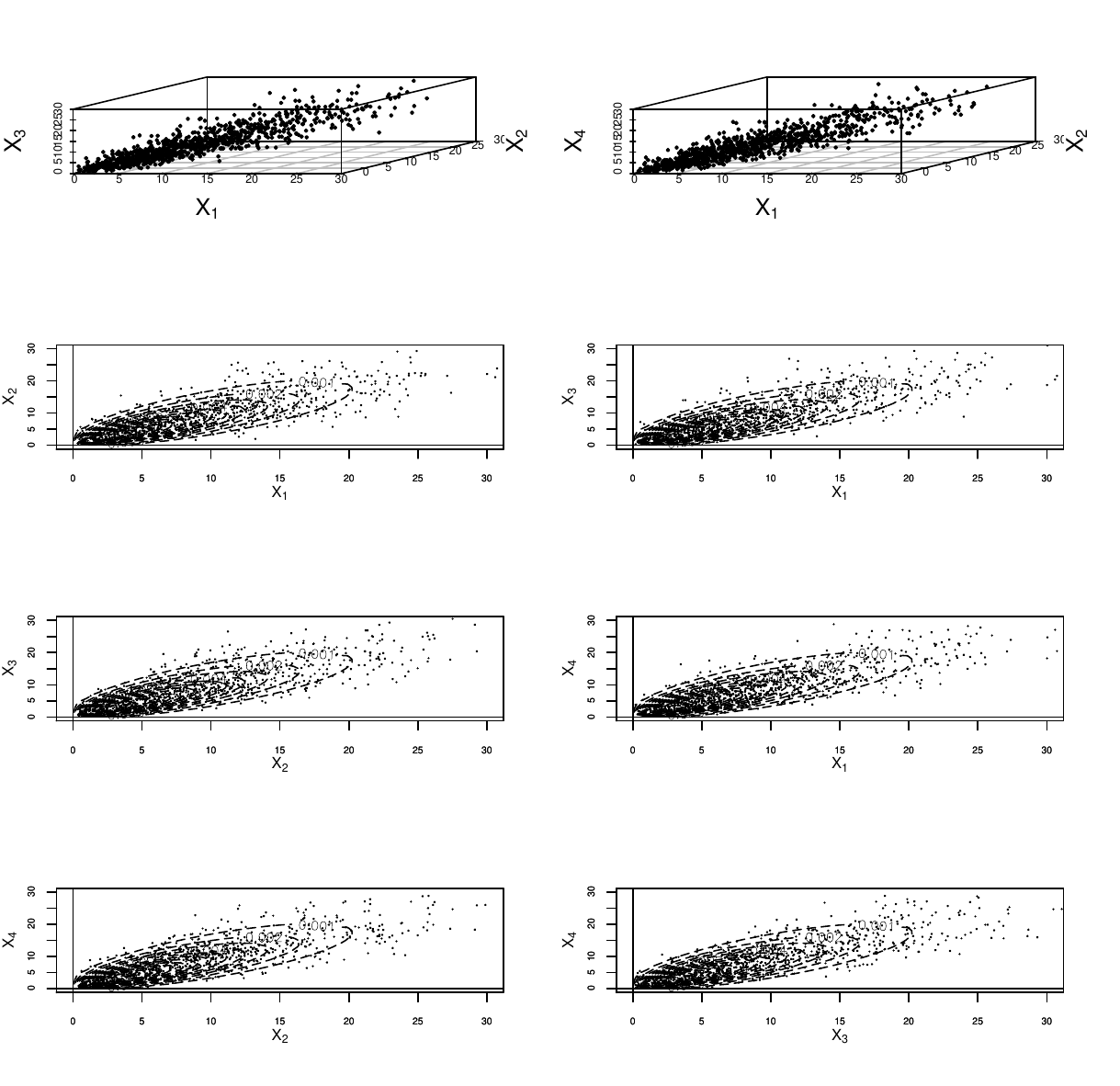}
\caption{Distribution and simulation of $\mathbf{X}_{\left[4\right]}$}
\label{fig10}
\end{figure}
\end{center}

Next, we search for the general case by slightly modifying the values of
$s_{1},s_{2},s_{3},s_{4}$ checking that the indefinite divisibility conditions of $\boldsymbol{\gamma}_{(P_{\left[  4\right]  },\lambda)}$ remain verified. For example, we obtain
\newline
$(  p_{1},p_{2},p_{3},p_{4},p_{1,2},p_{1,3},p_{1,4},p_{2,3},p_{2,4},p_{3,4},p_{1,2,3},p_{1,2,4},p_{1,3,4},p_{2,3,4},p_{1,2,3,4})
 =\\(4.75,4.8,4.85,4.7,3.5,3.55,3.6,3.65,3.45,3.4,2,1.99,2.02,2.01,1)$ with
 \newline 
$(\widetilde{b}_{1},\widetilde{b}_{2},\widetilde{b}_{3},\widetilde{b}_{4},\widetilde{b}_{1,2},\widetilde{b}_{1,3},\widetilde{b}_{1,4},\widetilde{b}_{2,3},\widetilde{b}_{2,4},\widetilde{b}_{3,4},\widetilde{b}_{1,2,3},\widetilde{b}_{1,2,4},\widetilde{b}_{1,3,4},\widetilde{b}_{2,3,4},\widetilde{b}_{1,2,3,4})=
(-2.01,-2.02,\\-1.99,-2,0.6602,0.5499,0.37,0.4198,0.49,0.48,0.111404,0.2177,0.3989,0.5053
 ,1.590846)$
 \newline
 , and $Q_{4}\left(  \theta_{1},\theta_{2},\theta_{3},\theta
_{4},\theta_{4}\right)  =1+4.75\theta_{1}+4.8\theta_{2}+4.85\theta
_{3}+4.7\theta_{4}+3.5\theta_{1}\theta_{2}+3.55\theta_{1}\theta_{3}
+3.6\theta_{1}\theta_{4}+3.65\theta_{2}\theta_{3}+3.45\theta_{2}\theta
_{4}+3.4\theta_{3}\theta_{4}+2\theta_{1}\theta_{2}\theta_{3}+1.99\theta
_{1}\theta_{2}\theta_{4}+2.02\theta_{1}\theta_{3}\theta_{4}+2.01\theta
_{2}\theta_{3}\theta_{4}+\theta_{1}\theta_{2}\theta_{3}\theta_{4}.$ Let
$\mathbf{Y}_{\left[4\right]}=\left(  Y_{1},Y_{2},Y_{3},Y_{4}\right)  \sim\boldsymbol{\gamma}_{(Q_{4},2)}$.
Simulations for samples of size 1,000 of mgd $\boldsymbol{\gamma}_{(Q_{4},2)}$ are illustrated by the graphical representations given in
Figure \ref{fig11} by four one-dimensional projections and Figure \ref{fig12} by various three-dimensional projections.

\begin{center}
\begin{figure}[!h]
\centering
\includegraphics[width=10cm,height =9cm]{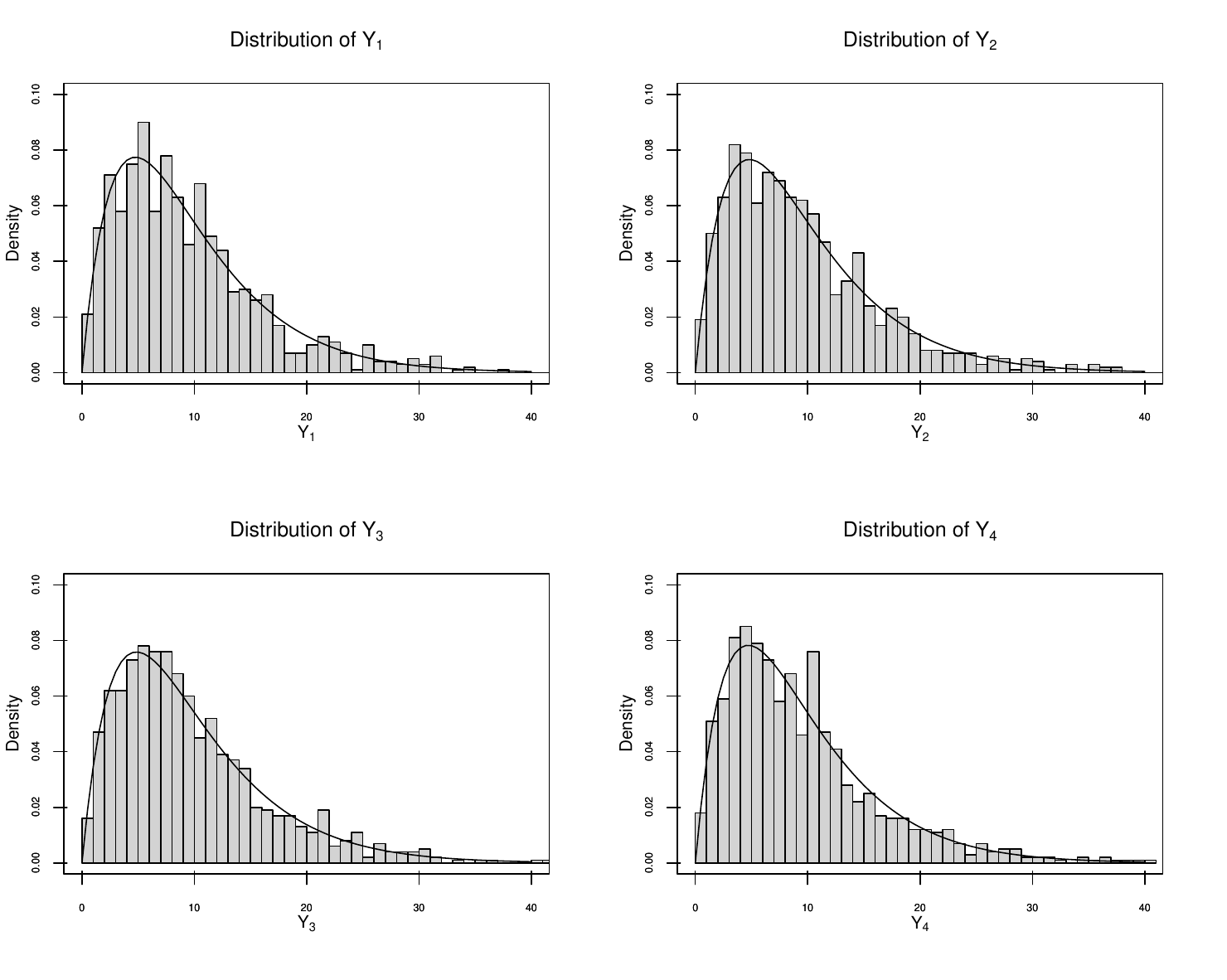}
\caption{Distribution and simulation of $Y_{1}, Y_{2}, Y_{3}, Y_{4}$}
\label{fig11}
\end{figure}
\end{center}

\pagebreak

\begin{center}
\begin{figure}[!h]
\centering
\includegraphics[width=10cm,height =18cm]{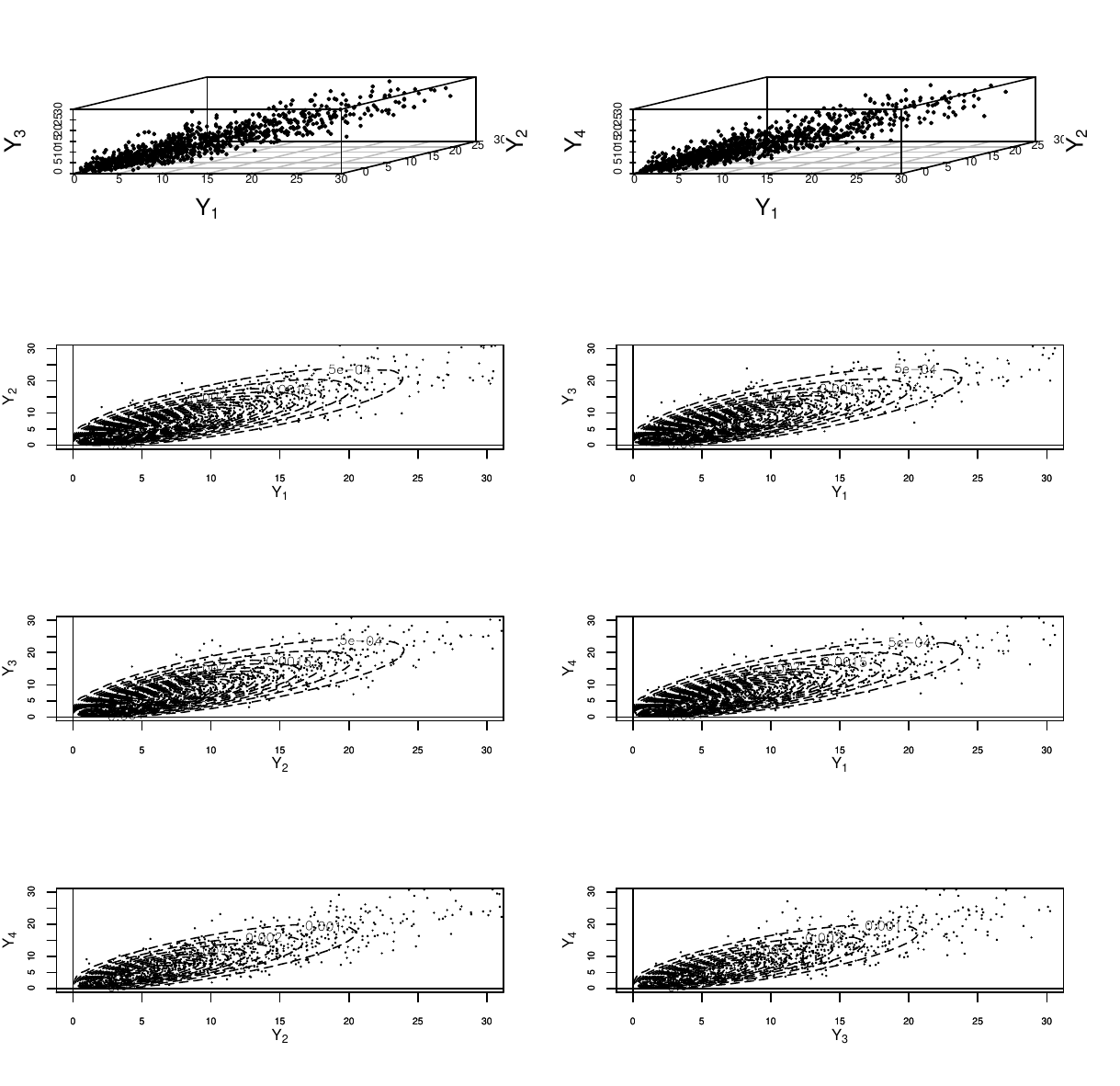}
\caption{Distribution and simulation of $\mathbf{Y}_{\left[4\right]}$}
\label{fig12}
\end{figure}
\end{center}

\subsection{Simulations of Mmgd in dimension 5}
For example, we obtain for
$\rho_{1,2}=0.9^{2},\rho_{2,3}=0.8^{2},\rho_{3,4}=0.7^{2},\rho_{4,5}=0.6^{2},$ \newline
$\boldsymbol{R}_{1/2}=\left(
\begin{smallmatrix}
1 & 0.9 & 0.72 & 0.504 & 0.3024\\
0.9 & 1 & 0.8 & 0.56 & 0.336\\
0.72 & 0.8 & 1 & 0.7 & 0.42\\
0.504 & 0.56 & 0.7 & 1 & 0.6\\
0.3024 & 0.336 & 0.42 & 0.6 & 1
\end{smallmatrix}\right)
$ and \newline
$
P_{5}\left(  \theta_{1},\theta_{2},\theta_{3},\theta
_{4},\theta_{4},\theta_{5}\right)    
  =1+\theta_{1}+\theta_{2}+\theta_{3}+\theta_{4}+\theta_{5}
  +0.19\theta_{1}\theta_{2}+0.481\,6\theta_{1}\theta_{3}+0.745\,984\,\theta
_{1}\theta_{4}+\\
0.908\,554\,24\theta_{1}\theta_{5}+0.36\theta_{2}\theta_{3}
  +0.686\,4\theta_{2}\theta_{4}+0.887\,104\,\theta_{2}\theta_{5}
+0.51\theta_{3}\theta_{4}+0.823\,6\theta_{3}\theta_{5}+0.64\theta_{4}
\theta_{5}+\\
0.068\,4\theta_{1}\theta_{2}\theta_{3}+0.130\,416\,\theta_{1}\theta
_{2}\theta_{4}+0.168\,549\,76\theta_{1}\theta_{2}\theta_{5}+0.245\,616\,\theta
_{1}\theta_{3}\theta_{4}+0.396\,645\,76\theta_{1}\theta_{3}\theta_{5}+\\
0.477\,429\,76\theta_{1}\theta_{4}\theta_{5}+0.183\,6\theta_{2}\theta_{3}\theta_{4}%
+0.296\,496\,\theta_{2}\theta_{3}\theta_{5}
  +0.439\,296\,\theta_{2}\theta_{4}\theta_{5}+0.326\,4\theta_{3}\theta
_{4}\theta_{5}+\\
0.034\,884\,\theta_{1}\theta_{2}\theta_{3}\theta_{4}%
+0.056\,334\,24\theta_{1}\theta_{2}\theta_{3}\theta_{5}+0.083\,466\,24\theta_{1}\theta_{2}\theta_{4}\theta_{5}+0.157\,194\,24\theta_{1}\theta_{3}\theta_{4}\theta_{5}+\\
0.117\,504\,\theta_{2}\theta_{3}\theta_{4}\theta_{5}
  +0.022\,325\,76\theta_{1}\theta_{2}\theta_{3}\theta_{4}\theta_{5}.
$\newline
 Let
$\mathbf{X}_{\left[5\right]}=\left(  X_{1},X_{2},X_{3},X_{4},X_{5}\right)  \sim\boldsymbol{\gamma}_{(P_{5},2)}$.
A smulation for a sample of size 1,000 of Mmgd $\boldsymbol{\gamma
}_{(P_{5},2 )}$ is illustrated by the graphical representations
given in Figure \ref{fig13} by various one-dimensional and two-dimensional projections.
\vspace{-0.5cm}
\begin{center}
\begin{figure}[!h]
\centering
\includegraphics[width=16cm,height =11cm]{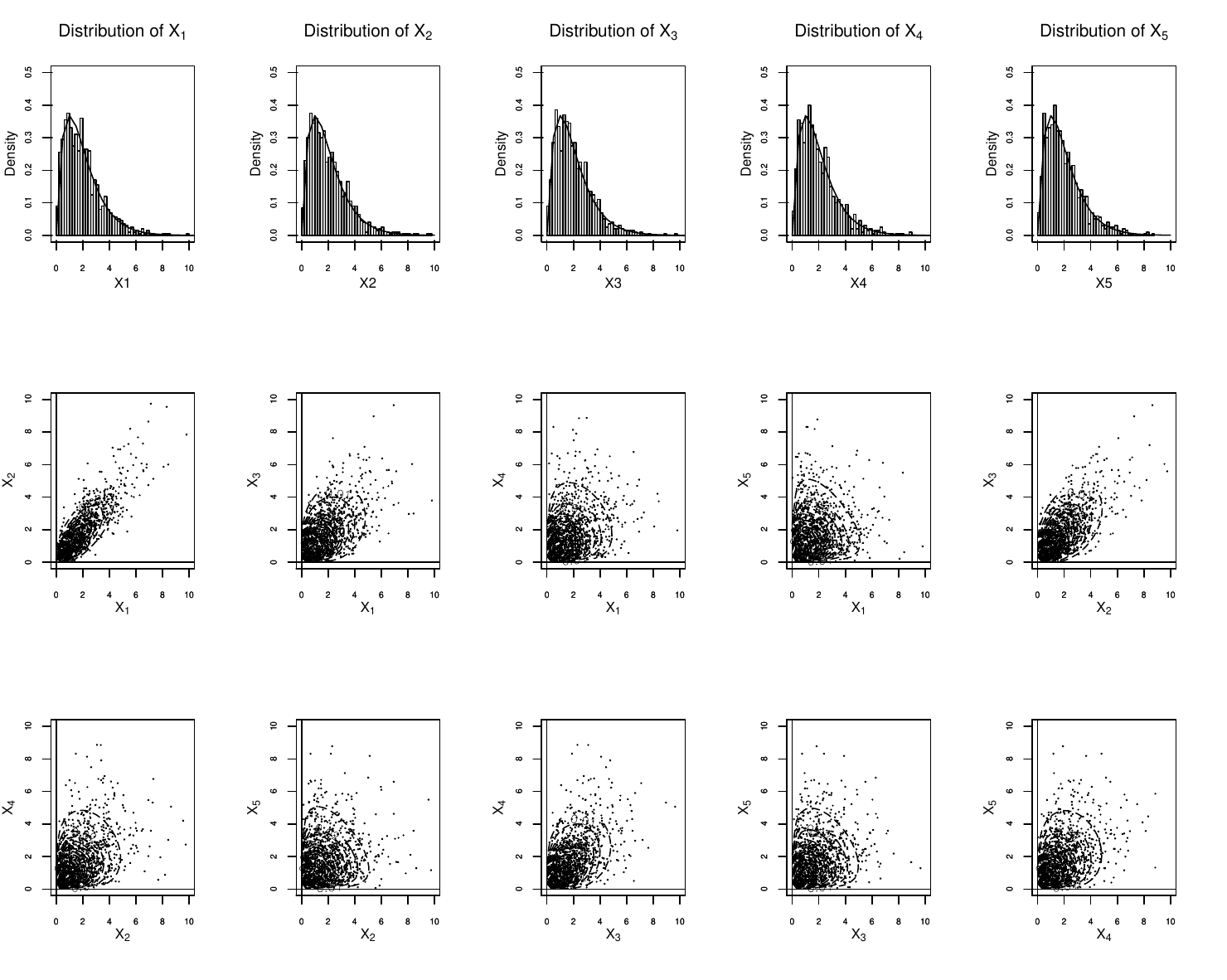}
\caption{Distribution and simulation of $\mathbf{X}_{\left[5\right]}$}
\label{fig13}
\end{figure}
\end{center}

\appendix
\setcounter{section}{0}

\section{Proofs}

\label{Annexe1}

\begin{proof}
[Proof of Proposition \ref{PropPnRn}]Using the Taylor formula in
$\boldsymbol{\theta}_{P_{n}},$we get

$P_{n}\left(  \boldsymbol{\theta}_{\left[  n\right]  }\right)=p_{\left[  n\right]  }\left(  \boldsymbol{\theta}_{\left[  n\right]
}-\boldsymbol{\theta}_{P_{n}}\right)  ^{\left[  n\right]  }\left(
1-\sum_{T\in\mathfrak{P}_{n},\left\vert T\right\vert \geqslant2}-\frac
{1}{p_{\left[  n\right]  }}\left(  \frac{\partial}{\partial\theta}\right)
^{\overline{T}}\left(  P_{n}\right)  \left(  \boldsymbol{\theta}_{P_{n}%
}\right)  \left(  \boldsymbol{\theta}_{\left[  n\right]  }-\boldsymbol{\theta
}_{P_{n}}\right)  ^{-T}\right)  ,$ then by (\ref{rT}) and (\ref{Rn}) , we obtain (\ref{P_nTaylor}). Now we compute $r_{T}$ for $T\in\mathfrak{P}_{n},\left\vert T\right\vert\geqslant2,$ from (\ref{rT}), we have\newline
$
r_{T}=\sum_{S\in\mathfrak{P}_{n}}\widetilde{p}_{\overline{S}}\boldsymbol{\theta
}_{n}^{S\smallsetminus\overline{T}}]\left(  \boldsymbol{\theta}_{P_{n}}\right)
 =\sum_{S\in\mathfrak{P}_{n},S=\overline{T}\cup T^{\prime},T^{\prime}%
\subset\left[  n\right]  \smallsetminus\overline{T}}\widetilde{p}%
_{\overline{S}}^{S\smallsetminus\overline{T}}\widetilde{\mathbf{p}}%
^{T^{\prime}}
  \sum_{T^{\prime}\in\mathfrak{P}_{T}}\widetilde{p}_{T\smallsetminus
T^{\prime}}^{T^{\prime}}\widetilde{\mathbf{p}}^{T^{\prime}}.
$

If $\left\vert T\right\vert =2,$ without loss of generality we compute  $r_{T}$ 
for $T=\left\{  1,2\right\}  $: $r_{T}  =\widetilde{p}_{1,2}+\widetilde{p}_{1}\widetilde{p}_{2}%
+\widetilde{p}_{2}\widetilde{p}_{1}-\widetilde{p}_{1}\widetilde{p}_{2}
 =\widetilde{p}_{1,2}+\widetilde{p}_{1}\widetilde{p}_{2}=\widetilde{b}%
_{\left\{  1,2\right\}  }=\widetilde{b}_{T}.$

Similarly, if $\left\vert T\right\vert =3,$  we compute  $r_{T}$ for $T=\left\{  1,2,3\right\}  $:
$
r_{T}   =\widetilde{p}_{1,2,3}+\widetilde{p}_{1,2}\widetilde{p}%
_{3}+\widetilde{p}_{1,3}\widetilde{p}_{2}+\widetilde{p}_{2,3}\widetilde{p}%
_{1}+\widetilde{p}_{1}\widetilde{p}_{2}\widetilde{p}_{3}+\widetilde{p}%
_{2}\widetilde{p}_{1}\widetilde{p}_{3}+\widetilde{p}_{3}\widetilde{p}%
_{1}\widetilde{p}_{2}+\widetilde{p}_{1}\widetilde{p}_{2}\widetilde{p}%
_{3}-\widetilde{p}_{1}\widetilde{p}_{2}\widetilde{p}_{3}
 =\widetilde{p}_{1,2,3}+\widetilde{p}_{1,2}\widetilde{p}_{3}+\widetilde{p}%
_{1,3}\widetilde{p}_{2}+\widetilde{p}_{2,3}\widetilde{p}_{1}+2\widetilde{p}%
_{1}\widetilde{p}_{2}\widetilde{p}_{3}
=\widetilde{b}_{1,2,3}=\widetilde{b}_{T}.
$

Similarly, if $\left\vert T\right\vert =4,$ we compute $r_{T}$
for $T=\left\{  1,2,3,4\right\}  .$ We prove that

$r_{\left\{  1,2,3,4\right\}  }  =\widetilde{b}_{1,2,3,4}-\left(
\widetilde{b}_{1,2}\widetilde{b}_{3,4}+\widetilde{b}_{1,3}\widetilde{b}
_{2,4}+\widetilde{b}_{1,4}\widetilde{b}_{2,3}\right) 
 =\widetilde{b}_{T}-\sum_{\left\{  U,V\right\}  \in\Pi_{T}^{2},\left\vert
U\right\vert =2,\left\vert V\right\vert =2}\widetilde{b}_{U}\widetilde{b}_{V}.$

Similarly, if $\left\vert T\right\vert =5,$ we compute $r_{T}$
for $T=\left\{  1,2,3,4,5\right\}  ,$ we prove that%

$r_{\left\{  1,2,3,4,5\right\}  }    =\widetilde{b}_{1,2,3,4,5}-\widetilde{b}%
_{1,2,3}\widetilde{b}_{4,5}-\widetilde{b}_{1,2,4}\widetilde{b}_{3,5}%
-\widetilde{b}_{1,2,5}\widetilde{b}_{3,4}-\widetilde{b}_{1,3,4}\widetilde{b}%
_{2,4}-\widetilde{b}_{1,3,5}\widetilde{b}_{2,4}
 -\widetilde{b}_{1,4,5}\widetilde{b}_{2,3}-\widetilde{b}_{2,3,4}%
\widetilde{b}_{1,5}-\widetilde{b}_{2,3,5}\widetilde{b}_{1,4}-\widetilde{b}%
_{2,4,5}\widetilde{b}_{1,3}-\widetilde{b}_{3,4,5}\widetilde{b}_{1,2}
 =\widetilde{b}_{T}-\sum_{\left\{  U,V\right\}  \in\Pi_{T}^{2},\left\vert
U\right\vert =3,\left\vert V\right\vert =2}\widetilde{b}_{U}\widetilde{b}_{V}.$

\end{proof}

\begin{proof}
[Proof of Proposition \ref{PropSU(ST)}]Indeed, we have

$S_{U}\left(  S_{T}\right)     =\frac{p_{\overline{T}}}{p_{\overline{T}%
\cup\overline{U}}}\frac{\partial}{\partial\boldsymbol{\theta}})^{\overline{U}%
}\{\frac{1}{p_{\overline{T}}}(\frac{\partial}{\partial\boldsymbol{\theta}%
})^{\overline{T}}[P_{n}(\boldsymbol{\theta})]\}
 =\frac{1}{p_{\overline{T}\cup\overline{U}}}(\frac{\partial}{\partial
\boldsymbol{\theta}})^{\overline{U}\cup\overline{T}}[P_{n}(\boldsymbol{\theta
})]
 =\frac{1}{p_{\overline{U\cap T}}}(\frac{\partial}{\partial
\boldsymbol{\theta}})^{\overline{U\cap T}}[P_{n}(\boldsymbol{\theta})]
 =S_{U}\left(  P_{n}\right)  .$

\end{proof}

\begin{proof}
[Proof of Proposition \ref{PropqtildeUbtildeUrU(ST)}]Equality
(\ref{q_tilde=p_tilde}) results from the following computation $
\widetilde{q}_{U}=\frac{p_{\overline{T}\cup T\smallsetminus U}}{p_{\overline
{T}}}/\frac{p_{\overline{T}\cup T}}{p_{\overline{T}}}=\frac{p_{\overline{U}}%
}{p_{\left[  n\right]  }}.
$ Equality (\ref{b_U(ST)=b_U}) comes from (\ref{b_tildeS}) and
(\ref{q_tilde=p_tilde}). Equality (\ref{rU(ST)=rU(Pn)}) comes from (\ref{rT1})
and (\ref{q_tilde=p_tilde}). According to (\ref{P_nTaylor}) for $S_{T}$ and
(\ref{rU(ST)=rU(Pn)}), we get
$S_{T}\left(  \boldsymbol{\theta}_{T}\right)    =\left(  -\widetilde{p}%
_{T}\right)  ^{-1}\left(  \boldsymbol{\theta}_{\left[  n\right]
}-\boldsymbol{\theta}_{P_{n}}\right)  ^{T}(1-\sum_{T^{\prime}\in
\mathfrak{P}_{T},\left\vert T^{\prime}\right\vert >1}r_{T^{\prime}}\left(
\boldsymbol{\theta}_{\left[  n\right]  }-\boldsymbol{\theta}_{P_{n}}\right)
^{-T^{\prime}})
=\left(  -\widetilde{p}_{T}\right)  ^{-1}[\left(  \boldsymbol{\theta
}_{\left[  n\right]  }-\boldsymbol{\theta}_{P_{n}}\right)  ^{T}-\sum
_{T^{\prime}\in\mathfrak{P}_{T},\left\vert T^{\prime}\right\vert
>1}r_{T^{\prime}}\left(  \boldsymbol{\theta}_{\left[  n\right]  }%
-\boldsymbol{\theta}_{P_{n}}\right)  ^{T\smallsetminus T^{\prime}}]
$
and the equalities $\theta_{i}-\widetilde{p}_{i}=\left(  -\widetilde{p}%
_{i}\right)  S_{i},i=1,\ldots n,$ gives (\ref{ST}). Equality (\ref{S^T}) is a
rewriting of (\ref{ST}).
\end{proof}

\begin{proof}
[Proof of Theorem \ref{Th_gamma_P_lambda}]Let us remember some definitions
introduced in \cite{Bernardoff(2006)}. We construct certain measures
on $\left[  0,\infty\right)  ^{n}$ indexed by $I\in\mathfrak{P}_{n}^{\ast}.$
For $i\in\left[  n\right]  ,$ define $l_{i}\left(  \mathtt{d}x_{i}\right)
=\mathbf{1}_{\left(  0,\infty\right)  }\left(  x_{i}\right)  \mathtt{d}x_{i}$
if $i\in I$ and $l_{i}\left(  \mathtt{d}x_{i}\right)  =\delta_{0}\left(
\mathtt{d}x_{i}\right)  $ if $i\notin I.$ We define the following measure on
$\left[  0,\infty\right)  ^{n}$: $
h_{I}\left(  \mathtt{d}\mathbf{x}\right)  =\bigotimes_{i=1}^{n}l_{i}\left(
\mathtt{d}x_{i}\right)  .
$
For instance, if $n=3$ and $I=\left\{  2,3\right\}  ,$ then $
h_{\left\{  2,3\right\}  }\left(  \mathtt{d}x_{1},\mathtt{d}x_{2}%
,\mathtt{d}x_{3}\right)  =\delta_{0}\left(  \mathtt{d}x_{1}\right)
1_{\left\{  0,\infty\right\}  ^{2}}\left(  x_{2},x_{3}\right)  \mathtt{d}%
x_{2}\mathtt{d}x_{3}.%
$
 We denote by $\mathbf{1}_{n}$ the vector $\left(  1,\ldots,1\right)
\in\mathbb{R}^{n},$ and by $\mathbf{1}$ if there is no ambiguity, and by
$\mathbf{0}_{n}$ the vector $\left(  0,\ldots,0\right)  \in\mathbb{R}^{n}$,
and by $\mathbf{0}$ if there is no ambiguity. For $I\in\mathfrak{P}_{n}^{\ast
},$ we write $\mathbb{N}_{i}^{I}=\mathbb{N}$ if $i\in I,$ $\mathbb{N}_{i}%
^{I}=\left\{  0\right\}  $ if $i\notin I,$ and $\mathbb{N}^{I}=\times
_{i=1}^{n}\mathbb{N}_{i}^{I}.$ For $\boldsymbol{\theta}=\left(  \theta
_{1},\ldots,\theta_{n}\right)  \in\mathbb{R}$ $^{n}$ with $\theta_{i}\neq0$
for all $i\in\left[  n\right]  ,$ recall the notations $\boldsymbol{\theta
}^{-1}=\left(  \theta_{1}^{-1},\ldots,\theta_{n}^{-1}\right)  $ and for
$\boldsymbol{\alpha}\in\mathbb{N}^{n}$, $\boldsymbol{\theta}%
^{-\boldsymbol{\alpha}}=\left(  \boldsymbol{\theta}^{-1}\right)
^{\boldsymbol{\alpha}}.$ For all $I\in\mathfrak{P}_{n}^{\ast},$ let
\begin{equation}
\mu_{\boldsymbol{\alpha},I}\left(  \mathtt{d}\mathbf{x}\right)  =\frac
{\mathbf{x}^{\boldsymbol{\alpha}-\mathbf{1}_{I}}}{\boldsymbol{\alpha
}-\mathbf{1}_{I}!}h_{I}\left(  \mathtt{d}\mathbf{x}\right)  \label{mu_alpha_I}%
\end{equation}
Thus, for $\theta_{1}>0,$\ldots,$\theta_{n}>0,$the Lt of $\mu
_{\boldsymbol{\alpha},I}$ is $L_{\mu_{\mathbf{\alpha},I}}\left(
\boldsymbol{\theta}\right)  =\boldsymbol{\theta}^{-\boldsymbol{\alpha}}.$ More
generally, for $-a_{1}+\theta_{1}>0,\ldots,-a_{n}+\theta_{n}>0,$ if
$\mathbf{a}=\left(  a_{1},\ldots,a_{n}\right)  \in\mathbb{R}^{n},$ then we
have
\begin{equation}
L_{\exp\left(  \mathbf{a},\mathbf{x}\right)  \mu_{\boldsymbol{\alpha},I}%
}\left(  \boldsymbol{\theta}\right)  =\left(  -\mathbf{a}+\boldsymbol{\theta
}\right)  ^{-\boldsymbol{\alpha}}. \label{L_mu_alpha_I}%
\end{equation}
The latter is still true if we replace $(\boldsymbol{\alpha}-\mathbf{1}_{I})!$
in (\ref{mu_alpha_I}) by $\Gamma\left(  \boldsymbol{\alpha}\right)
=\prod_{i\in I}\Gamma\left(  \alpha_{i}\right)  $ if $\alpha_{i}>0,$
$i\in\left[  n\right]  .$\newline Using (\ref{P_Taylor}) and
(\ref{c_alpha_lambda_R}), we write 
$
\left[  P\left(  \boldsymbol{\theta}\right)  \right]  ^{-\lambda}=p_{\left[
n\right]  }^{-\lambda}\sum_{\boldsymbol{\alpha}\in\mathbb{N}^{n}%
}c_{\boldsymbol{\alpha},\lambda}\left(  R\right)  \left(  -\boldsymbol{\theta
}_{P}+\boldsymbol{\theta}\right)  ^{-\left(  \boldsymbol{\alpha}%
+\lambda\mathbf{1}\right)  }.%
$
 Using (\ref{L_mu_alpha_I}) we get 
$
\left[  P\left(  \boldsymbol{\theta}\right)  \right]  ^{-\lambda
}=L_{p_{\left[  n\right]  }^{-\lambda}\exp\left(  \boldsymbol{\theta}%
_{P},\mathbf{x}\right)  \mathbf{x}^{\left(  \lambda-1\right)  \mathbf{1}}%
(\sum_{\boldsymbol{\alpha}\in\mathbb{N}^{n}}c_{\boldsymbol{\alpha},\lambda
}\left(  R\right)  \frac{\mathbf{x}^{\boldsymbol{\alpha}}}{\Gamma\left(
\boldsymbol{\alpha}+\lambda\mathbf{1}\right)  })h_{\left[  n\right]  }}\left(
\boldsymbol{\theta}\right)
$
and the equality $\left(  \lambda\right)  _{\boldsymbol{\alpha}}=\Gamma\left(
\boldsymbol{\alpha}+\lambda\mathbf{1}\right)  /\left[  \Gamma\left(
\mathbf{\lambda}\right)  \right]  ^{n}$ gives (\ref{gamma_P_lambda}).
\end{proof}

\begin{proof}
[Proof of Corollary \ref{cor_cor}]We apply formula (\ref{gamma_P_lambda2}) of
Theorem (\ref{Th_gamma_P_lambda}). We have $p_{i}=1,$ $i\in\left[  n\right]
,$ $p_{T}=p^{\left\vert T\right\vert -1},$, $\widetilde{p}_{T}=-p^{-\left\vert
T\right\vert }$, for all $T\in\mathfrak{P}_{n}^{\ast}$. Then
$\boldsymbol{\theta}_{P_{n}}=\left(  -p^{-1},\ldots,-p^{-1}\right)
=-p^{-1}\mathbf{1}_{n}.$ For $\left(  \boldsymbol{\theta}_{P_{n}%
}+\boldsymbol{\theta}_{\left[  n\right]  }\right)  ^{\left[  n\right]
}>qp^{-n}$ and $\theta_{i}>-p^{-1},$ $\forall i\in\left[  n\right]  $, we
obtain
$
P_{n}\left(  \boldsymbol{\theta}_{\left[  n\right]  }\right)  =p^{n-1}%
\prod_{i=1}^{n}(p^{-1}+\theta_{i})[1-qp^{-n}\prod_{i=1}^{n}(p^{-1}+\theta
_{i})^{-1}],
$
and if $\mathbf{z}=\left(  z_{1},\ldots,z_{n}\right)  ,$ then
$
R_{n}\left(  \mathbf{z}\right)  =qp^{-n}\mathbf{z}^{\left[  n\right]  }.
$
Since
$
\left[  1-R_{n}\left(  \mathbf{z}\right)  \right]  ^{-\lambda}=\sum
_{l=0}^{\infty}\frac{\left(  \lambda\right)  _{l}}{l!}\left(  qp^{-n}\right)
^{l}\mathbf{z}^{l1_{n}},
$
we have $c_{l\mathbf{1}_{n},\lambda}\left(  R_{n}\right)  =\frac{\left(
\lambda\right)  _{l}}{l!}\left(  qp^{-n}\right)  ^{l}$ if $l\in\mathbb{N},$
and $c_{\alpha,\lambda}\left(  R_{n}\right)  =0$ if $\alpha\neq l\mathbf{1}%
_{n},$ $l\in\mathbb{N}.$ Therefore (\ref{gamma_P_lambda2}) gives \newline
$
\boldsymbol{\gamma}_{\left(  P,\lambda\right)  }\left(  \mathtt{d}%
\mathbf{x}\right)   =\frac{\left(  p^{n-1}\right)  ^{-\lambda}}{\left[
\Gamma\left(  \lambda\right)  \right]  ^{n}}\exp\left(  \boldsymbol{-}%
\frac{x_{1}+\cdots+x_{n}}{p}\right)  \left(  \mathbf{x}^{\left[  n\right]
}\right)  ^{\left(  \lambda-1\right)  }[\sum_{l=0}^{\infty}\frac{1}{\left[
\left(  \lambda\right)  _{l}\right]  ^{n-1}l!}\left(  qp^{-n}\mathbf{x}%
^{\left[  n\right]  }\right)  ^{l}]\mathbf{1}_{\left(  0,\infty\right)  ^{n}%
}\left(  \mathbf{x}\right)  \left(  \mathtt{d}\mathbf{x}\right)  \\
 =\frac{\left(  p^{n-1}\right)  ^{-\lambda}}{\left[  \Gamma\left(
\lambda\right)  \right]  ^{n}}\exp\left(  \boldsymbol{-}\frac{x_{1}%
+\cdots+x_{n}}{p}\right)  \left(  \mathbf{x}^{\left[  n\right]  }\right)
^{\left(  \lambda-1\right)  }F_{n-1}\left(  \lambda,\ldots,\lambda
;qp^{-n}\mathbf{x}^{\left[  n\right]  }\right)  \mathbf{1}_{\left(
0,\infty\right)  ^{n}}\left(  \mathbf{x}\right)  \left(  \mathtt{d}%
\mathbf{x}\right) , 
$ 
and Definition (\ref{Hypergeomconf}) of $F_{n-1}$ gives (\ref{Phy2}).
\end{proof}

\begin{proof}
[Proof of Corollary \ref{cor_cor_n_2}]From (\ref{rT2}), we have (\ref{R2}),
and
$
\lbrack1-R_{2}(\mathbf{z)}]^{-\lambda}=\sum_{l=0}^{\infty}\frac{\left(
\lambda\right)  _{l}}{l!}\widetilde{b}_{1,2}^{l}\mathbf{z}^{l\mathbf{1}_{2}}.
$
Hence, if $\boldsymbol{\alpha}=l\mathbf{1}_{2}=\left(  l,l\right)  ,$
$l\in\mathbb{N},$
$
c_{\boldsymbol{\alpha},\lambda}\left(  R_{2}\right)  =\frac{\left(
\lambda\right)  _{l}}{l!}\widetilde{b}_{1,2}^{l},
$
and $c_{\boldsymbol{\alpha},\lambda}\left(  R_{2}\right)  =0$ otherwise.
Formula (\ref{gamma_P_lambda2}) therefore gives 
$
\boldsymbol{\gamma}_{\left(  P_{2},\lambda\right)  }\left(  \mathtt{d}%
\mathbf{x}\right)  =\frac{p_{1,2}^{-\lambda}}{\left[  \Gamma\left(
\lambda\right)  \right]  ^{2}}\exp(-\frac{p_{2}}{p_{1,2}}\boldsymbol{x}%
_{1}-\frac{p_{1}}{p_{1,2}}x_{2})\left(  x_{1}x_{2}\right)  ^{\left(
\lambda-1\right)  }(\sum_{\boldsymbol{l}\in\mathbb{N}}\frac{1}{\left(
\lambda\right)  _{\boldsymbol{l}}}\frac{\left(  \widetilde{b}_{1,2}x_{1}%
x_{2}\right)  ^{l}}{l!})\mathbf{1}_{\left(  0,\infty\right)  ^{2}}\left(
\mathbf{x}\right)  \left(  \mathtt{d}\mathbf{x}\right)  ,
$
and definition (\ref{Hypergeomconf}) of $F_{1}$ gives (\ref{gammaP2_lambda}).
\end{proof}

\begin{proof}
[Proof of Corollary \ref{cor_cor_n_3}]From (\ref{rT2}) and (\ref{rT3}), we
have, $r_{T}=\widetilde{b}_{T}$ for $T\in\mathfrak{P}_{3,}\left\vert
T\right\vert =2,3$. This proves Formula (\ref{R3}).

Now, for $\mathbf{z}_{3}=\left(  z_{1},z_{2},z_{3}\right)  \in\mathbb{R}^{3}$
, and $\lambda>0,$ we develop $[1-R_{3}\left(  \mathbf{z}\right)  ]^{-\lambda
}$ by \newline
$
[1-R_{3}\left(  \mathbf{z}_{3}\right)  ]^{-\lambda}  =\sum_{\mathbf{l}=\left(  l_{1},l_{2},l_{3},l_{4}\right)  \in\mathbb{N}%
^{4}}\frac{\left(  \lambda\right)  _{l_{1}+l_{2}+l_{3}+l_{4}}}{l_{1}%
!l_{2}!l_{3}!l_{4}!}\left(  \widetilde{b}_{1,2}z_{1}z_{2}\right)  ^{l_{1}%
}\left(  \widetilde{b}_{1,3}z_{1}z_{3}\right)  ^{l_{2}}\left(  \widetilde{b}%
_{2,3}z_{2}z_{3}\right)  ^{l_{3}}\left(  \widetilde{b}_{1,2,3}z_{1}z_{2}%
z_{3}\right)  ^{l_{4}}.
$

Now, the conditions $k\in\mathbb{N},l_{1}+l_{2}+l_{3}+l_{4}=k;\alpha_{1}%
=l_{1}+l_{2}+l_{4},\alpha_{2}=l_{1}+l_{3}+l_{4},\alpha_{3}=l_{2}+l_{3}+l_{4}$
give $l_{1}=k-\alpha_{3}\geqslant0,l_{2}=k-\alpha_{2}\geqslant0,l_{3}%
=k-\alpha_{1}\geqslant0,l_{4}=\alpha_{1}+\alpha_{2}+\alpha_{3}-2k\geqslant0$
and we get \newline
$
\left[  1-R_{3}\left(  \mathbf{z}\right)  \right]  ^{-\lambda}    
 =\sum_{\boldsymbol{\alpha}\in\mathbb{N}^{3},\left\Vert \boldsymbol{\alpha
}\right\Vert _{\infty}\leqslant\frac{\left\vert \boldsymbol{\alpha}\right\vert
}{2}}[\sum_{\left\Vert \boldsymbol{\alpha}\right\Vert _{\infty}\leqslant
k\leqslant\frac{\left\vert \boldsymbol{\alpha}\right\vert }{2},k\in\mathbb{N}%
}\tfrac{\left(  \lambda\right)  _{k}\widetilde{b}_{1,2}^{k-\alpha_{3}%
}\widetilde{b}_{1,3}^{k-\alpha_{2}}\widetilde{b}_{2,3}^{k-\alpha_{1}%
}\widetilde{b}_{1,2,3}^{\alpha_{1}+\alpha_{2}+\alpha_{3}-2k}}{\left(
k-\alpha_{3}\right)  !\left(  k-\alpha_{2}\right)  !\left(  k-\alpha
_{1}\right)  !\left(  \alpha_{1}+\alpha_{2}+\alpha_{3}-2k\right)
!}]\mathbf{z}_{3}^{\boldsymbol{\alpha}}%
$
and we get (\ref{calphaR3}) . Formula (\ref{gamma_P_lambda})
gives (\ref{gammaP3lambda}). According to (\ref{gammaP3lambda}), we have by 
$k\in\mathbb{N},l_{1}+l_{2}+l_{3}+l_{4}=k;\alpha_{1}=l_{1}+l_{2}+l_{4},\alpha_{2}=l_{1}+l_{3}+l_{4},\alpha_{3}=l_{2}+l_{3}+l_{4},$ 
$
\boldsymbol{\gamma}_{\left(  P_{3},\lambda\right)  }\left(  \mathtt{d}%
\mathbf{x}\right)    
 =\\ \frac{p_{\left[  3\right]  }^{-\lambda}}{\left[  \Gamma\left(
\lambda\right)  \right]  ^{3}}\exp\left(  \boldsymbol{\theta}_{P}%
,\mathbf{x}\right)  \mathbf{x}^{\left(  \lambda-1\right)  \mathbf{1}_{3}}\,\,_{1}\mathbf{F}_{3}(\lambda;\widetilde{b}_{1,2}x_{1}%
x_{2},\widetilde{b}_{1,3}x_{1}x_{3},\widetilde{b}_{2,3}x_{2}x_{3}%
,\widetilde{b}_{1,2,3}x_{1}x_{2}x_{3})\mathbf{1}_{\left(  0,\infty\right)
^{3}}\left(  \mathbf{x}\right)  \left(  \mathtt{d}\mathbf{x}\right).
$

\end{proof}

\begin{proof}
[Proof of Remark \ref{cor_cor_n_3_Rem1}]The only difference with
(\ref{gammaP3lambda}) is that all $c_{\boldsymbol{\alpha},\lambda}\left(
R_{3}\right)  $ defined by (\ref{calphaR3bTpositive}) for $\left\Vert
\boldsymbol{\alpha}\right\Vert _{\infty}\leqslant\frac{\left\vert
\boldsymbol{\alpha}\right\vert }{2}$ are positive and the formula
(\ref{gamma_P_lambda2}) gives (\ref{gammaP3lambda0}).
\end{proof}

\begin{proof}
[Proof of Remark \ref{cor_cor_n_3_Rem2}]As $\widetilde{b}_{1,2,3}=0,$
$\widetilde{b}_{1,2,3}^{\alpha_{1}+\alpha_{2}+\alpha_{3}-2k}=1$ only
for$\frac{\left\vert \boldsymbol{\alpha}\right\vert }{2}=k\in\mathbb{N}.$ In
this case, if $\left\Vert \boldsymbol{\alpha}\right\Vert _{\infty}%
\leqslant\frac{\left\vert \boldsymbol{\alpha}\right\vert }{2}=k\in
\mathbb{N},$ then 
$
c_{\boldsymbol{\alpha},\lambda}\left(  R_{3}\right)  =\tfrac{\left(
\lambda\right)  _{k}\widetilde{b}_{2,3}^{k-\alpha_{1}}\widetilde{b}%
_{1,3}^{k-\alpha_{2}}\widetilde{b}_{1,2}^{k-\alpha_{3}}}{\left(  k-\alpha
_{1}\right)  !\left(  k-\alpha_{2}\right)  !\left(  k-\alpha_{3}\right)
!}>0,
$
gives (\ref{calphaR3bTpb123z}) and (\ref{gamma_P_lambda2}) gives
(\ref{gammaP3lambda1}).
\end{proof}

\begin{proof}
[Proof of Remark \ref{cor_cor_n_3_Rem3}]If $\widetilde{b}_{1,2},\widetilde{b}%
_{1,3},\widetilde{b}_{2,3}=0$, $\widetilde{b}_{1,2,3}>0,$ then

$
c_{\boldsymbol{\alpha},\lambda}\left(  R_{3}\right)  =\sum_{\left\Vert
\boldsymbol{\alpha}\right\Vert _{\infty}\leqslant k\leqslant\frac{\left\vert
\boldsymbol{\alpha}\right\vert }{2},k\in\mathbb{N}}\tfrac{\left(
\lambda\right)  _{k}\widetilde{b}_{2,3}^{k-\alpha_{1}}\widetilde{b}%
_{1,3}^{k-\alpha_{2}}\widetilde{b}_{1,2}^{k-\alpha_{3}}\widetilde{b}%
_{1,2,3}^{\alpha_{1}+\alpha_{2}+\alpha_{3}-2k}}{\left(  k-\alpha_{1}\right)
!\left(  k-\alpha_{2}\right)  !\left(  k-\alpha_{3}\right)  !\left(
\alpha_{1}+\alpha_{2}+\alpha_{3}-2k\right)  !}\neq0
$
only for $\boldsymbol{\alpha=}k\mathbf{1}_{3}$ $,k\in\mathbb{N},$ and we
obtain (\ref{calphaR3bTzb123positive}), and $c_{\boldsymbol{\alpha
},\lambda}\left(  R_{3}\right)  =0$ otherwise. Thus we have\newline
$
\boldsymbol{\gamma}_{\left(  P_{3},\lambda\right)  }\left(  \mathtt{d}%
\mathbf{x}\right)    =\frac{p_{\left[  3\right]  }^{-\lambda}}{\left[
\Gamma\left(  \lambda\right)  \right]  ^{3}}\exp\left(  \boldsymbol{\theta
}_{P},\mathbf{x}\right)  \mathbf{x}^{\left(  \lambda-1\right)  \mathbf{1}_{3}%
}[\sum_{\boldsymbol{\alpha}=k\mathbf{1}_{3},k\in\mathbb{N}}\frac{\left(
\lambda\right)  _{k}}{k!}\widetilde{b}_{1,2,3}^{k}\frac{\left(  \mathbf{x}%
^{\left[  3\right]  }\right)  ^{k}}{\left[  \left(  \lambda\right)
_{\boldsymbol{k}}\right]  ^{3}}]\mathbf{1}_{\left(  0,\infty\right)  ^{3}%
}\left(  \mathbf{x}\right)  \left(  \mathtt{d}\mathbf{x}\right)  
$
and (\ref{Hypergeomconf}) gives (\ref{gammaP3lambda3}).
\end{proof}

\begin{proof}
[Proof of Theorem \ref{L_X1_n_0}]Let $\mathbf{1}_{n}=\left(  1,\ldots,1\right)
\in\mathbb{R}^{n},$ we denote by $\mathbf{1}_{k}=\left(  1,\ldots,1\right)
\in\mathbb{R}^{k}$ and $\mathbf{1}_{n-k}=\left(  1,\ldots,1\right)
\in\mathbb{R}^{n-k}$. We note that $\mathbf{X}_{\left[  k\right]  }$ is a
random real vector such that $\mathbf{X}_{\left[  k\right]  }\sim
\boldsymbol{\gamma}_{\left(  P_{k},\lambda\right)  },$ with
$
P_{k}\left(  \boldsymbol{\theta}_{\left[  k\right]  }\right)  =\sum
_{T\in\mathfrak{P}_{k}}p_{T}\boldsymbol{\theta}_{\left[  k\right]  }^{T}%
=\sum_{T\in\mathfrak{P}_{k}}p_{T}\boldsymbol{\theta}_{\left[  k\right]  }%
^{T}.
$
Using (\ref{gamma_P_lambda2}), the Lt $L_{\mathbf{X}_{\left[
n\right]  \smallsetminus\left[  k\right]  }}^{\mathbf{X}_{\left[  k\right]
}=\mathbf{x}_{\left[  k\right]  }}$ is given by 

$
 L_{\mathbf{X}_{\left[  n\right]  \smallsetminus\left[  k\right]  }%
}^{\mathbf{X}_{\left[  k\right]  }=\mathbf{x}_{\left[  k\right]  }}\left(
\boldsymbol{\theta}_{\left[  n\right]  \smallsetminus\left[  k\right]
}\right)= 
 \frac{\left[  \frac{p_{\left[  n\right]  }}{p_{\left[
k\right]  }}\right]  ^{-\lambda}\exp\left(  \left(  \boldsymbol{\theta}%
_{P_{n}}\right)  _{\left[  k\right]  }-\boldsymbol{\theta}_{P_{k}}%
,\mathbf{x}_{\left[  k\right]  }\right)  }{\sum_{\boldsymbol{\alpha}_{\left[
k\right]  }\in\mathbb{N}^{k},c_{\boldsymbol{\alpha}_{\left[  k\right]
},\lambda}\left(  R_{k}\right)  \neq0}\frac{c_{\boldsymbol{\alpha}_{\left[
k\right]  },\lambda}\left(  R_{k}\right)  }{\left(  \lambda\right)
_{\boldsymbol{\alpha}_{\left[  k\right]  }}}\mathbf{x}_{\left[  k\right]
}^{\boldsymbol{\alpha}_{\left[  k\right]  }}}\times\\
  \sum_{\boldsymbol{\alpha=}\left(  \boldsymbol{\alpha}_{\left[
k\right]  },\boldsymbol{\alpha}_{\left[  n\right]  \smallsetminus\left[
k\right]  }\right)  \in\mathbb{N}^{n},c_{\boldsymbol{\alpha},\lambda}\left(
R_{n}\right)  \neq0}\frac{\mathbf{x}_{\left[  k\right]  }^{\boldsymbol{\alpha
}_{\left[  k\right]  }}}{\left(  \lambda\right)  _{\boldsymbol{\alpha
}_{\left[  k\right]  }}}c_{\boldsymbol{\alpha},\lambda}\left(  R_{n}\right)
\left(  \boldsymbol{\theta}_{\left[  n\right]  \smallsetminus\left[  k\right]
}-\left(  \boldsymbol{\theta}_{P_{n}}\right)  _{\left[  n\right]
\smallsetminus\left[  k\right]  }\right)  ^{-\left(  \boldsymbol{\alpha
}_{\left[  n\right]  \smallsetminus\left[  k\right]  }+\lambda\mathbf{1}%
_{n-k}\right)  }.%
$\newline
We obtain
\begin{align}
L_{\mathbf{X}_{\left[  n\right]  \smallsetminus\left[  k\right]  }%
}^{\mathbf{X}_{\left[  k\right]  }=\mathbf{x}_{\left[  k\right]  }}\left(
\boldsymbol{\theta}_{\left[  n\right]  \smallsetminus\left[  k\right]
}\right)   &  =\frac{\exp\left(  \left(  \boldsymbol{\theta}_{P_{n}}\right)
_{\left[  k\right]  }-\boldsymbol{\theta}_{P_{k}},\mathbf{x}_{\left[
k\right]  }\right)  }{\sum_{\boldsymbol{\alpha}_{\left[  k\right]  }%
\in\mathbb{N}^{k},c_{\boldsymbol{\alpha}_{\left[  k\right]  },\lambda}\left(
R_{k}\right)  \neq0}\frac{c_{\boldsymbol{\alpha}_{\left[  k\right]  },\lambda
}\left(  R_{k}\right)  }{\left(  \lambda\right)  _{\boldsymbol{\alpha
}_{\left[  k\right]  }}}\mathbf{x}_{\left[  k\right]  }^{\boldsymbol{\alpha
}_{\left[  k\right]  }}}\left(  \frac{p_{\left[  n\right]  }}{p_{\left[
k\right]  }}\left(  \boldsymbol{\theta}_{\left[  n\right]  }%
-\boldsymbol{\theta}_{P_{n}}\right)  _{\left[  n\right]  \smallsetminus\left[
k\right]  }^{\mathbf{1}_{n-k}}\right)  ^{-\lambda}\times\nonumber\\
&  \hspace{-3cm}\sum_{\boldsymbol{\alpha=}\left(  \boldsymbol{\alpha
}_{\left[  k\right]  },\boldsymbol{\alpha}_{\left[  n\right]  \smallsetminus
\left[  k\right]  }\right)  \in\mathbb{N}^{n},c_{\boldsymbol{\alpha},\lambda
}\left(  R_{n}\right)  \neq0}\frac{\mathbf{x}_{\left[  k\right]
}^{\boldsymbol{\alpha}_{\left[  k\right]  }}}{\left(  \lambda\right)
_{\boldsymbol{\alpha}_{\left[  k\right]  }}}c_{\boldsymbol{\alpha},\lambda
}\left(  R_{n}\right)  \left(  \boldsymbol{\theta}_{\left[  n\right]
}-\boldsymbol{\theta}_{P_{n}}\right)  _{\left[  n\right]  \smallsetminus
\left[  k\right]  }^{-\boldsymbol{\alpha}_{\left[  n\right]  \smallsetminus
\left[  k\right]  }}. \label{LXn_k_Xk}%
\end{align}
Since $L_{\mathbf{X}_{\left[  n\right]  \smallsetminus\left[  k\right]  }%
}^{\mathbf{X}_{\left[  k\right]  }=\mathbf{x}_{\left[  k\right]  }}$ is the Lt
of a pd we have $L_{\mathbf{X}_{n-k}}^{\mathbf{X}_{k}=\mathbf{x}_{k}}\left(
\boldsymbol{0}_{n\mathbf{-}k}\right)  =1$ and we get
\begin{align}
&  \frac{\exp\left(  \left(  \boldsymbol{\theta}_{P_{n}}\right)  _{\left[
k\right]  }-\boldsymbol{\theta}_{P_{k}},\mathbf{x}_{\left[  k\right]
}\right)  }{\sum_{\boldsymbol{\alpha}_{\left[  k\right]  }\in\mathbb{N}%
^{k},c_{\boldsymbol{\alpha}_{\left[  k\right]  },\lambda}\left(  R_{k}\right)
\neq0}\frac{c_{\boldsymbol{\alpha}_{\left[  k\right]  },\lambda}\left(
R_{k}\right)  }{\left(  \lambda\right)  _{\boldsymbol{\alpha}_{\left[
k\right]  }}}\mathbf{x}_{\left[  k\right]  }^{\boldsymbol{\alpha}_{\left[
k\right]  }}}\label{LXn_k_Xk_0n_k}\\
&  =\left(  \frac{p_{\left[  n\right]  }}{p_{\left[  k\right]  }}\left(
-\boldsymbol{\theta}_{P_{n}}\right)  _{\left[  n\right]  \smallsetminus\left[
k\right]  }^{\mathbf{1}_{n-k}}\right)  ^{\lambda}[\sum_{\boldsymbol{\alpha
=}\left(  \boldsymbol{\alpha}_{\left[  k\right]  },\boldsymbol{\alpha
}_{\left[  n\right]  \smallsetminus\left[  k\right]  }\right)  \in
\mathbb{N}^{n},c_{\boldsymbol{\alpha},\lambda}\left(  R_{n}\right)  \neq
0}\frac{\mathbf{x}_{\left[  k\right]  }^{\boldsymbol{\alpha}_{\left[
k\right]  }}}{\left(  \lambda\right)  _{\boldsymbol{\alpha}_{\left[  k\right]
}}}c_{\boldsymbol{\alpha},\lambda}\left(  R_{n}\right)  \left(
-\boldsymbol{\theta}_{P_{n}}\right)  _{\left[  n\right]  \smallsetminus\left[
k\right]  }^{-\boldsymbol{\alpha}_{\left[  n\right]  \smallsetminus\left[
k\right]  }}]^{-1}.\nonumber
\end{align}
We carry (\ref{LXn_k_Xk_0n_k}) in (\ref{LXn_k_Xk}) and we obtain

\begin{align}
&  L_{\mathbf{X}_{\left[  n\right]  \smallsetminus\left[  k\right]  }%
}^{\mathbf{X}_{\left[  k\right]  }=\mathbf{x}_{\left[  k\right]  }}\left(
\boldsymbol{\theta}_{\left[  n\right]  \smallsetminus\left[  k\right]
}\right) \nonumber
 =[\frac{\left(  \boldsymbol{\theta}_{\left[  n\right]  }-\boldsymbol{\theta
}_{P_{n}}\right)  _{\left[  n\right]  \smallsetminus\left[  k\right]
}^{\mathbf{1}_{n-k}}}{\left(  -\boldsymbol{\theta}_{P_{n}}\right)  _{\left[
n\right]  \smallsetminus\left[  k\right]  }^{\mathbf{1}_{n-k}}}]^{-\lambda
}\times\nonumber\\
&  \frac{\sum_{\boldsymbol{\alpha=}\left(  \boldsymbol{\alpha}_{\left[
k\right]  },\boldsymbol{\alpha}_{\left[  n\right]  \smallsetminus\left[
k\right]  }\right)  \in\mathbb{N}^{n},c_{\boldsymbol{\alpha},\lambda}\left(
R_{n}\right)  \neq0}\frac{\mathbf{x}_{\left[  k\right]  }^{\boldsymbol{\alpha
}_{\left[  k\right]  }}}{\left(  \lambda\right)  _{\boldsymbol{\alpha
}_{\left[  k\right]  }}}c_{\boldsymbol{\alpha},\lambda}\left(  R_{n}\right)
\left(  \boldsymbol{\theta}_{\left[  n\right]  }-\boldsymbol{\theta}_{P_{n}%
}\right)  _{\left[  n\right]  \smallsetminus\left[  k\right]  }%
^{-\boldsymbol{\alpha}_{\left[  n\right]  \smallsetminus\left[  k\right]  }}%
}{\sum_{\boldsymbol{\alpha=}\left(  \boldsymbol{\alpha}_{\left[  k\right]
},\boldsymbol{\alpha}_{\left[  n\right]  \smallsetminus\left[  k\right]
}\right)  \in\mathbb{N}^{n},c_{\boldsymbol{\alpha},\lambda}\left(
R_{n}\right)  \neq0}\frac{\mathbf{x}_{\left[  k\right]  }^{\boldsymbol{\alpha
}_{\left[  k\right]  }}}{\left(  \lambda\right)  _{\boldsymbol{\alpha
}_{\left[  k\right]  }}}c_{\boldsymbol{\alpha},\lambda}\left(  R_{n}\right)
\left(  -\boldsymbol{\theta}_{P_{n}}\right)  _{\left[  n\right]
\smallsetminus\left[  k\right]  }^{-\boldsymbol{\alpha}_{\left[  n\right]
\smallsetminus\left[  k\right]  }}}. \label{LXn_k_Xk2}%
\end{align}
As we have $\left(  \boldsymbol{\theta}_{\left[  n\right]  }%
-\boldsymbol{\theta}_{P_{n}}\right)  _{\left[  n\right]  \smallsetminus\left[
k\right]  }^{\mathbf{1}_{n-k}}\left(  -\boldsymbol{\theta}_{P_{n}}\right)
_{\left[  n\right]  \smallsetminus\left[  k\right]  }^{-\mathbf{1}_{n-k}%
}=\prod_{i=k+1}^{n}[1+\theta_{i}\left(  -\boldsymbol{\theta}_{P_{n}}\right)
_{i}^{-1}],$ (\ref{LXn_k_Xk2}) gives (\ref{L_Zk_xk}).
\end{proof}

\begin{proof}
[Proof of Remark \ref{RemL_X1_n_1}]We have $\widetilde{p}_{T}(S_{\left[
n\right]  \smallsetminus\left[  1\right]  })=\frac{p_{\left\{  1\right\}  \cup
T}}{p_{1}}$ and
$
\widetilde{p}_{T}(S_{\left[  n\right]  \smallsetminus\left[  1\right]  })  
=-p_{\left(  \left[  n\right]  \smallsetminus1\right)  \smallsetminus
T}(S_{\left[  n\right]  \smallsetminus\left[  1\right]  })/\frac{p_{\left[
n\right]  }}{p_{1}}=-\frac{p_{\left\{  1\right\}  \cup\left(  \left[
n\right]  \smallsetminus1\right)  \smallsetminus T}}{p_{1}}/\frac{p_{\left[
n\right]  }}{p_{1}}
=-\frac{p_{\left[  n\right]  \smallsetminus T}}{p_{\left[  n\right]  }%
}=\widetilde{p}_{T}\left(  R_{n}\right)  .
$ 
We also have\newline
$
r_{T}\left(  S_{\left[  n\right]  \smallsetminus\left[  1\right]  }\right)
  =-\frac{1}{\frac{p_{\left[  n\right]  }}{p_{1}}}\left(  \frac{\partial
}{\partial\boldsymbol{\theta}}\right)  ^{\left(  \left[  n\right]
\smallsetminus\left\{  1\right\}  \right)  \smallsetminus T}\left[  \frac
{1}{p_{1}}\frac{\partial}{\partial\theta_{1}}P_{n}\left(  \boldsymbol{\theta
}_{P_{n}}\right)  \right] 
 =-\frac{1}{p_{\left[  n\right]  }}\left(  \frac{\partial}{\partial
\boldsymbol{\theta}}\right)  ^{\left[  n\right]  \smallsetminus T}P_{n}\left(
\boldsymbol{\theta}_{P_{n}}\right)  =r_{T}\left(  P_{n}\right)  .
$

\end{proof}

\begin{proof}
[Proof of Theorem \ref{L_X1_n_1}]We calculate, in five steps, another
expression of (\ref{Fkdef}) for $k=1$ to obtain a new formula for
$L_{\mathbf{X}_{\left[  n\right]  \smallsetminus\left[  1\right]  }}%
^{X_{1}=x_{1}}\left(  \boldsymbol{\theta}_{\left[  n\right]  \smallsetminus
\left[  1\right]  }\right)  $ from (\ref{L_Zk_xk}).\newline\textbf{First
step:} For $\alpha_{1}\in I_{1}\left(  R_{n}\right)  ,$ let $J_{\left[
n\right]  \smallsetminus\left[  1\right]  }\left(  \alpha_{1}\right)
=\left\{  \boldsymbol{\alpha}_{\left[  n\right]  \smallsetminus\left[
1\right]  }\in\mathbb{N}^{n-1},c_{\alpha_{1},\boldsymbol{\alpha}_{\left[
n\right]  \smallsetminus\left[  1\right]  },\lambda}\left(  R_{n}\right)
\neq0\right\}  ,$ then we have by (\ref{Fkdef})
\begin{equation}
\mathbf{F}_{1}\left(  \lambda,R_{n},x_{1},\boldsymbol{\theta}_{\left[
n\right]  \smallsetminus\left[  1\right]  }\right)  =\sum_{\alpha_{1}\in
I_{1}\left(  R_{n}\right)  }\frac{x_{1}^{\alpha_{1}}}{\left(  \lambda\right)
_{\alpha_{1}}}\sum_{\boldsymbol{\alpha}_{\left[  n\right]  \smallsetminus
\left[  1\right]  }\in J_{\left[  n\right]  \smallsetminus\left[  1\right]
}\left(  \alpha_{1}\right)  }c_{\alpha_{1},\boldsymbol{\alpha}_{\left[
n\right]  \smallsetminus\left[  1\right]  },\lambda}\left(  R_{n}\right)
\left(  \boldsymbol{\theta}_{\left[  n\right]  }-\boldsymbol{\theta}_{P_{n}%
}\right)  _{\left[  n\right]  \smallsetminus\left[  1\right]  }%
^{-\boldsymbol{\alpha}_{\left[  n\right]  \smallsetminus\left[  1\right]  }}.
\label{F1def_inter_1}%
\end{equation}
\textbf{Second step:} For $\alpha\in\mathbb{N}$ and $j\in\mathbb{N},$ we
denote by $\left\langle \alpha\right\rangle _{j}\ $the number $\alpha\left(
\alpha-1\right)  \ldots\left(  \alpha-j+1\right)  ,$ in particular
$\left\langle \alpha\right\rangle _{\alpha}=\alpha!$ and if $j>\alpha
,\left\langle \alpha\right\rangle _{j}=0.$ For $\boldsymbol{\alpha}_{\left[
k\right]  }\in\mathbb{N}^{\left[  k\right]  },$ and $\alpha_{i}\geqslant
\beta_{i}\in\mathbb{N},$ we denote by $\left\langle \boldsymbol{\alpha
}_{\left[  k\right]  }\right\rangle _{\boldsymbol{\beta}_{\left[  k\right]  }%
}$ the number $\left\langle \alpha_{1}\right\rangle _{\beta_{1}}%
\cdots\left\langle \alpha_{k}\right\rangle _{\beta_{k}}$, in particular
$\left\langle \boldsymbol{\alpha}_{\left[  k\right]  }\right\rangle
_{\boldsymbol{\alpha}_{\left[  k\right]  }}=\boldsymbol{\alpha}_{\left[
k\right]  }!$ and if $\exists j\in\left[  k\right]  ,\beta_{j}>\alpha_{j},$
then $\left\langle \boldsymbol{\alpha}_{\left[  k\right]  }\right\rangle
_{\boldsymbol{\beta}_{\left[  k\right]  }}=0.$ Let $\beta_{1}\in\mathbb{N}$,
we apply $\left(  \partial/\partial z_{1}\right)  ^{\beta_{1}}$ to
(\ref{c_alpha_lambda_R}) with $k=1$, because $R_{n}\mathbf{\ }$is an affine
polynomial, we obtain by \cite{Comtet(1974)} p. 42%

\begin{align}
&
\left(  \tfrac{\partial}{\partial z_{1}}\right)  ^{\beta_{1}}\left[  1-R\left(
\mathbf{z}\right)  \right]  ^{-\lambda}=\left(  \tfrac{\partial}{\partial
z_{1}}\right)  ^{\beta_{1}}\sum_{\boldsymbol{\alpha}\in\mathbb{N}^{n}%
}c_{\boldsymbol{\alpha},\lambda}\left(  R\right)  \mathbf{z}%
^{\boldsymbol{\alpha}}%
\left(  \lambda\right)  _{\beta_{1}}\left(  \tfrac{\partial}{\partial z_{1}%
}R_{n}\left(  \mathbf{z}\right)  \right)  ^{\beta_{1}}\left[  1-R_{n}\left(
\mathbf{z}\right)  \right]  ^{-\left(  \lambda+\beta_{1}\right)  }\nonumber\\
&  =\sum_{\alpha_{1}\in I_{1}\left(  R_{n}\right)  }\sum_{\boldsymbol{\alpha
}_{\left[  n\right]  \smallsetminus\left[  1\right]  }\in J_{\left[  n\right]
\smallsetminus\left[  1\right]  }\left(  \alpha_{1}\right)  ,\alpha
_{1}\geqslant\beta_{1}}c_{\alpha_{1},\boldsymbol{\alpha}_{\left[  n\right]
\smallsetminus\left[  1\right]  },\lambda}\left(  R_{n}\right)  \left\langle
\alpha_{1}\right\rangle _{\beta_{1}}z_{1}^{\alpha_{1}-\beta_{1}}%
\mathbf{z}_{\left[  n\right]  \smallsetminus\left[  1\right]  }%
^{\boldsymbol{\alpha}_{\left[  n\right]  \smallsetminus\left[  1\right]  }}.
\label{F1def_inter_2}%
\end{align}
\textbf{Third step:} Making $z_{1}=0$ in (\ref{F1def_inter_2}), we get
\newline
$
 \left(  \lambda\right)  _{\beta_{1}}\left[  1-R_{n}\left(  0,\mathbf{z}%
_{\left[  n\right]  \smallsetminus\left[  1\right]  }\right)  \right]
^{-\left(  \lambda+\beta_{1}\right)  }\left(  \frac{\partial}{\partial z_{1}%
}R_{n}(0,\mathbf{z}_{\left[  n\right]  \smallsetminus\left[  1\right]
})\right)  ^{\beta_{1}}=\\ \sum_{\boldsymbol{\alpha}_{\left[  n\right]  \smallsetminus\left[
1\right]  }\in J_{\left[  n\right]  \smallsetminus\left[  1\right]  }\left(
\beta_{1}\right)  }c_{\beta_{1},\boldsymbol{\alpha}_{\left[  n\right]
\smallsetminus\left[  k\right]  },\lambda}\left(  R_{n}\right)  (\beta
_{1}!)\mathbf{z}_{\left[  n\right]  \smallsetminus\left[  1\right]
}^{\boldsymbol{\alpha}_{\left[  n\right]  \smallsetminus\left[  1\right]  }}.
$ \newline
\textbf{Fourth step:} Making $\beta_{1}=\alpha_{1}$ and $\mathbf{z}_{\left[
n\right]  \smallsetminus\left[  1\right]  }=\left(  \boldsymbol{\theta
}_{\left[  n\right]  }-\boldsymbol{\theta}_{P_{n}}\right)  _{\left[  n\right]
\smallsetminus\left[  1\right]  }^{-1}$ in last Equality, we
get%
\begin{align}
&  \sum_{\boldsymbol{\alpha}_{\left[  n\right]  \smallsetminus\left[
1\right]  }\in J_{\left[  n\right]  \smallsetminus\left[  1\right]  }\left(
\alpha_{1}\right)  }c_{\alpha_{1},\boldsymbol{\alpha}_{\left[  n\right]
\smallsetminus\left[  k\right]  },\lambda}\left(  R_{n}\right)  \left(
\boldsymbol{\theta}_{\left[  n\right]  }-\boldsymbol{\theta}_{P_{n}}\right)
_{\left[  n\right]  \smallsetminus\left[  1\right]  }^{-\boldsymbol{\alpha
}_{\left[  n\right]  \smallsetminus\left[  1\right]  }} =\nonumber\\
& \tfrac{\left(  \lambda\right)  _{\alpha_{1}}}{\alpha_{1}!}\left[
1-R_{n}\left(  0,\mathbf{z}_{\left[  n\right]  \smallsetminus\left[  1\right]
}\right)  \right]  ^{-\left(  \lambda+\alpha_{1}\right)  }\left(
\tfrac{\partial}{\partial z_{1}}R_{n}(0,\left(  \boldsymbol{\theta}_{\left[
n\right]  }-\boldsymbol{\theta}_{P_{n}}\right)  _{\left[  n\right]
\smallsetminus\left[  1\right]  }^{-1}\right)  ^{\alpha_{1}}
\label{F1def_inter_4}%
\end{align}
\textbf{Fifth step:} Using $\mathfrak{z}$ and $S$ in 
(\ref{F1def_inter_4}), we get\newline
$
 \sum_{\boldsymbol{\alpha}_{\left[  n\right]  \smallsetminus\left[
1\right]  }\in J_{\left[  n\right]  \smallsetminus\left[  1\right]  }\left(
\alpha_{1}\right)  }c_{\alpha_{1},\boldsymbol{\alpha}_{\left[  n\right]
\smallsetminus\left[  k\right]  },\lambda}\left(  R_{n}\right)  \left(
\boldsymbol{\theta}_{\left[  n\right]  }-\boldsymbol{\theta}_{P_{n}}\right)
_{\left[  n\right]  \smallsetminus\left[  1\right]  }^{-\boldsymbol{\alpha
}_{\left[  n\right]  \smallsetminus\left[  1\right]  }}=\\
 \left[  \frac{\left(  \boldsymbol{\theta}_{\left[  n\right]
}-\boldsymbol{\theta}_{P_{n}}\right)  _{\left[  n\right]  \smallsetminus
\left[  1\right]  }^{-\mathbf{1}_{n-1}}}{\frac{p_{\left[  n\right]  }}{p_{1}}%
}S\left(  \boldsymbol{\theta}_{\left[  n\right]  \smallsetminus\left[
1\right]  }\right)  \right]  ^{-\lambda}\frac{\left(  \lambda\right)
_{\alpha_{1}}}{\alpha_{1}!}[\mathfrak{z}\left(  \boldsymbol{\theta}_{\left[
n\right]  \smallsetminus\left[  1\right]  }\right)  ]^{\alpha_{1}},
$
 hence last Equality gives
\begin{equation}
\mathbf{F}_{1}\left(  \lambda,R_{n},x_{1},\boldsymbol{\theta}_{\left[
n\right]  \smallsetminus\left[  1\right]  }\right)  =[\tfrac{\left(
\boldsymbol{\theta}_{\left[  n\right]  }-\boldsymbol{\theta}_{P_{n}}\right)
_{\left[  n\right]  \smallsetminus\left[  1\right]  }^{-\mathbf{1}_{n-1}}%
}{\frac{p_{\left[  n\right]  }}{p_{1}}}S\left(  \boldsymbol{\theta}_{\left[
n\right]  \smallsetminus\left[  1\right]  }\right)  ]^{-\lambda}%
\mathbf{G}\left(  R_{n},\mathfrak{z}\left(  \boldsymbol{\theta}_{\left[
n\right]  \smallsetminus\left[  1\right]  }\right)  x_{1}\right)
\label{F1def_inter_6}%
\end{equation}
Substituting (\ref{F1def_inter_6}) in (\ref{L_Zk_xk}) for $k=1$ and
because $S\left(  \boldsymbol{0}_{\left[  n\right]  \smallsetminus\left[
1\right]  }\right)  =1$ from formula (\ref{S2}) and (\ref{Qn_k}), we get
\begin{equation}
L_{\mathbf{X}_{\left[  n\right]  \smallsetminus\left[  1\right]  }}%
^{X_{1}=x_{1}}\left(  \boldsymbol{\theta}_{\left[  n\right]  \smallsetminus
\left[  1\right]  }\right)  =[Q_{\left[  n\right]  \smallsetminus\left[
1\right]  }\left(  \boldsymbol{\theta}_{\left[  n\right]  \smallsetminus
\left[  1\right]  }\right)  ]^{-\lambda}\frac{[\tfrac{\left(
\boldsymbol{\theta}_{\left[  n\right]  }-\boldsymbol{\theta}_{P_{n}}\right)
_{\left[  n\right]  \smallsetminus\left[  1\right]  }^{-\mathbf{1}_{n-1}}%
}{\frac{p_{\left[  n\right]  }}{p_{1}}}S\left(  \boldsymbol{\theta}_{\left[
n\right]  \smallsetminus\left[  1\right]  }\right)  ]^{-\lambda}%
\mathbf{G}\left(  R_{n},\mathfrak{z}\left(  \boldsymbol{\theta}_{\left[
n\right]  \smallsetminus\left[  1\right]  }\right)  x_{1}\right)  }%
{[\tfrac{\left(  -\boldsymbol{\theta}_{P_{n}}\right)  _{\left[  n\right]
\smallsetminus\left[  1\right]  }^{-\mathbf{1}_{n-1}}}{\frac{p_{\left[
n\right]  }}{p_{1}}}S\left(  \boldsymbol{0}_{\left[  n\right]  \smallsetminus
\left[  1\right]  }\right)  ]^{-\lambda}\mathbf{G}\left(  R_{n},\mathfrak{z}%
\left(  \boldsymbol{0}_{\left[  n\right]  \smallsetminus\left[  1\right]
}\right)  x_{1}\right)  } \label{F1def_inter_7}%
\end{equation}
Finally (\ref{F1def_inter_7}) and (\ref{Qn_k}) gives (\ref{LXn_1_X1_Yn_1}).
Formulas (\ref{S2}), (\ref{B2}) and (\ref{S3}) remain to be proven.

First, according to (\ref{P_Taylor}), we have 
$
\frac{P_{n}\left(  \boldsymbol{\theta}_{\left[  n\right]  }\right)
}{p_{\left[  n\right]  }\left(  \theta_{1}-\widetilde{p}_{1}\right)
}=\left\{  1-R_{n}\left[  \left(  \boldsymbol{\theta}-\boldsymbol{\theta}%
_{P}\right)  ^{-1}\right]  \right\}  \left(  \boldsymbol{\theta}_{\left[
n\right]  }-\boldsymbol{\theta}_{P_{n}}\right)  _{\left[  n\right]
\smallsetminus\left[  1\right]  }^{\mathbf{1}_{n-1}}.
$
We make $\theta_{1}\rightarrow\infty$ in last Equality and we get%
\begin{align}
&  \lim_{\theta_{1}\rightarrow\infty}\frac{P_{n}\left(  \boldsymbol{\theta
}_{\left[  n\right]  }\right)  }{p_{\left[  n\right]  }\left(  \theta
_{1}-\widetilde{p}_{1}\right)  }
 =\left[1-R_{n}\left(  \mathbf{0},\left(  \boldsymbol{\theta}-\boldsymbol{\theta
}_{P}\right)  _{\left[  n\right]  \smallsetminus\left[  1\right]  }%
^{-1}\right)  \right]\left(  \boldsymbol{\theta}_{\left[  n\right]  }%
-\boldsymbol{\theta}_{P_{n}}\right)  _{\left[  n\right]  \smallsetminus\left[
1\right]  }^{\mathbf{1}_{n-1}}\label{BQ1}
\end{align}
and, because $P$ is an affine polynomial, we have\newline
$
 \lim_{\theta_{1}\rightarrow\infty}\frac{P_{n}\left(  \boldsymbol{\theta
}_{\left[  n\right]  }\right)  }{p_{\left[  n\right]  }\left(
\boldsymbol{\theta}_{\left[  n\right]  }-\boldsymbol{\theta}_{P_{n}}\right)
_{1}}
=   \lim_{\theta_{1}\rightarrow\infty}\frac{P\left(  0,\boldsymbol{\theta
}_{\left[  n\right]  \smallsetminus\left[  1\right]  }\right)  +\theta
_{1}\frac{\partial}{\partial\theta_{1}}\left(  0,\boldsymbol{\theta}_{\left[
n\right]  \smallsetminus\left[  1\right]  }\right)  }{p_{\left[  n\right]
}\left(  \theta_{1}-\widetilde{p}_{1}\right)  }
 =\frac{\frac{\partial}{\partial\theta_{1}}P_{n}\left(  \boldsymbol{\theta
}_{\left[  n\right]  }\right)  }{p_{\left[  n\right]  }}
.$ 
Finally last Equality and (\ref{BQ1}) prove (\ref{S2}).

Secondly, according to (\ref{P_Taylor}), with $\mathbf{z=}\left(
\boldsymbol{\theta}_{\left[  n\right]  }-\boldsymbol{\theta}_{P_{n}}\right)
^{-1},$ that is $z_{i}=\left(  \theta_{i}-\widetilde{p}_{i}\right)  ^{-1}$ or
$\theta_{i}=\widetilde{p}_{i}+z_{i}^{-1},i\in\left[  n\right]  ,$ we have 
$  
R_{n}\left(  \mathbf{z}_{\left[  n\right]  }\right)  =1-\frac{1}{p_{\left[
n\right]  }}\left(  \prod_{j=2}^{n}z_{j}\right)  z_{1}P_{n}\left(
\widetilde{p}_{1}+z_{1}^{-1},\ldots,\widetilde{p}_{n}+z_{n}^{-1}\right)  .
$
Deriving with respect to the variable $z_{1},$ we obtain $
 \frac{\partial}{\partial z_{1}}R_{n}\left(  z_{1},z_{2},\ldots,z_{n}\right)
=\\-\frac{1}{p_{\left[  n\right]  }}\left(  \prod_{j=2}^{n}z_{j}\right)
\left\lbrace  P_{n}\left(  \widetilde{p}_{1}+z_{1}^{-1},\ldots,\widetilde{p}%
_{n}+z_{n}^{-1}\right)  -z_{1}^{-1}\frac{\partial}{\partial\theta_{1}}\left[
P_{n}\left(  \widetilde{p}_{1}+z_{1}^{-1},\ldots,\widetilde{p}_{n}+z_{n}%
^{-1}\right)  \right]  \right\rbrace  .
$

Because $R_{n}$ is an affine polynomial with respect to the $n$ variables
$z_{1},z_{2},\ldots,z_{n},$ we know that $\frac{\partial}{\partial z_{1}}%
R_{n}\left(  z_{1},z_{2},\ldots,z_{n}\right)  $ is an affine polynomial with
respect to the $n-1$ variables $z_{2},\ldots,z_{n}.$ Putting $z_{1}=0$ in the
left-hand side of last Equality and making $z_{1}\rightarrow\infty$ in the
right-hand side of last Equality, we get\newline 
$
 \frac{\partial}{\partial z_{1}}R_{n}\left(  0,z_{2}\ldots,z_{n}\right)
 =-\frac{1}{p_{\left[  n\right]  }}\left(  \prod_{j=2}^{n}z_{j}\right)
\left[  P_{n}\left(  \widetilde{p}_{1},\widetilde{p}_{2}+z_{2}^{-1}%
,\ldots,\widetilde{p}_{n}+z_{n}^{-1}\right)  \right]  
$
and \newline
$
\left(  \prod_{j=2}^{n}z_{j}^{-1}\right)  \frac{\partial}{\partial z_{1}%
}R_{n}\left(  0,z_{2}\ldots,z_{n}\right) 
=-\frac{1}{p_{\left[  n\right]  }}\left[  P_{n}\left(  \widetilde{p}%
_{1},\widetilde{p}_{2}+z_{2}^{-1},\ldots,\widetilde{p}_{n}+z_{n}^{-1}\right)
\right]  .
$ 
Substituting $z_{j}=\left(  \theta_{j}-\widetilde{p}_{j}\right)  ^{-1},$
$j=2,\ldots,n,$ we obtain \newline
$
  \prod_{j=2}^{n}\left(  \theta_{j}-\widetilde{p}_{j}\right)  \frac{\partial
}{\partial z_{i}}R_{n}\left(  0,\left(  \theta_{2}-\widetilde{p}_{2}\right)
^{-1},\ldots,\left(  \theta_{n}-\widetilde{p}_{n}\right)  ^{-1}\right) 
  =-\frac{1}{p_{\left[  n\right]  }}\left[  P_{n}\left(  \widetilde{p}%
_{1},\theta_{2},\ldots,\theta_{n}\right)  \right]  .
$
Multiplying by $p_{\left[  n\right]  }/p_{\left[  1\right]  }$ we
obtain (\ref{B2}).

Thirdly, applying $\frac{1}{p_{\left[  1\right]  }}%
\frac{\partial}{\partial\theta_{1}}$ to the equality 
$
P_{n}\left(  \boldsymbol{\theta}\right)  =\sum_{T\subset\left(  2,\ldots
,n\right)  }p_{\left[  1\right]  \cup T}\boldsymbol{\theta}^{\left[  1\right]
}\boldsymbol{\theta}_{\left[  n\right]  \smallsetminus\left[  1\right]  }%
^{T}+\sum_{T\in\mathfrak{P}_{n},T\cap\left[  1\right]  \neq\left[  1\right]
}p_{T}\boldsymbol{\theta}^{T}%
$ we obtain (\ref{S3}).
\end{proof}

\begin{proof}
[Proof of Lemma \ref{I1(Rn)}]\textbf{Firstly.} If $\exists T\in\mathfrak{P}%
_{n},\left\vert T\right\vert =2,$ $r_{T}\neq0,$ then without loss of
generality, we can studies the case $T=\left[  2\right]  ,$ that is
$r_{1,2}\neq0.$ For $\mathbf{z=}\left(  z_{1},z_{2},\mathbf{0}_{n-2}\right)
,$ we get  
$
\left[  1-R_{n}\left(  z_{1},z_{2},\mathbf{0}_{n-2}\right)  \right]
^{-\lambda}   =\sum_{\boldsymbol{\alpha}\in\mathbb{N}^{2}}c_{\left(
\alpha_{1}\boldsymbol{,}\alpha_{2},\mathbf{0}_{n-2}\right)  ,\lambda}\left(
R_{n}\right)  z_{1}^{\alpha_{1}}z_{2}^{\alpha_{2}}
 =\left(  1-r_{1,2}z_{1}z_{2}\right)  ^{-\lambda}
 =\sum_{l=0}^{\infty}\frac{\left(  \lambda\right)  _{l}r_{1,2}^{l}}{l!}%
z_{1}^{l}z_{2}^{l}.
$
Then we have $c_{\left(  l,l,\mathbf{0}_{n-2}\right)  ,\lambda}=\frac{\left(
\lambda\right)  _{l}r_{1,2}^{l}}{l!}\neq0,l\in\mathbb{N}.$ Finally,
$\forall\alpha_{1}\in\mathbb{N},$ $\exists\mathbb{\alpha}_{\left[  n\right]
\smallsetminus\left\{  1\right\}  }=\left(  \alpha_{1},\mathbf{0}%
_{n-2}\right)  $ such that $c_{(\alpha_{1},\mathbb{\alpha}_{\left[  n\right]
\smallsetminus\left\{  1\right\}  }),\lambda}=\frac{\left(  \lambda\right)
_{\alpha_{1}}r_{1,2}^{\alpha_{1}}}{\alpha_{1}!}\neq0,$ and $I_{1}\left(
R_{n}\right)  =\mathbb{N}.$\newline\textbf{Secondly.} Let $k\in\mathbb{N},$ such
that $\forall T\in\mathfrak{P}_{n},\left\vert T\right\vert \leqslant k<n;$
$r_{T}=0,$ if $\exists T\in\mathfrak{P}_{n},\left\vert T\right\vert =k+1;$
$r_{T}\neq0,$ then without loss of generality, we can studies the case
$T=\left[  k+1\right]  ,$ that is $r_{1,\ldots,k+1}\neq0.$ For $\mathbf{z=}%
\left(  z_{1},\ldots,z_{k+1},\mathbf{0}_{n-k-1}\right)  ,$ we get
$
\left[  1-R_{n}\left(  z_{1},\ldots,z_{k+1,}\mathbf{0}_{n-k-1}\right)
\right]  ^{-\lambda}   =\\\sum_{\boldsymbol{\alpha}\in\mathbb{N}^{2}%
}c_{\left(  \alpha_{1}\boldsymbol{,}\ldots,\alpha_{k+1},\mathbf{0}%
_{n--k-1}\right)  ,\lambda}\left(  R_{n}\right)  z_{1}^{\alpha_{1}}\ldots
z_{k+1}^{\alpha_{k+1}}
 =\left(  1-r_{1,\ldots,k+1}z_{1}\ldots z_{k+1}\right)  ^{-\lambda}
 =\sum_{l=0}^{\infty}\frac{\left(  \lambda\right)  _{l}r_{1,\ldots,k+1}^{l}%
}{l!}z_{1}^{l}\ldots z_{k+1}^{l}.
$
Then we have $c_{\left(  l\mathbf{1}_{k+1},\mathbf{0}_{n-k+1}\right)
,\lambda}=\frac{\left(  \lambda\right)  _{l}r_{1,\ldots,k+1}^{l}}{l!}%
\neq0,l\in\mathbb{N}.$ Finally, $\forall\alpha_{1}\in\mathbb{N},$
$\exists\mathbb{\alpha}_{\left[  n\right]  \smallsetminus\left\{  1\right\}
}=\left(  \alpha_{1}\mathbf{1}_{k},\mathbf{0}_{n-k-1}\right)  $ such that
$c_{(\alpha_{1},\mathbb{\alpha}_{\left[  n\right]  \smallsetminus\left\{
1\right\}  }),\lambda}=\frac{\left(  \lambda\right)  _{l}r_{1,\ldots,k+1}^{l}%
}{l!}\neq0,$ and $I_{1}\left(  R_{n}\right)  =\mathbb{N}.$ \newline%
\textbf{Thirdly. }Now, if $\forall T\in\mathfrak{P}_{n},\left\vert T\right\vert
\leqslant n;$ $r_{T}=0,$ then $R_{n}=0,$and $P_{n}\left(  \boldsymbol{\theta
}_{n}\right)  =p_{\left[  n\right]  }\prod_{i=1}^{n}\left(  \theta
_{i}-\widetilde{p}_{i}\right)  =\prod_{i=1}^{n}\left(  1+\left(
-\widetilde{p}_{i}\right)  ^{-1}\theta_{i}\right)  ,$ because $P_{n}\left(
\boldsymbol{0}_{n}\right)  =1$. In this case we have $\boldsymbol{X}%
=(X_{1},\ldots,X_{n}),$ with $X_{i}\sim\gamma_{\left(  -\widetilde{p}%
_{i}\right)  ^{-1},\lambda},$ $X_{i},i\in\left[  n\right]  $ being independent
Therefore, $c_{\left(  \alpha_{1}\boldsymbol{,}\ldots,\alpha_{n},\right)
,\lambda}\left(  P_{n}\right)  =0$ unless $c_{\left(  \mathbf{0}_{n},\right)
,\lambda}\left(  P_{n}\right)  =1$, and $I_{1}\left(  R_{n}\right)  =\left\{
0\right\} . $  This completes the proof by finite induction.
\end{proof}

\begin{proof}
[Proof of Theorem \ref{ThLXn_1_X1}]We have by (\ref{S2}), (\ref{B2}) and
(\ref{dzeta_}): 
$
B\left(  \boldsymbol{\theta}_{\left[  n\right]  \smallsetminus\left[
1\right]  }\right)     =-\frac{1}{p_{1}}[P_{n}\left(  0,\boldsymbol{\theta
}_{\left[  n\right]  \smallsetminus\left[  1\right]  }\right)  +\widetilde{p}%
_{1}\frac{\partial P_{n}}{\partial\theta_{1}}\left(  0,\boldsymbol{\theta
}_{\left[  n\right]  \smallsetminus\left[  1\right]  }\right)  ]
  =-\widetilde{p_{1}}S_{\left[  n\right]  \smallsetminus\left[  1\right]
}\left(  \boldsymbol{\theta}_{\left[  n\right]  \smallsetminus\left[
1\right]  }\right)  -\frac{1}{p_{1}}P_{n}\left(  0,\boldsymbol{\theta
}_{\left[  n\right]  \smallsetminus\left[  1\right]  }\right).
$
We deduce 
$
\mathfrak{z}_{n-1}\left(  \boldsymbol{\theta}_{\left[  n\right]
\smallsetminus\left[  1\right]  }\right)  =-\widetilde{p_{1}}-\frac{1}{p_{1}%
}\frac{P_{n}\left(  0,\boldsymbol{\theta}_{\left[  n\right]  \smallsetminus
\left[  1\right]  }\right)  }{S_{\left[  n\right]  \smallsetminus\left[
1\right]  }\left(  \boldsymbol{\theta}_{\left[  n\right]  \smallsetminus
\left[  1\right]  }\right)  }.
$
Then, by (\ref{LXn_1_X1_Yn_1}), we get\newline
$
L_{\mathbf{X}_{\left[  n\right]  \smallsetminus\left[  1\right]  }}%
^{X_{1}=x_{1}}\left(  \boldsymbol{\theta}_{\left[  n\right]  \smallsetminus
\left[  1\right]  }\right)  =[S_{\left[  n\right]  \smallsetminus\left[
1\right]  }\left(  \boldsymbol{\theta}_{\left[  n\right]  \smallsetminus
\left[  1\right]  }\right)  ]^{-\lambda}\frac{\exp\{[-\widetilde{p_{1}}%
-\frac{1}{p_{1}}\frac{P_{n}\left(  0,\boldsymbol{\theta}_{\left[  n\right]
\smallsetminus\left[  1\right]  }\right)  }{S_{\left[  n\right]
\smallsetminus\left[  1\right]  }\left(  \boldsymbol{\theta}_{\left[
n\right]  \smallsetminus\left[  1\right]  }\right)  }]x_{1}\}}{\exp
\{[-\widetilde{p}_{1}-\frac{1}{p_{1}}]x_{1}\}},
$
which gives (\ref{LXn_1_Yn_1s}). By (\ref{P_nTaylor}) and (\ref{S_n_1_rT}),
expanding $P_{n}\left(  0,\boldsymbol{\theta}_{\left[  n\right]
\smallsetminus\left[  1\right]  }\right)  -S_{\left[  n\right]  \smallsetminus
\left[  1\right]  }\left(  \boldsymbol{\theta}_{\left[  n\right]
\smallsetminus\left[  1\right]  }\right)  ,$ we obtain (\ref{LXn_1_Yn_1s_d}%
). By (\ref{P_nTaylor}) and (\ref{S_n_1_rT}) we get\newline
$
 -(\tfrac{P_{n}\left(  0,\boldsymbol{\theta}_{\left[  n\right]
\smallsetminus\left[  1\right]  }\right)  }{S_{\left[  n\right]
\smallsetminus\left[  1\right]  }\left(  \boldsymbol{\theta}_{\left[
n\right]  \smallsetminus\left[  1\right]  }\right)  }-1)\frac{x_{1}}{p_{1}}
 =
(\{\widetilde{p}_{1}-\widetilde{p}_{1}\sum_{T\subset\left[  n\right]
\smallsetminus\left\{  1\right\}  ,\left\vert T\right\vert >1}r_{T}[\left(
\boldsymbol{\theta}_{\left[  n\right]  }-\boldsymbol{\theta}_{P_{n}}\right)
_{\left[  n\right]  \smallsetminus\left\{  1\right\}  }]^{-T}+\\\sum
_{T\subset\left[  n\right]  \smallsetminus\left\{  1\right\}  ,\left\vert
T\right\vert >0}r_{\left\{  1\right\}  \cup T}[\left(  \boldsymbol{\theta
}_{\left[  n\right]  }-\boldsymbol{\theta}_{P_{n}}\right)  _{\left[  n\right]
\smallsetminus\left\{  1\right\}  }]^{-T}\\
 +\frac{1}{p_{1}}-\frac{1}{p_{1}}\sum_{T\subset\left[  n\right]
\smallsetminus\left\{  1\right\}  ,\left\vert T\right\vert >1}r_{T}[\left(
\boldsymbol{\theta}_{\left[  n\right]  }-\boldsymbol{\theta}_{P_{n}}\right)
_{\left[  n\right]  \smallsetminus\left\{  1\right\}  }]^{-T}\})p_{\left[
n\right]  }\left(  \boldsymbol{\theta}_{\left[  n\right]  }-\boldsymbol{\theta
}_{P_{n}}\right)  ^{\left[  n\right]  \smallsetminus\left\{  1\right\}  }%
\frac{x_{1}}{p_{1}}S_{\left[  n\right]  \smallsetminus\left[  1\right]  }%
^{-1}\left(  \theta_{2},\ldots,\theta_{n}\right),
$ \newline
then we have

$
 -\left(  \frac{P_{n}\left(  0,\theta_{2},\ldots,\theta_{n}\right)
}{S_{\left[  n\right]  \smallsetminus\left[  1\right]  }\left(  \theta
_{2},\ldots,\theta_{n}\right)  }-1\right)  \frac{x_{1}}{p_{1}}\\
  =\left\{
\begin{array}
[c]{c}%
\left(  \widetilde{p}_{1}+\frac{1}{p_{1}}\right)  \left(  \boldsymbol{\theta
}_{\left[  n\right]  }-\boldsymbol{\theta}_{P_{n}}\right)  ^{\left[  n\right]
\smallsetminus\left\{  1\right\}  }+\sum_{k=2}^{n}r_{\left\{  1,k\right\}
}\left(  \boldsymbol{\theta}_{\left[  n\right]  }-\boldsymbol{\theta}_{P_{n}%
}\right)  ^{\left[  n\right]  \smallsetminus\left\{  1,k\right\}  }\\
+\sum_{T\subset\left[  n\right]  \smallsetminus\left\{  1\right\}  ,\left\vert
T\right\vert >1}\left[  r_{\left\{  1\right\}  \cup T}-\left(  \widetilde{p}%
_{1}+\frac{1}{p_{1}}\right)  r_{T}\right]  \left(  \boldsymbol{\theta
}_{\left[  n\right]  }-\boldsymbol{\theta}_{P_{n}}\right)  ^{\left[  n\right]
\smallsetminus\left\{  1\right\}  \smallsetminus T}%
\end{array}
\right\}  \frac{p_{\left[  n\right]  }}{p_{1}}x_{1}S_{\left[  n\right]
\smallsetminus\left[  1\right]  }^{-1}\left(  \theta_{2},\ldots,\theta
_{n}\right)
$
and because \newline
$
 S_{\left[  n\right]  \smallsetminus\left[  1\right]  }\left(  \theta
_{2},\ldots,\theta_{n}\right) 
  =\frac{p_{\left[  n\right]  }}{p_{1}}\left(  \boldsymbol{\theta}_{\left[
n\right]  }-\boldsymbol{\theta}_{P_{n}}\right)  ^{\left[  n\right]
\smallsetminus\left\{  1\right\}  }\{1-R_{n-1}(S_{\left[  n\right]
\smallsetminus\left[  1\right]  }^{-1})[\left(  \boldsymbol{\theta}_{\left[
n\right]  }-\boldsymbol{\theta}_{P_{n}}\right)  _{\left[  n\right]
\smallsetminus\left\{  1\right\}  }]^{-1}\}\\
 =\frac{p_{\left[  n\right]  }}{p_{1}}\{\left(  \boldsymbol{\theta}_{\left[
n\right]  }-\boldsymbol{\theta}_{P_{n}}\right)  ^{\left[  n\right]
\smallsetminus\left\{  1\right\}  }-\sum_{T\subset\left[  n\right]
\setminus\left\{  1\right\}  ,\left\vert T\right\vert >1}r_{T}\left(
\boldsymbol{\theta}_{\left[  n\right]  }-\boldsymbol{\theta}_{P_{n}}\right)
^{\left[  n\right]  \smallsetminus\left\{  1\right\}  \smallsetminus T}\},
$\newline
 we get%
\begin{align}
&  \left(  \boldsymbol{\theta}_{\left[  n\right]  }-\boldsymbol{\theta}%
_{P_{n}}\right)  ^{\left[  n\right]  \smallsetminus\left\{  1\right\}
} =\frac{p_{1}}{p_{\left[  n\right]  }}\left(  -\widetilde{p}_{2}\right)
^{-1}\ldots\left(  -\widetilde{p}_{n}\right)  ^{-1}+\nonumber\\
& \hspace{-0.3cm} \sum_{T\subset\left[  n\right]  \setminus\left\{  1\right\}  ,\left\vert
T\right\vert >1}r_{T}\left(  -\boldsymbol{\theta}_{P_{n}}\right)  ^{-T}\left(
1+\left(  -\widetilde{p}_{1}\right)  ^{-1}\theta_{1},\ldots,1+\left(
-\widetilde{p}_{n}\right)  ^{-1}\theta_{n}\right)  ^{\left\{  \left[
n\right]  \setminus\left\{  1\right\}  \right\}  \setminus T}S_{\left[
n\right]  \smallsetminus\left[  1\right]  }^{-1}\left(  \theta_{2}%
,\ldots,\theta_{n}\right)  . \label{S_n_1_1development}%
\end{align}
Now, we have 
$
 -\left(  \frac{P_{n}\left(  0,\theta_{2},\ldots,\theta_{n}\right)
}{S_{\left[  n\right]  \smallsetminus\left[  1\right]  }\left(  \theta
_{2},\ldots,\theta_{n}\right)  }-1\right)  \frac{x_{1}}{p_{1}}=\\
  x_{1}\left(  \widetilde{p}_{1}+\frac{1}{p_{1}}\right)  +\sum
_{T\subset\left[  n\right]  \smallsetminus\left\{  1\right\}  ,\left\vert
T\right\vert >0}r_{\left\{  1\right\}  \cup T}\frac{p_{\left[  n\right]  }%
}{p_{1}}x_{1}\left(  \boldsymbol{\theta}_{\left[  n\right]  }%
-\boldsymbol{\theta}_{P_{n}}\right)  ^{\left[  n\right]  \smallsetminus
\left\{  1\right\}  \smallsetminus T}S_{\left[  n\right]  \smallsetminus
\left[  1\right]  }^{-1}\left(  \theta_{2},\ldots,\theta_{n}\right)
$
and \newline
$
  L_{\left(  X_{2},\ldots,X_{n}\right)  }^{X_{1}=x_{1}}\left(  \theta
_{2},\ldots,\theta_{n}\right)
  =[S_{\left[  n\right]  \smallsetminus\left[  1\right]  }\left(  \theta
_{2},\ldots,\theta_{n}\right)  ]^{-\lambda}\times\\
  \exp\left\{  x_{1}\left(  \widetilde{p}_{1}+\frac{1}{p_{1}}\right)
+\sum_{T\subset\left[  n\right]  \smallsetminus\left\{  1\right\}  ,\left\vert
T\right\vert >0}r_{\left\{  1\right\}  \cup T}\frac{p_{\left[  n\right]  }%
}{p_{1}}x_{1}\left(  \boldsymbol{\theta}_{\left[  n\right]  }%
-\boldsymbol{\theta}_{P_{n}}\right)  ^{\left[  n\right]  \smallsetminus
\left\{  1\right\}  \smallsetminus T}S_{\left[  n\right]  \smallsetminus
\left[  1\right]  }^{-1}\left(  \theta_{2},\ldots,\theta_{n}\right)  \right\}
.$ \newline
Because
$
1   =L_{\left(  X_{2},\ldots,X_{n}\right)  }^{X_{1}=x_{1}}\left(
0,\ldots,0\right) =\\
 \exp\left\{  \left(  \widetilde{p}_{1}+\frac{1}{p_{1}}\right)
x_{1}\right\}  \exp\left\{  \sum_{T\subset\left[  n\right]  \smallsetminus
\left\{  1\right\}  ,\left\vert T\right\vert >0}r_{\left\{  1\right\}  \cup
T}\frac{p_{\left[  n\right]  }}{p_{1}}x_{1}\left(  -\widetilde{p}_{2}%
,\ldots,-\widetilde{p}_{n}\right)  ^{\left\{  \left[  n\right]  \setminus
\left\{  1\right\}  \right\}  \setminus T}\right\},
$ we have \newline
$
  L_{\left(  X_{2},\ldots,X_{n}\right)  }^{X_{1}=x_{1}}\left(
\boldsymbol{\theta}_{\left[  n\right]  \smallsetminus\left[  1\right]
}\right)=S_{\left[  n\right]  \smallsetminus\left[  1\right]  }^{-\lambda}\left(
\boldsymbol{\theta}_{\left[  n\right]  \smallsetminus\left[  1\right]
}\right)\times\\
 \mathbf{e}^{\{x_{1}(\widetilde{p}_{1}+\tfrac{1}{p_{1}})+\sum
_{T\subset\left[  n\right]  \ \smallsetminus\left\{  1\right\}  \ ,\left\vert
T\right\vert >0}r_{\left\{  1\right\}  \cup T}\tfrac{p_{\left[  n\right]  }%
}{p_{1}}x_{1}\left(  \boldsymbol{\theta}_{\left[  n\right]  }%
-\boldsymbol{\theta}_{P_{n}}\right)  ^{\left[  n\right]  \smallsetminus
\left\{  1\right\}  \smallsetminus T}S_{\left[  n\right]  \smallsetminus
\left[  1\right]  }^{-1}\left(  \theta_{2},\ldots,\theta_{n}\right)  \}}\\
  =S_{\left[  n\right]  \smallsetminus\left[  1\right]  }^{-\lambda}\left(
\boldsymbol{\theta}_{\left[  n\right]  \smallsetminus\left[  1\right]
}\right) \times \\
  \mathbf{e}^{\{\sum_{T\subset\left[  n\right]  \ \smallsetminus
\left\{  1\right\}  \ ,0<\left\vert T\right\vert }r_{\left\{  1\right\}  \cup
T}\tfrac{p_{\left[  n\right]  }}{p_{1}}x_{1}\left(  -\boldsymbol{\theta
}_{P_{n}}\right)  ^{\left\{  \left[  n\right]  \smallsetminus\left\{
1\right\}  \right\}  \smallsetminus T}[(\mathbf{1}_{n}+\left(
-\boldsymbol{\theta}_{P_{n}}\right)  ^{-1}\boldsymbol{\theta}_{\left[
n\right]  })^{\left[  n\right]  \smallsetminus\left\{  1\right\}
\smallsetminus T}S_{\left[  n\right]  \smallsetminus\left[  1\right]  }%
^{-1}\left(  \boldsymbol{\theta}_{\left[  n\right]  \smallsetminus\left[
1\right]  }\right)  -1]\}}.
$  \newline
Unless $Rn=0$, for $T\in\mathfrak{P([}n]\smallsetminus\lbrack1]),\overline
{T}=\mathfrak{[}n]\smallsetminus T,$ and if there is no ambiguity, for
simplicity we denote $S_{T}\left(  \boldsymbol{\theta}_{T}\right)  $ by
$S_{T},$ the polynomial defined by 
$
S_{T}\left(  \boldsymbol{\theta}_{T}\right)    =\frac{1}{p_{\overline{T}}%
}\left(  \frac{\partial}{\partial\boldsymbol{\theta}}\right)  ^{\overline{T}%
}\left(  P_{n}\left(  \boldsymbol{\theta}\right)  \right) 
 =\sum_{T^{\prime}\in P\left(  T\right)  }\frac{p_{\overline{T}\cup
T^{\prime}}}{p_{\overline{T}}}\boldsymbol{\theta}^{T^{\prime}}%
$

We have
$
S_{T}\left(  \boldsymbol{\theta}_{T}\right)  =\sum_{T^{\prime}\in
\mathfrak{P}\left(  T\right)  }\frac{p_{\overline{T}\cup T^{\prime}}%
}{p_{\overline{T}}}\boldsymbol{\theta}^{T^{\prime}}=\sum_{T^{\prime}\in
P\left(  T\right)  }q_{T^{\prime}}\boldsymbol{\theta}^{T^{\prime}},
$
with $q_{T^{\prime}}=\frac{p_{\overline{T}\cup T^{\prime}}}{p_{\overline{T}}%
},$ and for $T^{\prime}\in\mathfrak{P}\left(  T\right)  ,$ we have 
$
\widetilde{q}_{T^{\prime}}=-\frac{q_{T\smallsetminus T^{\prime}}}{q_{T}}%
=\frac{p_{\overline
{T^{\prime}}}}{p_{\left[  n\right]  }}=\widetilde{p}_{T^{\prime}}.
$
Therefore, we have 
$
\widetilde{b}_{T^{\prime}}\left(  S_{T}\right)  =\widetilde{b}_{T^{\prime}},
$
and if $\boldsymbol{\gamma}_{\left(  P,\lambda\right)  }$ is an infinitely
divisible gamma distribution, $\boldsymbol{\gamma}_{\left(  S_{T} 
,\lambda\right)  }$ is also an infinitely divisible gamma distribution.

\end{proof}

\begin{proof}
[Proof of Corollary \ref{cor1}]From Proof of Corollary (\ref{cor_cor}), we
have $c_{\boldsymbol{\alpha},\lambda}\left(  R_{n}\right)  =0$ unless
$\boldsymbol{\alpha}=k\mathbf{1}_{n},$ $k\in\mathbb{N}$ in which case
$c_{k\mathbf{1}_{n},\lambda}\left(  R_{n}\right)  =\frac{\left(
\lambda\right)  _{k}}{k!}\left(  qp^{-n}\right)  ^{k}.$ Therefore, we have \newline
$
L_{\mathbf{X}_{\left[  n\right]  \mathbf{\smallsetminus}\left[  k\right]  }%
}^{\mathbf{X}_{\left[  k\right]  }=\mathbf{x}_{\left[  k\right]  }}\left(
\boldsymbol{\theta}_{\left[  n\right]  \mathbf{\smallsetminus}\left[
k\right]  }\right)    =[\prod_{i=k+1}^{n}\left(  1+p\theta_{i}\right)
^{-\lambda}]\frac{\sum_{k=0}^{\infty}\frac{\left(  \mathbf{x}_{k}\right)
^{k\mathbf{1}_{k}}}{\left\{  \left(  \lambda\right)  _{\boldsymbol{k}%
}\right\}  ^{k}}\frac{\left(  \lambda\right)  _{k}}{k!}\left(  qp^{-n}\right)
^{k}\left(  \boldsymbol{\theta}_{\left[  n\right]  }+\frac{1}{p}\mathbf{1}%
_{n}\right)  _{n-k}^{-k\mathbf{1}_{n-k}}}{\sum_{k=0}^{\infty}\frac{\left(
\mathbf{x}_{k}\right)  ^{k\mathbf{1}_{k}}}{\left\{  \left(  \lambda\right)
_{\boldsymbol{k}}\right\}  ^{k}}\frac{\left(  \lambda\right)  _{k}}{k!}\left(
qp^{-n}\right)  ^{k}\left(  \frac{1}{p}\mathbf{1}_{n}\right)  _{n-k}%
^{-k\mathbf{1}_{n-k}}}\\
=[\prod_{i=k+1}^{n}\left(  1+p\theta_{i}\right)  ^{-\lambda}]\frac
{\sum_{k=0}^{\infty}\frac{1}{\left\{  \left(  \lambda\right)  _{\boldsymbol{k}%
}\right\}  ^{k-1}}\frac{1}{k!}\left(  qp^{-k}\mathbf{x}_{k}^{\left[  k\right]
}\prod_{i=k+1}^{n}\left(  1+p\boldsymbol{\theta}_{i}\right)  ^{-1}\right)
^{k}}{\sum_{k=0}^{\infty}\frac{1}{\left\{  \left(  \lambda\right)
_{\boldsymbol{k}}\right\}  ^{k}}\frac{\left(  \lambda\right)  _{k}}{k!}\left(
qp^{-k}\mathbf{x}_{k}^{\left[  k\right]  }\right)  ^{k}}%
$
and the definition of $F_{k-1}$ gives (\ref{LCp}).
\end{proof}

\begin{proof}
[Proof of Corollary \ref{cor1_k_1}]Doing $k=1$ in Equality (\ref{LCp}), we get
(\ref{LCp_k_1}). Equality (\ref{LCp_k_1_2}) comes from the definition of
$\exp.$ A second proof can be given by application of Theorem
\ref{L_X1_n_1}. We successively have $\mathbf{G}\left(  R_{n},u_{1}\right)
=\sum_{\alpha_{1}\in\mathbb{N}}\frac{u_{1}^{\alpha_{1}}}{\alpha_{1}!}%
=\exp\left(  u_{1}\right)  ,$ $S_{n-1}\left(  \boldsymbol{\theta}_{\left[
n\right]  \smallsetminus\left[  1\right]  }\right)  =\prod_{i=2}^{n}\left(
1+p\theta_{i}\right)  ,$ $B\left(  \boldsymbol{\theta}_{\left[  n\right]
\smallsetminus\left[  1\right]  }\right)  =qp^{-1}$ and 
$\mathfrak{z}\left(  \boldsymbol{\theta}_{\left[  n\right]  \smallsetminus
\left[  1\right]  }\right)  =qp^{-1}[\prod_{i=2}^{n}\left(  1+p\theta
_{i}\right)  ]^{-1}$. As a result, by application of (\ref{LXn_1_X1_Yn_1}) we
get (\ref{LCp_k_1}). Another proof of this result is given by
(\ref{LXn_1_X1_Yn_1s_d_main_s}) as follows. In this case, we have $P_{n}\left(
\boldsymbol{\theta}\right)  =\frac{-q}{p}+\frac{1}{p}\prod_{i=1}^{n}\left(
1+p\theta_{i}\right)  ,$ $p_{T}=p^{\left\vert T\right\vert -1},p_{\left[
n\right]  }=p^{n-1},\widetilde{p}_{T}=-p^{-\left\vert T\right\vert },$
$\widetilde{p}_{\left\{  i\right\}  }=-p^{-1},\boldsymbol{\theta}_{P_{n}%
}=\widetilde{\mathbf{p}}=-p^{-1}\mathbf{1}_{n}$,  $-\widetilde{p}_{\left[
n\right]  \mathbf{\smallsetminus}\left[  1\right]  }=p^{-(n-1)},S_{\left[
n\right]  \smallsetminus\left\{  1\right\}  }\left(  \boldsymbol{\theta
}_{\left[  n\right]  \smallsetminus\left[  1\right]  }\right)  =\prod
_{i=2}^{n}\left(  1+p\theta_{i}\right)  $, and \newline
$
P_{n}\left(  \boldsymbol{\theta}\right)  =p_{\left[  n\right]  }\prod
_{i=1}^{n}\left(  \theta_{i}+\left(  -p^{-1}\right)  \right)  [1-qp^{-n}%
\prod_{i=1}^{n}\left(  \theta_{i}+\left(  -p^{-1}\right)  \right)  ^{-1}]
.$
Hence we have $r_{T}=0,$ if $1<\left\vert T\right\vert <n$ and $r_{\left[
n\right]  }=qp^{-n}.$ We deduce from (\ref{LXn_1_X1_Yn_1s_d_main_s}):  
$
L_{\mathbf{X}_{\left[  n\right]  \smallsetminus\left[  1\right]  }}%
^{X_{1}=x_{1}}   =S_{\left[  n\right]  \smallsetminus\left[  1\right]
}^{-\lambda}\mathbf{e}^{\{qp^{-1}x_{1}[S_{\left[  n\right]  \smallsetminus
\left[  1\right]  }^{-1}-1]\}}
  =\sum_{k=0}^{\infty}\frac{(qp^{-1}x_{1})^{k}}{k!}\exp(-qp^{-1}%
x_{1})S_{\left[  n\right]  \smallsetminus\left[  1\right]  }^{-(\lambda+k)}.
$

\end{proof}

\begin{proof}
[Proof of Theorem \ref{Thequicor}]Indeed, $X_{1}\sim\gamma_{(p,\lambda)},$ let
$V_{1}\sim\mathcal{P}\left(  qp^{-1}X_{1}\right)  $, we have $P\left(
V_{1}=k|X_{1}=x_{1}\right)  =\frac{\left(  qp^{-1}x_{1}\right)  ^{k}}{k!}%
\exp\left(  -qp^{-1}x_{1}\right)  .$ Clearly the variable $X_{i},i=2\ldots,n$
are conditionally independent and $X_{i}|\left(  X_{1}=x_{1}\right)
\sim\gamma_{(p,\lambda+V_{1})}.$ We have\newline
$
L_{\mathbf{X}_{\left[  n\right]  \mathbf{\smallsetminus}\left[  1\right]  }%
}^{X_{1}=x_{1}}\left(  \boldsymbol{\theta}_{\left[  n\right]
\mathbf{\smallsetminus}\left[  1\right]  }\right)  =\sum_{v_{1}=0}^{\infty
}P\left(  V_{1}=v_{1}|X_{1}=x_{1}\right)  L_{\mathbf{X}_{\left[  n\right]
\mathbf{\smallsetminus}\left[  1\right]  }}^{X_{1}=x_{1}}\left(
\boldsymbol{\theta}_{\left[  n\right]  \mathbf{\smallsetminus}\left[
1\right]  }\right)  ,
$
and Formula (\ref{LCp_k_1_2}) is verified. Finally we have $\mathbf{X}%
_{\left[  n\right]  }\sim\boldsymbol{\gamma}_{(P_{n},\lambda)},$ with
$P_{n}\left(  \boldsymbol{\theta}_{\left[  n\right]  }\right)  =\frac{-q}%
{p}+\frac{1}{p}\prod_{i=1}^{n}\left(  1+p\theta_{i}\right)  .$
\end{proof}

\begin{proof}
[Proof of Theorem \ref{Th14Ber(2023)}]From (\ref{LXn_1_X1_Yn_1s_d_main_s}), we
deduce the Lt of the conditional distribution of $X_{2}|X_{1}=x_{1}$, for
$\left(  X_{1},X_{2}\right)  \sim\boldsymbol{\gamma}_{\left(  P_{2}%
,\lambda\right)  }$, with $P_{2}\left(  \theta_{1},\theta_{2}\right)
=1+p_{1}\theta_{1}+p_{2}\theta_{2}+p_{1,2}\theta_{1}\theta_{2}$ with
$p_{1}>0,$ $p_{2}>0$, $p_{1,2}>0$, $\widetilde{b}_{1,2}=\widetilde{b}%
_{1,2}\left(  P_{2}\right)  =p_{1}p_{2}/p_{1,2}^{2}-1/p_{1,2}>0$. We have
$S_{2}\left(  \theta_{2}\right)  =1+\frac{p_{1,2}}{p_{1}}\theta_{2}$ and%
\begin{equation}
L_{X_{2}}^{X_{1}=x_{1}}=S_{2}^{-\lambda}\exp\{\frac{\widetilde{b}_{1,2}%
}{(-\widetilde{p}_{2})}x_{1}(S_{2}^{-1}-1)\}. \label{L_X2_X1_n=2}%
\end{equation}
By expanding the exponential function, we obtain the following expansion:
\begin{equation}
L_{X_{2}}^{X_{1}=x_{1}}=\sum_{v_{1}\in\mathbb{N}}\mathbf{P}\left(  V_{1}%
=v_{1}\right)  S_{2}^{-(\lambda+v_{1})}. \label{LCn=2}%
\end{equation}
Equality (\ref{LCn=2}) proves that $X_{2}|X_{1}$ has distribution
$\gamma_{(S_{2},\lambda+V_{1})}$, then $\mathbf{X}_{\left[  2\right]
}=\left(  X_{1},X_{2}\right)  \sim\mathbf{\gamma}_{\left(  P_{2}%
,\lambda\right)  }$.
\end{proof}

\begin{proof}
[Proof of Theorem \ref{Th1_n=3_Bern2024}]Hence for $n=3$, we have from
definition of $S_{T},$ $S_{2,3}\left(  \theta_{2},\theta_{3}\right)
=1+\frac{p_{1,2}}{p_{1}}\theta_{2}+\frac{p_{1,3}}{p_{1}}\theta_{3}%
+\frac{p_{1,2,3}}{p_{1}}\theta_{2}\theta_{3},S_{2}\left(  \theta_{2}\right)
=1+\frac{p_{1,2,3}}{p_{1,3}}\theta_{2},S_{3}\left(  \theta_{3}\right)
=1+\frac{p_{1,2,3}}{p_{1,2}}\theta_{3},$ and from
(\ref{LXn_1_X1_Yn_1s_d_main_s_C}) and (\ref{rT3}), we get%
\begin{equation}
L_{(X_{2},X_{3})}^{X_{1}=x_{1}}=S_{2,3}^{-\lambda}\exp\{\left(  -\widetilde{p}%
_{2,3}\right)  ^{-1}x_{1}[\widetilde{b}_{\left\{  1,2\right\}  }\left(
-\widetilde{p}_{3}\right)  S_{3}S_{2,3}^{-1}+\widetilde{b}_{\left\{
1,3\right\}  }\left(  -\widetilde{p}_{2}\right)  S_{2}S_{2,3}^{-1}%
+\widetilde{b}_{\left\{  1,2,3\right\}  }S_{2,3}^{-1}-C]\}
\label{L_X2_X_3_X1_n=3}%
\end{equation}
To compute another expression for $L_{\left(  X_{2},X_{3}\right)  }%
^{X_{1}=x_{1}}$ we use (\ref{S^T}), then we have $\left(  -\widetilde{p}_{2}\right)  \left(  -\widetilde{p}_{3}\right)
S_{2}S_{3}=\left(  -\widetilde{p}_{2,3}\right)  S_{2,3}+\widetilde{b}_{2,3},
$

By respectively dividing the last equality by $S_{3}S_{2,3}$ and $S_{2}S_{2,3}$,
we successively obtain
$\left(  -\widetilde{p}_{3}\right)  \left(  -\widetilde{p}_{2}\right)
S_{2}S_{2,3}^{-1}  =\left(  -\widetilde{p}_{2,3}\right)  S_{3}
^{-1}+\widetilde{b}_{2,3}S_{3}^{-1}S_{2,3}^{-1}$ and
$\left(  -\widetilde{p}_{2}\right)  \left(  -\widetilde{p}_{3}\right)
S_{3}S_{2,3}^{-1}   =\left(  -\widetilde{p}_{2,3}\right)  S_{2}
^{-1}+\widetilde{b}_{2,3}S_{2}^{-1}S_{2,3}^{-1}.$
Using the two last equalities into (\ref{L_X2_X_3_X1_n=3}),
we get
$
L_{(X_{2},X_{3})}^{X_{1}=x_{1}}=S_{2,3}^{-\lambda}\exp\{x_{1}[\alpha_{1}%
S_{2}^{-1}+\alpha_{2}S_{3}^{-1}+\alpha_{3}S_{2,3}^{-1}+\alpha_{4}S_{2}%
^{-1}S_{2,3}^{-1}+\alpha_{5}S_{3}^{-1}S_{2,3}^{-1}-C)]\},
$
and by the condition $L_{(X_{2},X_{3})}^{X_{1}=x_{1}}\left(  0,0\right)  =1$,
we obtain \newline
$
L_{(X_{2},X_{3})}^{X_{1}=x_{1}}=S_{2,3}^{-\lambda}\exp\{x_{1}[\alpha_{1}%
(S_{2}^{-1}-1)+\alpha_{2}(S_{3}^{-1}-1)+\alpha_{3}(S_{2,3}^{-1}-1)+\alpha
_{4}(S_{2}^{-1}S_{2,3}^{-1}-1)+\alpha_{5}(S_{3}^{-1}S_{2,3}^{-1}-1)]\}.
$
Expanding $\exp$ in the last equality, we get (\ref{LCn=3}). Equality (\ref{LCn=3}) proves that
$\mathbf{X}_{\left[  3\right]  }
\sim\mathbf{\gamma}_{\left(  P_{3},\lambda\right)  }$.
\end{proof}

\begin{proof}
[Proof of Theorem \ref{Th2_n=3_Bern2024}]We use Theorem (\ref{Th14Ber(2023)}).
Let $X_{2}^{\prime}\sim\gamma_{\left(  \frac{p_{1,2}}{p_{1}},\lambda
+V_{3}+V_{4}+V_{5}\right)  },$ let $\alpha_{6}=\frac{\widetilde{b}%
_{2,3}\left(  S_{2,3}\right)  }{-\widetilde{p}_{3}\left(  S_{2,3}
\right)}=\frac{\widetilde{b}_{2,3}}{(-\widetilde{p}_{3})}$, and $V_{6}%
\sim\mathcal{P}\left(  \alpha_{6}X_{2}^{\prime}\right)  $, let $X_{3}^{\prime
}\sim\gamma_{\left(  S_{3},\lambda+V_{3}+V_{4}+V_{5}+V_{6}\right)  }$, then
$\left(  X_{2}^{\prime},X_{3}^{\prime}\right)  \sim\mathcal{\gamma}\left(
S_{2,3},\lambda+V_{3}+V_{4}+V_{5}\right)  .$
\end{proof}

\begin{proof}
[Proof of Theorem \ref{Th1_n=4_Bern2024}]For $n=4$, from
(\ref{LXn_1_X1_Yn_1s_d_main_s_C}), (\ref{rT4}) and (\ref{ST}), we get
$
\left(  -\widetilde{p}_{2,3}\right)  S_{2,3}=\left(  -\widetilde{p}%
_{2}\right)  \left(  -\widetilde{p}_{3}\right)  S_{2}S_{3}-\widetilde{b}%
_{2,3},$
$
\left(  -\widetilde{p}_{2,4}\right)  S_{2,4}=\left(  -\widetilde{p}%
_{2}\right)  \left(  -\widetilde{p}_{4}\right)  S_{2}S_{4}-\widetilde{b}%
_{2,4},$
$
\left(  -\widetilde{p}_{3,4}\right)  S_{3,4}=\left(  -\widetilde{p}%
_{3}\right)  \left(  -\widetilde{p}_{4}\right)  S_{3}S_{4}-\widetilde{b}%
_{3,4},$ and
\begin{align}
&  L_{\left(  X_{2},X_{3},X_{4}\right)  }^{X_{1}=x_{1}}=\nonumber\\
&  S_{2,3,4}^{-\lambda}\exp\{\left(  -\widetilde{p}_{2,3,4}\right)  ^{-1}%
x_{1}[\widetilde{b}_{1,2}(-\widetilde{p}_{3,4})S_{3,4}S_{2,3,4}^{-1}%
+\widetilde{b}_{1,3}(-\widetilde{p}_{2,4})S_{2,4}S_{2,3,4}^{-1}+\widetilde{b}%
_{1,4}(-\widetilde{p}_{2,3})S_{2,3}S_{2,3,4}^{-1}\nonumber\\
&  +\widetilde{b}_{1,2,3}\left(  -\widetilde{p}_{4}\right)  S_{4}%
S_{2,3,4}^{-1}+\widetilde{b}_{1,2,4}\left(  -\widetilde{p}_{3}\right)
S_{3}S_{2,3,4}^{-1}+\widetilde{b}_{1,3,4}\left(  -\widetilde{p}_{2}\right)
S_{2}S_{2,3,4}^{-1}+\widetilde{b}_{1,2,3,4}S_{2,3,4}^{-1}-C]\}
\label{LX2,3,4|X1}%
\end{align}
To compute $L_{\left(  X_{2},X_{3},X_{4}\right)  }^{X_{1}=x_{1}},$ we need to
express the following expressions in terms of inverses of $S_{T}
,T\in\mathfrak{P}\left(  \left\{  2,3,4\right\}  \right)  $ 
$
S_{3,4}S_{2,3,4}^{-1},S_{2,4}S_{2,3,4}^{-1},S_{2,3}S_{2,3,4}^{-1}%
,S_{4}S_{2,3,4}^{-1},S_{3}S_{2,3,4}^{-1},S_{2}S_{2,3,4}^{-1}.
$
From (\ref{S^T}), we then successively have
\begin{equation}
\left(  -\widetilde{p}_{2}\right)  \left(  -\widetilde{p}_{3}\right)
S_{2}S_{3}=\left(  -\widetilde{p}_{2,3}\right)  S_{2,3}+\widetilde{b}_{2,3},
\label{S^23}%
\end{equation}%

\begin{equation}
\left(  -\widetilde{p}_{2}\right)  \left(  -\widetilde{p}_{4}\right)
S_{2}S_{4}=\left(  -\widetilde{p}_{2,4}\right)  S_{2,4}+\widetilde{b}_{2,4},
\label{S^24}%
\end{equation}%

\begin{equation}
\left(  -\widetilde{p}_{3}\right)  \left(  -\widetilde{p}_{4}\right)
S_{3}S_{4}=\left(  -\widetilde{p}_{3,4}\right)  S_{3,4}+\widetilde{b}_{3,4},
\label{S^34}%
\end{equation}
According to (\ref{ST}), we also have successively \newline
$
S_{2,3,4}   =\left(  -\widetilde{p}_{2,3,4}\right)  ^{-1}\left[  \left(  -\widetilde{p}%
_{2,3}\right)  S_{2,3}\left(  -\widetilde{p}_{4}\right)  S_{4}-\widetilde{b}%
_{2,4}\left(  -\widetilde{p}_{3}\right)  S_{3}-\widetilde{b}_{3,4}\left(
-\widetilde{p}_{2}\right)  S_{2}-\widetilde{b}_{2,3,4}\right]  ,\\
S_{2,3,4} =\left(  -\widetilde{p}_{2,3,4}\right)  ^{-1}\left[  \left(  -\widetilde{p}%
_{2,4}\right)  S_{2,4}\left(  -\widetilde{p}_{3}\right)  S_{3}-\widetilde{b}%
_{2,3}\left(  -\widetilde{p}_{4}\right)  S_{4}-\widetilde{b}_{3,4}\left(
-\widetilde{p}_{2}\right)  S_{2}-\widetilde{b}_{2,3,4}\right]  ,\\
S_{2,3,4} =\left(  -\widetilde{p}_{2,3,4}\right)  ^{-1}\left[  \left(  -\widetilde{p}%
_{3,4}\right)  S_{3,4}\left(  -\widetilde{p}_{2}\right)  S_{2}-\widetilde{b}%
_{2,3}\left(  -\widetilde{p}_{4}\right)  S_{4}-\widetilde{b}_{2,4}\left(
-\widetilde{p}_{3}\right)  S_{3}-\widetilde{b}_{2,3,4}\right]  ,
$ and
\begin{equation}
\left(  -\widetilde{p}_{4}\right)  \left(  -\widetilde{p}_{2,3}\right)
S_{4}S_{2,3}=\left(  -\widetilde{p}_{2,3,4}\right)  S_{2,3,4}+\widetilde{b}%
_{2,4}\left(  -\widetilde{p}_{3}\right)  S_{3}+\widetilde{b}_{3,4}\left(
-\widetilde{p}_{2}\right)  S_{2}+\widetilde{b}_{2,3,4}, \label{S4S23}%
\end{equation}%
\begin{equation}
\left(  -\widetilde{p}_{3}\right)  \left(  -\widetilde{p}_{2,4}\right)
S_{3}S_{2,4}=\left(  -\widetilde{p}_{2,3,4}\right)  S_{2,3,4}+\widetilde{b}%
_{2,3}\left(  -\widetilde{p}_{4}\right)  S_{4}+\widetilde{b}_{3,4}\left(
-\widetilde{p}_{2}\right)  S_{2}+\widetilde{b}_{2,3,4}, \label{S3S24}%
\end{equation}%
\begin{equation}
\left(  -\widetilde{p}_{2}\right)  \left(  -\widetilde{p}_{3,4}\right)
S_{2}S_{3,4}=\left(  -\widetilde{p}_{2,3,4}\right)  S_{2,3,4}+\widetilde{b}%
_{2,3}\left(  -\widetilde{p}_{4}\right)  S_{4}+\widetilde{b}_{2,4}\left(
-\widetilde{p}_{3}\right)  S_{3}+\widetilde{b}_{2,3,4}, \label{S2S34}%
\end{equation}
Dividing (\ref{S^23}) by $S_{3}S_{2,3}$ and $S_{2}S_{2,3}$,
(\ref{S^24}) by $S_{4}S_{2,4}$ and $S_{2}S_{2,4}$, and (\ref{S^34}) by
$S_{4}S_{3,4}$ and $S_{3}S_{3,4}$, we obtain successively 
\begin{align}
\left(  -\widetilde{p}_{2}\right)  \left(  -\widetilde{p}_{3}\right)
\tfrac{S_{2}}{S_{2,3}}  &  =\left(  -\widetilde{p}_{2,3}\right)  \tfrac{1}%
{S_{3}}+\widetilde{b}_{2,3}\tfrac{1}{S_{3}S_{2,3}},\label{S2/S23}\\
\left(  -\widetilde{p}_{2}\right)  \left(  -\widetilde{p}_{3}\right)
\tfrac{S_{3}}{S_{2,3}}  &  =\left(  -\widetilde{p}_{2,3}\right)  \tfrac{1}%
{S_{2}}+\widetilde{b}_{2,3}\tfrac{1}{S_{2}S_{2,3}}, \label{S3/S23}%
\end{align}
\begin{align}
\left(  -\widetilde{p}_{2}\right)  \left(  -\widetilde{p}_{4}\right)
\tfrac{S_{2}}{S_{2,4}}  &  =\left(  -\widetilde{p}_{2,4}\right)  \tfrac{1}%
{S_{4}}+\widetilde{b}_{2,4}\tfrac{1}{S_{4}S_{2,4}},\label{S2/S24}\\
\left(  -\widetilde{p}_{2}\right)  \left(  -\widetilde{p}_{4}\right)
\tfrac{S_{4}}{S_{2,4}}  &  =\left(  -\widetilde{p}_{2,4}\right)  \tfrac{1}%
{S_{2}}+\widetilde{b}_{2,4}\tfrac{1}{S_{2}S_{2,4}}, \label{S4/S24}%
\end{align}
\begin{align}
\left(  -\widetilde{p}_{3}\right)  \left(  -\widetilde{p}_{4}\right)
\tfrac{S_{3}}{S_{3,4}}  &  =\left(  -\widetilde{p}_{3,4}\right)  \tfrac{1}%
{S_{4}}+\widetilde{b}_{3,4}\tfrac{1}{S_{4}S_{3,4}},\label{S3/S34}\\
\left(  -\widetilde{p}_{3}\right)  \left(  -\widetilde{p}_{4}\right)
\tfrac{S_{4}}{S_{3,4}}  &  =\left(  -\widetilde{p}_{3,4}\right)  \tfrac{1}%
{S_{3}}+\widetilde{b}_{3,4}\tfrac{1}{S_{3}S_{3,4}}. \label{S4/S34}%
\end{align}
Dividing (\ref{S4S23}) by $S_{2,3}S_{2,3,4}$ and
$S_{4}S_{2,3,4}$, (\ref{S3S24}) by $S_{2,4}S_{2,3,4}$ and $S_{3}S_{2,3,4}$,
and (\ref{S2S34}) by $S_{3,4}S_{2,3,4}$ and $S_{2}S_{2,3,4}$, we obtain successively 
\begin{align}
&\hspace{-0.5cm}  \left(  -\widetilde{p}_{4}\right)  \left(  -\widetilde{p}_{2,3}\right)
\tfrac{S_{4}}{S_{2,3,4}}  =\left(  -\widetilde{p}_{2,3,4}\right)  \tfrac{1}{S_{2,3}}+\widetilde{b}%
_{2,4}\left(  -\widetilde{p}_{3}\right)  \tfrac{1}{S_{2,3,4}}\tfrac{S_{3}%
}{S_{2,3}}+\widetilde{b}_{3,4}\left(  -\widetilde{p}_{2}\right)  \tfrac
{1}{S_{2,3,4}}\tfrac{S_{2}}{S_{2,3}}+\widetilde{b}_{2,3,4}\tfrac{1}%
{S_{2,3}S_{2,3,4}},\label{S4/S234}\\
&  \left(  -\widetilde{p}_{4}\right)  \left(  -\widetilde{p}_{2,3}\right)
\tfrac{S_{2,3}}{S_{2,3,4}} =\left(  -\widetilde{p}_{2,3,4}\right)  \tfrac{1}{S_{4}}+\widetilde{b}%
_{2,4}\left(  -\widetilde{p}_{3}\right)  \tfrac{1}{S_{4}}\tfrac{S_{3}}%
{S_{2,3,4}}+\widetilde{b}_{3,4}\left(  -\widetilde{p}_{2}\right)  \tfrac
{1}{S_{4}}\tfrac{S_{2}}{S_{2,3,4}}+\widetilde{b}_{2,3,4}\tfrac{1}{S_{4}%
S_{2,3,4}}, \label{S23/S234}%
\end{align}%
\begin{align}
&  \hspace{-0.5cm}\left(  -\widetilde{p}_{3}\right)  \left(  -\widetilde{p}_{2,4}\right)
\tfrac{S_{3}}{S_{2,3,4}} =\left(  -\widetilde{p}_{2,3,4}\right)  \tfrac{1}{S_{2,4}}+\widetilde{b}%
_{2,3}\left(  -\widetilde{p}_{4}\right)  \tfrac{1}{S_{2,3,4}}\tfrac{S_{4}%
}{S_{2,4}}+\widetilde{b}_{3,4}\left(  -\widetilde{p}_{2}\right)  \tfrac
{1}{S_{2,3,4}}\tfrac{S_{2}}{S_{2,4}}+\widetilde{b}_{2,3,4}\tfrac{1}%
{S_{2,4}S_{2,3,4}},\label{S3/S234}\\
&  \left(  -\widetilde{p}_{3}\right)  \left(  -\widetilde{p}_{2,4}\right)
\tfrac{S_{2,4}}{S_{2,3,4}} =\left(  -\widetilde{p}_{2,3,4}\right)  \tfrac{1}{S_{3}}+\widetilde{b}%
_{2,3}\left(  -\widetilde{p}_{4}\right)  \tfrac{1}{S_{3}}\tfrac{S_{4}}%
{S_{2,3,4}}+\widetilde{b}_{3,4}\left(  -\widetilde{p}_{2}\right)  \tfrac
{1}{S_{3}}\tfrac{S_{2}}{S_{2,3,4}}+\widetilde{b}_{2,3,4}\tfrac{1}{S_{3}%
S_{2,3,4}}\label{S24/S234}
\end{align}%
\begin{align}
&  \hspace{-0.5cm}\left(  -\widetilde{p}_{2}\right)  \left(  -\widetilde{p}_{3,4}\right)
\tfrac{S_{2}}{S_{2,3,4}} =\left(  -\widetilde{p}_{2,3,4}\right)  \tfrac{1}{S_{3,4}}+\widetilde{b}%
_{2,3}\left(  -\widetilde{p}_{4}\right)  \tfrac{1}{S_{2,3,4}}\tfrac{S_{4}%
}{S_{3,4}}+\widetilde{b}_{2,4}\left(  -\widetilde{p}_{3}\right)  \tfrac
{1}{S_{2,3,4}}\tfrac{S_{3}}{S_{3,4}}+\widetilde{b}_{2,3,4}\tfrac{1}%
{S_{3,4}S_{2,3,4}},\label{S2/S234}\\
&  \left(  -\widetilde{p}_{2}\right)  \left(  -\widetilde{p}_{3,4}\right)
\tfrac{S_{3,4}}{S_{2,3,4}} =\left(  -\widetilde{p}_{2,3,4}\right)  \tfrac{1}{S_{2}}+\widetilde{b}%
_{2,3}\left(  -\widetilde{p}_{4}\right)  \tfrac{1}{S_{2}}\tfrac{S_{4}}%
{S_{2,3,4}}+\widetilde{b}_{2,4}\left(  -\widetilde{p}_{3}\right)  \tfrac
{1}{S_{2}}\tfrac{S_{3}}{S_{2,3,4}}+\widetilde{b}_{2,3,4}\tfrac{1}{S_{2}%
S_{2,3,4}},\label{S34/S234}
\end{align}
Using (\ref{S2/S23}) and (\ref{S3/S23}) into (\ref{S4/S234}), we get%
\begin{align}
&  \left(  -\widetilde{p}_{4}\right)  \left(  -\widetilde{p}_{2,3}\right)
\tfrac{S_{4}}{S_{2,3,4}} =\left(  -\widetilde{p}_{2,3,4}\right)  \tfrac{1}{S_{2,3}}+\left(
-\widetilde{p}_{2,3}\right)  \left(  -\widetilde{p}_{2}\right)  ^{-1}%
\widetilde{b}_{2,4}\tfrac{1}{S_{2}S_{2,3,4}}+\left(  -\widetilde{p}%
_{2,3}\right)  \left(  -\widetilde{p}_{3}\right)  ^{-1}\widetilde{b}%
_{3,4}\tfrac{1}{S_{3}S_{2,3,4}}\nonumber\\
&  +\widetilde{b}_{2,3,4}\tfrac{1}{S_{2,3}S_{2,3,4}}+\left(  -\widetilde{p}%
_{2}\right)  ^{-1}\widetilde{b}_{2,3}\widetilde{b}_{2,4}\tfrac{1}{S_{2}%
S_{2,3}S_{2,3,4}}+\left(  -\widetilde{p}_{3}\right)  ^{-1}\widetilde{b}%
_{2,3}\widetilde{b}_{3,4}\tfrac{1}{S_{3}S_{2,3}S_{2,3,4}}.\label{S4/S2,3,4}
\end{align}
Using (\ref{S2/S24}) and (\ref{S4/S24}) into (\ref{S3/S234}), we get%
\begin{align}
&  \left(  -\widetilde{p}_{3}\right)  \left(  -\widetilde{p}_{2,4}\right)
\tfrac{S_{3}}{S_{2,3,4}} =\left(  -\widetilde{p}_{2,3,4}\right)  \tfrac{1}{S_{2,4}}+\left(
-\widetilde{p}_{2,4}\right)  \left(  -\widetilde{p}_{2}\right)  ^{-1}%
\widetilde{b}_{2,3}\tfrac{1}{S_{2}S_{2,3,4}}+\left(  -\widetilde{p}%
_{2,4}\right)  \left(  -\widetilde{p}_{4}\right)  ^{-1}\widetilde{b}%
_{3,4}\tfrac{1}{S_{4}S_{2,3,4}}\nonumber\\
&  +\widetilde{b}_{2,3,4}\tfrac{1}{S_{2,4}S_{2,3,4}}+\left(  -\widetilde{p}%
_{2}\right)  ^{-1}\widetilde{b}_{2,3}\widetilde{b}_{2,4}\tfrac{1}{S_{2}%
S_{2,4}S_{2,3,4}}+\left(  -\widetilde{p}_{4}\right)  ^{-1}\widetilde{b}%
_{2,4}\widetilde{b}_{3,4}\tfrac{1}{S_{4}S_{2,4}S_{2,3,4}}\label{S3/S2,3,4}
\end{align}
Using (\ref{S3/S34}) and (\ref{S4/S34}) into (\ref{S2/S234}), we get%
\begin{align}
&  \left(  -\widetilde{p}_{2}\right)  \left(  -\widetilde{p}_{3,4}\right)
\tfrac{S_{2}}{S_{2,3,4}}=\left(  -\widetilde{p}_{2,3,4}\right)  \tfrac{1}{S_{3,4}}+\left(
-\widetilde{p}_{3,4}\right)  \left(  -\widetilde{p}_{3}\right)  ^{-1}%
\widetilde{b}_{2,3}\tfrac{1}{S_{3}S_{2,3,4}}+\left(  -\widetilde{p}%
_{3,4}\right)  \left(  -\widetilde{p}_{4}\right)  ^{-1}\widetilde{b}%
_{2,4}\tfrac{1}{S_{4}S_{2,3,4}}\nonumber\\
&  +\widetilde{b}_{2,3,4}\tfrac{1}{S_{3,4}S_{2,3,4}}+\left(  -\widetilde{p}%
_{3}\right)  ^{-1}\widetilde{b}_{2,3}\widetilde{b}_{3,4}\tfrac{1}{S_{3}%
S_{3,4}S_{2,3,4}}+\left(  -\widetilde{p}_{4}\right)  ^{-1}\widetilde{b}%
_{2,4}\widetilde{b}_{3,4}\tfrac{1}{S_{4}S_{3,4}S_{2,3,4}}\label{S2/S2,3,4}
\end{align}
Using (\ref{S2/S2,3,4}) and (\ref{S3/S2,3,4}) into (\ref{S23/S234}), we get
\begin{align}
&  \left(  -\widetilde{p}_{4}\right)  \left(  -\widetilde{p}_{2,3}\right)
\tfrac{S_{2,3}}{S_{2,3,4}} =\left(  -\widetilde{p}_{2,3,4}\right)  \tfrac{1}{S_{4}}+\widetilde{b}%
_{2,3,4}\tfrac{1}{S_{4}S_{2,3,4}}+\left(  -\widetilde{p}_{2,3,4}\right)
\left(  -\widetilde{p}_{2,4}\right)  ^{-1}\widetilde{b}_{2,4}\tfrac{1}%
{S_{4}S_{2,4}}  +\left(  -\widetilde{p}_{2}\right)  ^{-1}\widetilde{b}_{2,3}\widetilde{b}%
_{2,4}\tfrac{1}{S_{2}S_{4}S_{2,3,4}}+\nonumber\\
&\left(  -\widetilde{p}_{4}\right)
^{-1}\widetilde{b}_{3,4}\widetilde{b}_{2,4}\tfrac{1}{S_{4}^{2}S_{2,3,4}%
}  +\left(  -\widetilde{p}_{2,4}\right)  ^{-1}\widetilde{b}_{2,3,4}%
\widetilde{b}_{2,4}\tfrac{1}{S_{4}S_{2,4}S_{2,3,4}}+\left(  -\widetilde{p}%
_{2}\right)  ^{-1}\left(  -\widetilde{p}_{2,4}\right)  ^{-1}\widetilde{b}%
_{2,3}\widetilde{b}_{2,4}\widetilde{b}_{2,4}\tfrac{1}{S_{2}S_{4}S_{2,4}%
S_{2,3,4}}+\nonumber\\
&  \left(  -\widetilde{p}_{4}\right)  ^{-1}\left(  -\widetilde{p}%
_{2,4}\right)  ^{-1}\widetilde{b}_{2,4}\widetilde{b}_{3,4}\widetilde{b}%
_{2,4}\tfrac{1}{S_{4}^{2}S_{2,4}S_{2,3,4}}+\left(  -\widetilde{p}%
_{2,3,4}\right)  \left(  -\widetilde{p}_{3,4}\right)  ^{-1}\widetilde{b}%
_{3,4}\tfrac{1}{S_{4}S_{3,4}} +\left(  -\widetilde{p}_{3}\right)  ^{-1}\widetilde{b}_{2,3}\widetilde{b}
_{3,4}\tfrac{1}{S_{3}S_{4}S_{2,3,4}}+\nonumber\\
& \left(  -\widetilde{p}_{4}\right)
^{-1}\widetilde{b}_{2,4}\widetilde{b}_{3,4}\tfrac{1}{S_{4}^{2}S_{2,3,4}%
} +\left(  -\widetilde{p}_{3,4}\right)  ^{-1}\widetilde{b}_{2,3,4}%
\widetilde{b}_{3,4}\tfrac{1}{S_{4}S_{3,4}S_{2,3,4}}+\left(  -\widetilde{p}
_{3}\right)  ^{-1}\left(  -\widetilde{p}_{3,4}\right)  ^{-1}\widetilde{b}%
_{2,3}\widetilde{b}_{3,4}\widetilde{b}_{3,4}\tfrac{1}{S_{3}S_{4}S_{3,4}%
S_{2,3,4}}+\nonumber\\
&  \left(  -\widetilde{p}_{4}\right)  ^{-1}\left(  -\widetilde{p}%
_{3,4}\right)  ^{-1}\widetilde{b}_{2,4}\widetilde{b}_{3,4}\widetilde{b}%
_{3,4}\tfrac{1}{S_{4}^{2}S_{3,4}S_{2,3,4}} \label{S2,3/S2,3,4}%
\end{align}
Using (\ref{S2/S2,3,4}) and (\ref{S4/S2,3,4}) into (\ref{S24/S234}), we get%
\begin{align}
&  \left(  -\widetilde{p}_{3}\right)  \left(  -\widetilde{p}_{2,4}\right)
\tfrac{S_{2,4}}{S_{2,3,4}} =\left(  -\widetilde{p}_{2,3,4}\right)  \tfrac{1}{S_{3}}+\widetilde{b}_{2,3,4}\tfrac{1}{S_{3}S_{2,3,4}}+\left(  -\widetilde{p}_{2,3,4}\right)  \left(  -\widetilde{p}%
_{2,3}\right)  ^{-1}\widetilde{b}_{2,3}\tfrac{1}{S_{3}S_{2,3}}+\left(
-\widetilde{p}_{2}\right)  ^{-1}\widetilde{b}_{2,4}\widetilde{b}_{2,3}\tfrac
{1}{S_{2}S_{3}S_{2,3,4}}+\nonumber\\
&  \left(  -\widetilde{p}_{3}\right)  ^{-1}%
\widetilde{b}_{3,4}\widetilde{b}_{2,3}\tfrac{1}{S_{3}^{2}S_{2,3,4}} +\left(  -\widetilde{p}_{2,3}\right)  ^{-1}\widetilde{b}_{2,3,4}%
\widetilde{b}_{2,3}\tfrac{1}{S_{3}S_{2,3}S_{2,3,4}}+\left(  -\widetilde{p}%
_{2}\right)  ^{-1}\left(  -\widetilde{p}_{2,3}\right)  ^{-1}\widetilde{b}%
_{2,3}\widetilde{b}_{2,4}\widetilde{b}_{2,3}\tfrac{1}{S_{2}S_{3}S_{2,3}%
S_{2,3,4}}+\nonumber\\
&  \left(  -\widetilde{p}_{3}\right)  ^{-1}\left(  -\widetilde{p}%
_{2,3}\right)  ^{-1}\widetilde{b}_{2,3}\widetilde{b}_{3,4}\widetilde{b}%
_{2,3}\tfrac{1}{S_{3}^{2}S_{2,3}S_{2,3,4}} +\left(  -\widetilde{p}_{2,3,4}\right)  \left(  -\widetilde{p}%
_{3,4}\right)  ^{-1}\widetilde{b}_{3,4}\tfrac{1}{S_{3}S_{3,4}}+\left(
-\widetilde{p}_{3}\right)  ^{-1}\widetilde{b}_{2,3}\widetilde{b}_{3,4}\tfrac
{1}{S_{3}^{2}S_{2,3,4}}+\nonumber\\
& \left(  -\widetilde{p}_{4}\right)  ^{-1}%
\widetilde{b}_{2,4}\widetilde{b}_{3,4}\tfrac{1}{S_{3}S_{4}S_{2,3,4}}+\left(  -\widetilde{p}_{3,4}\right)  ^{-1}\widetilde{b}_{2,3,4}%
\widetilde{b}_{3,4}\tfrac{1}{S_{3}S_{3,4}S_{2,3,4}}+\left(  -\widetilde{p}
_{3}\right)  ^{-1}\left(  -\widetilde{p}_{3,4}\right)  ^{-1}\widetilde{b}%
_{2,3}\widetilde{b}_{3,4}\widetilde{b}_{3,4}\tfrac{1}{S_{3}^{2}S_{3,4}%
S_{2,3,4}}\nonumber\\
&  +\left(  -\widetilde{p}_{4}\right)  ^{-1}\left(  -\widetilde{p}%
_{3,4}\right)  ^{-1}\widetilde{b}_{2,4}\widetilde{b}_{3,4}\widetilde{b}%
_{3,4}\tfrac{1}{S_{3}S_{4}S_{3,4}S_{2,3,4}} \label{S2,4/S2,3,4}
\end{align}
Using (\ref{S3/S2,3,4}) and (\ref{S4/S2,3,4}) into (\ref{S34/S234}), we get \newline
\begin{align}
&  \left(  -\widetilde{p}_{2}\right)  \left(  -\widetilde{p}_{3,4}\right)
\tfrac{S_{3,4}}{S_{2,3,4}}  =\left(  -\widetilde{p}_{2,3,4}\right)  \tfrac{1}{S_{2}}+\widetilde{b}_{2,3,4}\tfrac{1}{S_{2}S_{2,3,4}} +[\left(  -\widetilde{p}_{2,3,4}\right)  \left(  -\widetilde{p}_{2,3}\right)  ^{-1}\widetilde{b}_{2,3}\tfrac{1}{S_{2}S_{2,3}}+\left(
-\widetilde{p}_{2}\right)  ^{-1}\widetilde{b}_{2,4}\widetilde{b}_{2,3}\tfrac{1}{S_{2}^{2}S_{2,3,4}}+\nonumber\\
& \left(  -\widetilde{p}_{3}\right)  ^{-1}%
\widetilde{b}_{3,4}\widetilde{b}_{2,3}\tfrac{1}{S_{2}S_{3}S_{2,3,4}}  +\left(  -\widetilde{p}_{2,3}\right)  ^{-1}\widetilde{b}_{2,3,4}%
\widetilde{b}_{2,3}\tfrac{1}{S_{2}S_{2,3}S_{2,3,4}}+\left(  -\widetilde{p}_{2}\right)  ^{-1}\left(  -\widetilde{p}_{2,3}\right)  ^{-1}\widetilde{b}%
_{2,3}\widetilde{b}_{2,4}\widetilde{b}_{2,3}\tfrac{1}{S_{2}^{2}S_{2,3}%
S_{2,3,4}}+\nonumber\\
&  \left(  -\widetilde{p}_{3}\right)  ^{-1}\left(  -\widetilde{p}%
_{2,3}\right)  ^{-1}\widetilde{b}_{2,3}\widetilde{b}_{3,4}\widetilde{b}%
_{2,3}\tfrac{1}{S_{2}S_{3}S_{2,3}S_{2,3,4}}]+[\left(  -\widetilde{p}%
_{2,3,4}\right)  \left(  -\widetilde{p}_{2,4}\right)  ^{-1}\widetilde{b}%
_{2,4}\tfrac{1}{S_{2}S_{2,4}}+\left(  -\widetilde{p}_{2}\right)  ^{-1}\widetilde{b}_{2,3}\widetilde{b}%
_{2,4}\tfrac{1}{S_{2}^{2}S_{2,3,4}}+\nonumber\\
&  \left(  -\widetilde{p}_{4}\right)
^{-1}\widetilde{b}_{2,4}\widetilde{b}_{3,4}\tfrac{1}{S_{2}S_{4}S_{2,3,4}%
}+\left(  -\widetilde{p}_{2,4}\right)  ^{-1}\widetilde{b}_{2,3,4}%
\widetilde{b}_{2,4}\tfrac{1}{S_{2}S_{2,4}S_{2,3,4}} +\left(  -\widetilde{p}_{2}\right)  ^{-1}\left(  -\widetilde{p}%
_{2,4}\right)  ^{-1}\widetilde{b}_{2,3}\widetilde{b}_{2,4}\widetilde{b}%
_{2,4}\tfrac{1}{S_{2}^{2}S_{2,4}S_{2,3,4}}  +\nonumber\\
& \left(  -\widetilde{p}_{4}\right)  ^{-1}\left(  -\widetilde{p}%
_{2,4}\right)  ^{-1}\widetilde{b}_{2,4}\widetilde{b}_{3,4}\widetilde{b}%
_{2,4}\tfrac{1}{S_{2}S_{4}S_{2,4}S_{2,3,4}}] \label{S3,4/S2,3,4}%
\end{align}
Using (\ref{S2/S2,3,4}), (\ref{S3/S2,3,4}), (\ref{S4/S2,3,4}),
(\ref{S2,3/S2,3,4}),(\ref{S2,4/S2,3,4}) and (\ref{S3,4/S2,3,4}) into
(\ref{LX2,3,4|X1}), by grouping terms of the same total degree and using the
condition $L_{\left(  X_{2},X_{3},X_{4}\right)  }^{X_{1}=x_{1}}\left(
0,0,0\right)  =1$, we obtain\newline
$
  L_{\left(  X_{2},X_{3},X_{4}\right)  }^{X_{1}=x_{1}}\left(  \theta
_{2},\theta_{3},\theta_{4}\right)= \\
  S_{2,3,4}^{-\lambda}\exp(x_{1}\{\alpha_{1}[\tfrac{1}{S_{2}}-1]+\alpha
_{2}[\tfrac{1}{S_{3}}-1]+\alpha_{3}[\tfrac{1}{S_{4}}-1]+\alpha_{4}[\tfrac
{1}{S_{2,3}}-1]+\alpha_{5}[\tfrac{1}{S_{2,4}}-1]+
\alpha_{6}[\tfrac{1}{S_{3,4}}-1]+\alpha_{7}[\tfrac{1}{S_{2}S_{2,3}}-1]+\alpha_{8}[\tfrac{1}{S_{2}S_{2,4}}-1]+\alpha_{9}[\tfrac{1}{S_{3}S_{2,3}}-1]+\alpha_{10}[\tfrac{1}{S_{3}S_{3,4}}-1]+\alpha_{11}[\tfrac{1}{S_{4}S_{2,4}}-1]+\alpha_{12}[\tfrac{1}{S_{4}S_{3,4}}-1]+\alpha_{13}[\tfrac{1}{S_{2,3,4}}-1]+\alpha_{14}[\tfrac{1}{S_{2}S_{2,3,4}}-1]+\\
\alpha_{15}[\tfrac{1}{S_{3}S_{2,3,4}}-1]+\alpha_{16}[\tfrac{1}{S_{4}S_{2,3,4}}-1]+\alpha_{17}[\tfrac{1}{S_{2}^{2}S_{2,3,4}}-1]+\alpha_{18}[\tfrac{1}{S_{3}^{2}S_{2,3,4}}-1]+\\
\alpha_{19}[\tfrac{1}{S_{4}^{2}S_{2,3,4}}-1]+\alpha_{20}[\tfrac{1}{S_{2}S_{3}S_{2,3,4}}-1]+\alpha_{21}[\tfrac{1}%
{S_{2}S_{4}S_{2,3,4}}-1] +\alpha_{22}[\tfrac{1}{S_{3}S_{4}S_{2,3,4}}-1]+\alpha_{23}[\tfrac{1}{S_{2,3}S_{2,3,4}}-1]+\alpha_{24}[\tfrac{1}{S_{2,4}S_{2,3,4}}-1]+\alpha_{25}[\tfrac{1}{S_{3,4}S_{2,3,4}}-1]+\alpha_{26}[\tfrac{1}{S_{2}S_{2,3}S_{2,3,4}}-1]+\alpha_{27}[\tfrac
{1}{S_{2}S_{2,4}S_{2,3,4}}-1]+\alpha_{28}[\tfrac{1}{S_{3}S_{2,3}S_{2,3,4}%
}-1]+\alpha_{29}[\tfrac{1}{S_{3}S_{3,4}S_{2,3,4}}-1]\\
  +\alpha_{30}[\tfrac{1}{S_{4}S_{2,4}S_{2,3,4}}-1]+\alpha_{31}[\tfrac
{1}{S_{4}S_{3,4}S_{2,3,4}}-1]+\alpha_{32}[\tfrac{1}{S_{2}^{2}S_{2,3}S_{2,3,4}}-1]+\\
\alpha_{33}[\tfrac{1}{S_{2}^{2}S_{2,4}S_{2,3,4}}-1]+\alpha_{34}[\tfrac{1}{S_{3}^{2}S_{2,3}S_{2,3,4}}-1]+\alpha_{35}[\tfrac{1}{S_{3}^{2}S_{3,4}S_{2,3,4}}-1]+\\
\alpha_{36}[\tfrac{1}{S_{4}^{2}S_{2,4}S_{2,3,4}}-1]+\alpha_{37}[\tfrac{1}{S_{4}^{2}S_{3,4}S_{2,3,4}}-1]+\alpha_{38}[\tfrac{1}{S_{2}S_{3}S_{2,3}S_{2,3,4}}-1]+\\
\alpha_{39}[\tfrac
{1}{S_{2}S_{4}S_{2,4}S_{2,3,4}}-1]+\alpha_{40}[\tfrac{1}{S_{3}S_{4}%
S_{3,4}S_{2,3,4}}-1]\})
$ \newline
Combining, we obtain (\ref{LCn=4}). Definitions (\ref{X1}), (\ref{Y2}),
(\ref{Y3}), (\ref{Y4}), (\ref{U1,2,U1,3}), (\ref{U2,2,U2,4}), (\ref{U3,3,U3,4}), (\ref{W2,W3,W4}), and (\ref{X2}), (\ref{X3}), (\ref{X4}) with Equality
(\ref{LCn=4}) give (\ref{X[4]}). This completes the proof of Theorem
(\ref{Th1_n=4_Bern2024}).
\end{proof}

\bibliographystyle{elsarticle-num}
\bibliography{SPL_DEOLT2_1}

\pagebreak

\end{document}